\documentclass{article}%
\usepackage{amsmath}%
\usepackage{amsfonts}%
\usepackage{amssymb}%
\usepackage{graphicx}
\providecommand{\U}[1]{\protect\rule{.1in}{.1in}}
\newtheorem{theorem}{Theorem}[section]

\newtheorem{corollary}[theorem]{Corollary}

\newtheorem{example}[theorem]{Example}

\newtheorem{lemma}[theorem]{Lemma}

\newtheorem{proposition}[theorem]{Proposition}
\newtheorem{remark}[theorem]{Remark}

\newenvironment{proof}[1][Proof]{\noindent\textbf{#1.} }{\ \rule{0.5em}{0.5em}}

\numberwithin{equation}{section}

\begin{document}

\begin{center}
\textbf{Cauchy problem for a fractional anisotropic parabolic equation in
anisotropic H\"{o}lder spaces.}

\bigskip

\textbf{Sergey Degtyarev}

\bigskip

\textbf{State Agency "Institute of applied mathematics and mechanics",}

\textbf{Donetsk, Ukraine }

\end{center}

\bigskip

\begin{abstract}
We consider a Cauchy problem for a fractional anisotropic parabolic equation
in anisotropic H\"{o}lder spaces. The equation generalizes the heat equation
to the case of fractional power of the Laplace operator and the power of this
operator can be different with respect to different groups of space variables.
The time derivative can be either fractional Caputo - Jrbashyan derivative or
usual derivative. Under some necessary conditions on the order of the time
derivative we show that the operator of the whole problem is an isomorphism of
appropriate anisotropic H\"{o}lder spaces. Under some another conditions we
prove unique solvability of the Cauchy problem in the same spaces.

\end{abstract}

\section{Introduction.}

\label{s1}

The present paper is devoted to correctness in anisotropic H\"{o}lder spaces
of a Cauchy problem for a fractional partial differential equation which
generalizes the classical heat equation. Let $N$ be the dimension of the space
$R^{N}$, $T>0$ is a given positive number, $R_{T}^{N}\equiv R^{N}\times(0,T)$.
Let further the set of coordinates of a point $x=(x_{1},x_{2},...,x_{N})$ is
split into $r$ groups of lengths $N_{i}$, $i=1,2,...,r$, $N_{1}+N_{2}%
+...N_{r}=N$. Let, besides, $\theta>0$ be an integer or a noninteger number
and $\sigma_{k}>0$, $k=1,2,...,r$. Denote
\begin{equation}
z_{1}=(x_{1},...,x_{N_{1}}),z_{2}=(x_{N_{1}+1},...,x_{N_{1}+N_{2}}%
),...,z_{r}=(x_{N_{1}+...+N_{r-1}+1},...,x_{N}).\label{1.0}%
\end{equation}
Consider the following Cauchy problem for the unknown function $u(x,t)$,
$(x,t)\in$ $R_{T}^{N}$,
\begin{equation}
D_{\ast t}^{\theta}u(x,t)+{\sum_{k=1}^{r}}(-\Delta_{z_{k}})^{\frac{\sigma_{k}%
}{2}}u(x,t)=f(x,t),\quad(x,t)\in R_{T}^{N},\label{1.1}%
\end{equation}%
\begin{equation}
\frac{\partial^{i}u}{\partial t^{i}}(x,0)=u_{i}(x),x\in R^{N},\quad%
%TCIMACRO{\QDATOPD{\{}{.}{i=0,1,...,[\theta],\quad\theta\ \text{is a
%noninteger},}{i=0,1,...,\theta-1,\quad\theta\ \text{is an integer},}}%
%BeginExpansion
\genfrac{\{}{.}{0pt}{0}{i=0,1,...,[\theta],\quad\theta\ \text{is a
noninteger},}{i=0,1,...,\theta-1,\quad\theta\ \text{is an integer},}%
%EndExpansion
\label{1.2}%
\end{equation}
where $[\theta]$ is the integer part of a number $\theta$. Let's explain the
above notations. Firstly, here $f(x,t)$ and $u_{0}(x)$ are some given
functions, defined in $\overline{R_{T}^{N}}$ and $R^{N}$ correspondingly.
Further, the symbol $D_{\ast t}^{\theta}$ (with the lower asterisk) means the
fractional Caputo - Jrbashyan derivative (often called just the Caputo
derivative) of a non-integer order $\theta>0$ with respect to $t$. Such
derivative of order $\theta\in(0,1)$ is defined by
\begin{equation}
D_{\ast t}^{\theta}g(t)=\frac{1}{\Gamma(1-\theta)}{\int\limits_{0}^{t}}%
\frac{g^{\prime}(\tau)d\tau}{(t-\tau)^{\theta}}=\frac{1}{\Gamma(1-\theta
)}\frac{d}{dt}{\int\limits_{0}^{t}}\frac{\left[  g(\tau)-g(0)\right]  d\tau
}{(t-\tau)^{\theta}},\quad t>0,\label{1.3}%
\end{equation}
and for the case of order $\theta\in(n-1,n)$%
\begin{equation}
D_{\ast t}^{\theta}g(t)=D_{\ast t}^{\theta-n+1}g^{(n-1)}(t).\label{1.4.0}%
\end{equation}
The second equality in \eqref{1.3} links the Caputo - Jrbashyan derivative to
the classical Riemann - Liouville fractional derivative $D_{t}^{\theta}$
(without the lower asterisk), and the last for $\theta\in(0,1)$ is defined by
\begin{equation}
D_{t}^{\theta}g(t)=\frac{1}{\Gamma(1-\theta)}\frac{d}{dt}{\int\limits_{0}^{t}%
}\frac{g(\tau)d\tau}{(t-\tau)^{\theta}},\quad t>0.\label{1.4.1}%
\end{equation}
Thus, for $\theta\in(0,1)$,
\begin{equation}
D_{\ast t}^{\theta}g(t)=D_{t}^{\theta}\left[  g(t)-g(0)\right]  .\label{1.4.2}%
\end{equation}
In general, for $\theta\in(n-1,n)$ ($n$ is a positive integer), the Riemann -
Liouville fractional derivative is defined by
\begin{equation}
D_{t}^{\theta}g(t)=\frac{1}{\Gamma(1-\theta)}\frac{d^{n}}{dt^{n}}%
{\int\limits_{0}^{t}}\frac{g(\tau)d\tau}{(t-\tau)^{\theta-n+1}},\quad
t>0,\label{1.4}%
\end{equation}
and the Caputo - Jrbashyan derivative can be expressed as
\begin{equation}
D_{\ast t}^{\theta}g(t)=\frac{1}{\Gamma(1-\theta)}{\int\limits_{0}^{t}}%
\frac{g^{(n)}(\tau)d\tau}{(t-\tau)^{\theta-n+1}}=\label{1.5}%
\end{equation}%
\[
=\frac{1}{\Gamma(1-\theta)}\frac{d}{dt}{\int\limits_{0}^{t}}\frac{\left[
g^{(n-1)}(\tau)-g^{(n-1)}(0)\right]  d\tau}{(t-\tau)^{\theta-n+1}}=
\]%
\[
=\frac{1}{\Gamma(1-\theta)}\frac{d^{n}}{dt^{n}}{\int\limits_{0}^{t}}%
\frac{\left[  g(\tau)-{\sum\limits_{k=0}^{n-1}}\frac{g^{(k)}(0)}{k!}\tau
^{k}\right]  d\tau}{(t-\tau)^{\theta-n+1}}=
\]%
\[
=D_{t}^{\theta}\left[  g(t)-{\sum\limits_{k=0}^{n-1}}\frac{t^{k}g^{(k)}%
(0)}{k!}\right]  ,\quad t>0.
\]
As for different definitions of some others fractional derivatives, we can
refer to, for example, \cite{Samko}, \cite{Kilbas}, \cite{Umarov}, Ch. 3. We
consider the above Cauchy problem with the Caputo - Jrbashyan derivative with
respect to time $t$ because it is well known that the Cauchy problem with
initial condition \eqref{1.2} in the case of the Riemann - Liouville
derivative in the equation is incorrect and it does not carry a proper
physical meaning. The correct statement for equations with the Riemann -
Liouville derivative must include some nonlocal condition instead of
\eqref{1.2} - see \cite{Samko}, \cite{Umarov}, Ch. 3. In the present paper we
consider a Cauchy problem with usual initial condition \eqref{1.2} and
therefore we use namely the Caputo - Jrbashyan derivative in equation \eqref{1.1}.

At last, the summands $(-\Delta_{z_{k}})^{\frac{\sigma_{k}}{2}}u(x,t)$ in
equation \eqref{1.1}, $k=1,...,r,$ are fractional powers of the Laplace
operators of powers $\sigma_{k}/2>0$ with respect to the corresponding group
of the space variables $z_{k}$. These fractional operators can be defined in
terms of the Fourier transform as follows. Let $z_{k}=(x_{j+1},x_{j+1}%
,...,x_{j+N_{k}})$ be such a group of the space variables and let the
corresponding group of "dual" (in the sense of the Fourier transform) group of
variables be $\zeta_{k}\equiv(\xi_{j+1},...,\xi_{j+N_{k}})$. Denote by
$F_{k}[u]$ the Fourier image of a function $u(x,t)$ with respect to the
variables $z_{k}$ that is
\[
F_{k}[u](x_{1},...,x_{j},\zeta_{k},x_{j+N_{k}+1},...,x_{N},t)\equiv
\]%
\begin{equation}
\equiv\frac{1}{(2\pi)^{\frac{N_{k}}{2}}}{\int\limits_{z_{k}\in R^{N_{k}}}%
}u(x_{1},...,x_{j},z_{k},x_{j+N_{k}+1},...,x_{N},t)e^{-i(z_{k},\zeta_{k}%
)}dz_{k}.\label{1.7}%
\end{equation}
Then by definition
\begin{equation}
(-\Delta_{z_{k}})^{\frac{\sigma_{k}}{2}}u(x,t)\equiv F_{k}^{-1}[|\zeta
_{k}|^{\sigma_{k}}F_{k}[u]],\label{1.8}%
\end{equation}
where $F_{k}^{-1}$ is the inverse transform to \eqref{1.7}. Besides, operator
$(-\Delta_{z_{k}})^{\frac{\sigma_{k}}{2}}$ from \eqref{1.8} can be defined for
$\sigma_{k}>0$ as follows (see, for example, \cite{Samko}, Ch.5,
\cite{Umarov}, Ch.3). Denote finite difference of a function $u(x,t)$ with
respect to a group of space variables $z_{k}$ with a step $\eta_{k}$ by
\[
\delta_{\eta_{k},z_{k}}u(x,t)=u(x_{1},...,x_{j},z_{k}+\eta_{k},x_{j+N_{k}%
+1},...,x_{N},t)-u(x,t)
\]
and denote the corresponding finite difference of an order $m>1$ by
\[
\delta_{\eta_{k},z_{k}}^{m}u(x,t)=\delta_{\eta_{k},z_{k}}\left(  \delta
_{\eta_{k},z_{k}}^{m-1}u(x,t)\right)  .
\]
Then
\begin{equation}
(-\Delta_{z_{k}})^{\frac{\sigma_{k}}{2}}u(x,t)=C_{N_{k},\sigma_{k},m}%
{\int\limits_{\eta_{k}\in R^{N_{k}}}}\frac{\delta_{\eta_{k},z_{k}}^{m}%
u(x,t)}{|\eta_{k}|^{N_{k}+\sigma_{k}}}d\eta_{k},\label{1.9}%
\end{equation}
where $m>\sigma_{k}$ is arbitrary, $C_{N_{k},\sigma_{k},m}$ is some constant,
which depends only on $m$, on the dimension $N_{k}$ of the group of the
variables $z_{k}$, and on $\sigma_{k}$.

The question of a possible domain of definition for the operators $D_{\ast
t}^{\theta}$ and $(-\Delta_{z_{k}})^{\frac{\sigma_{k}}{2}}$ will be discussed
a little bit below.

Problem \eqref{1.1}, \eqref{1.2} belongs to a class of mathematical models
with fractional differentiation. Similar models arises in very many
contemporary investigations in different branches of science and technology.
In particular, such models describe different processes in fractal media and
differential operators like in \eqref{1.1} serve as generators for definite
random processes, which are actively investigated at present. It is impossible
even outline here all different applications of the model under consideration
and for details we refer the reader to the monographs \cite{Samko},
\cite{Kilbas}, \cite{Umarov}, to the survey papers \cite{Gorenflo},
\cite{Metzler}, and to a plenty of other recent monographs, devoted to
investigations and applications of fractional models.

Note that we can observe nowadays almost explosive growth of investigations on
properties of mathematical models with fractional differential operators in
view of their great importance and numerous applications. At the same time
problems for equations of the kind \eqref{1.1}, that contains and a fractional
time derivative, and a fractional Laplace operator are investigated quite a
little. As for the investigations of correctness for such problems in classes
of smooth functions (especially up to the initial time moment), the author is
aware of just a few papers, devoted to such questions. Therefore we describe
shortly only the main known results related to the issues studied in this
paper. And we stress that we do not comment, for example, all the papers with
different results on fundamental solutions for more-less related equations, -
at least just because we do not use the methods of explicit fundamental solutions.

Nevertheless, we start with the paper \cite{0}, which contains the fundamental
solution for the "doubly fractional" equation (in our notations)
\[
D_{t}^{\theta}u(x,t)+(-\Delta_{x})^{\frac{\sigma}{2}}u(x,t)=0,\quad(x,t)\in
R^{N}\times(0,\infty)
\]
under the restrictions $0<\theta<2$, $0<\sigma\leq2$ on the orders of
fractional differentiation. It was shown that in dependence on relations
between $\theta$, $\sigma$, and the space dimension $N$, the fundamental
solution can be either positive or changing it's sign.

In the paper \cite{01} the maximum principle is proved for an initial problem
for a similar equation with nonzero right hand side.

A Cauchy problem for the equation with the usual first derivative in time
\[
\frac{\partial u}{\partial t}(x,t)+(-\Delta_{x})^{\frac{\sigma}{2}%
}u(x,t)=f(x,t),\quad(x,t)\in R^{N}\times(0,\infty)
\]
is considered in \cite{2}. The paper deals with two types of "heat"
potentials, which formally give in their sum a solution to the problem
according to the Duhamel formula. These potentials have as their kernel the
fundamental solution for the equation. One of them, a volume potential, is
associated with the right hand side $f(x,t)$ of the equation, and the another,
an initial potential, - with the initial datum $u_{0}(x)$. The paper gives for
these potentials estimates in the spaces $L_{p}([0,T],F_{c}^{a,b}(R^{N}))$ and
$L_{p}([0,T],B_{c}^{a,b}(R^{N}))$, where $F_{c}^{a,b}(R^{N})$ is the Triebel -
Lizorkin space, and $B_{c}^{a,b}(R^{N})$ is the Besov - Lipschitz space with
respect to the space variables. It is known that the Besov - Lipschitz spaces
in their particular case $B_{c}^{\infty,\infty}(R^{N})$ include the H\"{o}lder
spaces, and the paper contains estimates for the volume potential (with the
density $f(x,t)$) for this particular case and for $p=\infty$. Thus the paper
gives estimates of the volume potential in the space of functions with the
bounded in time norm in the space $C^{\sigma+\alpha}(R^{N})$ with respect to
$x$ under the condition of boundedness in time of the density $f(x,t)$ in the
space $C^{\alpha}(R^{N})$, $\alpha\in(0,1)$. However, the range of the spaces,
considered in the paper for the initial potential with the density $u_{0}(x)$,
does not contain estimates in the space $L_{\infty}([0,T],B_{\alpha}%
^{\infty,\infty}(R^{N}))=L_{\infty}([0,T],C^{\alpha}(R^{N}))$. That is the
classical smoothness of the initial potential (both with respect time and
space) is not considered.

An analogous Cauchy problem for a similar equation with $f(x,t)\equiv0$ was
investigated in  \cite{11}. In this paper the initial data can have a growth
at infinity with the restriction
\[
|u_{0}(x)|\leq C(1+|x|)^{\sigma-\varepsilon},\quad\varepsilon>0.
\]
For such initial data an existence and uniqueness of the classical solution
is, in particular, proved - in the sense that the solution itself and it's
derivatives, included in the equation, are continuous, and the solution
approaches it's initial datum in the classical sense.

The paper \cite{1} is devoted, in particular, to smoothness issues of the
solution to the Cauchy problem for the equation
\[
\frac{\partial u}{\partial t}(x,t)+\emph{L}u(x,t)=\emph{L}f(x,t),\quad(x,t)\in
R^{N}\times(0,\infty)
\]
with the usual derivative in time and with rather general nonlocal operator
$\emph{L}$ with respect to the space variables. This operator generalizes
space operator from equation \eqref{1.1} in the sense of \eqref{1.9} (but for
the same order of fractional differentiation in each space direction
$\sigma_{1}=...=\sigma_{r}=\sigma$). Note that the right hand side $\emph{L}
f(x,t)$ in the equation is a distribution since smoothness of the function
$f(x,t)$ from an anisotropic H\"{o}lder space is less than the order of
operator $\emph{L}$. Therefore the paper deals with a weak solution and the
initial data are supposed from some Lebesgue space. It is shown that this
solution $u(x,t)$ for $t>0$ inherits the smoothness properties of the function
$f(x,t)$, and thus $u(x,t)$ belongs to the same anisotropic H\"{o}lder space
as $f(x,t)$ with some estimate of H\"{o}lder seminorm of $u(x,t)$ over the
same seminorm of $f(x,t)$.

The papers \cite{3}, \cite{4}, \cite{7} are also devoted to studying of
equations of the form
\begin{equation}
\frac{\partial u}{\partial t}(x,t)+\emph{L}u(x,t)=f(x,t),\quad(x,t)\in
R^{N}\times(0,\infty) \label{1.10}%
\end{equation}
with some nonlocal operator $\emph{L}$, which generalizes the fractional
Laplace operator $(-\Delta_{x})^{\frac{\sigma}{2}}$ in the sense of definition \eqref{1.9}.

Moreover, the paper \cite{3} deals with a completely nonlinear operator
$\emph{L}$ and the equation has the form
\[
\frac{\partial u}{\partial t}(x,t)=\inf_{a\in\mathit{A}}(\emph{L}%
_{a}u(x,t)+f_{a}(x,t)),
\]
where a parameter $a$ runs through some index set $\mathit{A}$. For a solution
to this equation sharp local (inner) Schauder and some other estimates of
smoothness are obtained.

The paper \cite{4} is also devoted to studying of equation \eqref{1.10}. For
the Cauchy problem with zero initial condition existence and uniqueness of the
H\"{o}lder smooth solution is obtained. Moreover, under the assumption that
the right hand side $f(x,t)$ has a finite H\"{o}lder seminorm with respect to
the space variables $x$ there were proved sharp partial Schauder estimates
with respect to the same space variables $x$. Under the same assumption there
were also proved some interesting estimates of the smoothness with respect to
time $t$ for the solution.

The paper \cite{7} in it's turn contains sharp inner Schauder estimates for
equation \eqref{1.10} in some natural H\"{o}lder space and some interesting
boundary estimates.

The papers \cite{5}, \cite{6} deals with Cauchy problem \eqref{1.1},
\eqref{1.2} with a fractional derivative in time and with the usual Laplacian
as the space operator. Here some estimates of the solution in different
Sobolev spaces are obtained.

Further, the papers \cite{9} - \cite{12} are devoted to studying of abstract
parabolic equation with a fractional time derivative in Banach spaces.

Some other questions of regularity and qualitative behavior of solutions to
fractional equations of the kind were considered, in particular, in \cite{13}
- \cite{181}.

\begin{remark}
\label{R0} Let us stress that all the above regularity results and estimates
do not contain complete sharp coercive estimates in smooth functional classes
up to the initial time moment $t=0$. And question of qualified smoothness up
to the initial time moment with a corresponding sharp estimate in the case of
smooth initial data is still open. Therefore the goal of the present paper is
to find conditions for smoothness of possible solutions to \eqref{1.1},
\eqref{1.2} up to $t=0$ likewise it takes place for usual parabolic equations.
\end{remark}

Naturally, since the present paper is not a survey, it can not give even a
brief description of all huge amount of the existing results on regularity for
fractional parabolic equations. Therefore we confine ourselves to some known
to us existing investigations that are mostly adjacent (in our opinion !) to
the issues that are the main purpose of the paper.

Nevertheless, it is very important to refer here the mostly recent papers
\cite{182} and \cite{183}, which were published after the present paper was
prepared. These papers are closely related to the questions under our
consideration. They consider a fractional parabolic equation with a general
nonlocal space operator, which generalizes the fractional Laplacian. The
investigations in \cite{182} and \cite{183} are based on the methods of
operator semigroups and include properties of the corresponding semigroups. In
particular, along with other questions, the questions of the Schauder
estimates for the related equations are considered.

Besides, we would like to refer shortly one more paper \cite{184}, which also
was published after the present paper was prepared. Here the authors derive
long time $L_{p} - L_{q}$ decay estimates, in the full range $1 \leq p, q
\leq\infty$, for the time-dependent Fourier multipliers
\[
\widetilde{m}(t,\xi)=e^{\pm i|\xi|^{\sigma}t-|\xi|^{\theta}t},
\]
which correspond to the Cauchy problem for the homogeneous equation
\[
u_{tt}+(-\Delta)^{\sigma}u+(-\Delta)^{\frac{\theta}{2}}u_{t} = 0.
\]

The subsequent content of the paper is as follows. In the next section we
define standard anisotropic H\"{o}lder spaces and those of them that
corresponds to the anisotropy of equation \eqref{1.1}. These particular spaces
will be the working spaces for our considerations of problem \eqref{1.1},
\eqref{1.2}. The section is concluded by the formulation of the main results
of the paper in terms of the mentioned spaces.

In section \ref{s3} we formulate some (mainly known) results on operators of
fractional differentiation in isotropic H\"{o}lder spaces.

Sections \ref{sS} and \ref{sfi} are also auxiliary and they are devoted to the
actions of fractional differentiation on Schwartz and Lizorkin spaces and also
on their dual spaces of distributions.

In sections \ref{Hold} and \ref{sDOP} we study operators of fractional
differentiation in anisotropic H\"{o}lder spaces. The results of this section
show that the operator of problem \eqref{1.1}, \eqref{1.2} is a bounded linear
operator in the corresponding anisotropic spaces.

Below, to prove the existence of the bounded inverse operator to the operator
of problem \eqref{1.1}, \eqref{1.2}, we use some results on Fourier
multipliers in anisotropic H\"{o}lder spaces. Therefore we formulate these
results in section \ref{s4}.

The fact of existence (under proper conditions) of the mentioned inverse
operator to problem \eqref{1.1}, \eqref{1.2} is proved in the subsequent
sections \ref{s5} - \ref{s10}.

At that we fist consider separately in sections \ref{s5} and \ref{s6} the
cases of minimal fractional and integer orders of differentiation with respect
to time $\theta\in(0,1)$ and $\theta= 1$. Here the exponents of smoothness in
time are supposed to be $\theta+ \theta\alpha$ and $1 + \alpha$
correspondingly, where $\theta\alpha\in(0,1)$ and $\alpha\in(0,1)$.

Then, in sections \ref{s7} and \ref{s8} we show that the smoothness of the
solution to \eqref{1.1}, \eqref{1.2} rises in accordance with the rising of
the data of the problem. That is we in fact consider the case of arbitrary
high smoothness of the data.

To move to the problems with an arbitrary large order of time differentiation
$\theta> 1$, we consider next the question of constructing of functions from
anisotropic H\"{o}lder spaces with given initial functional values at $t=0$ up
to maximal possible order. The corresponding construction is described in
section \ref{s9}. This is necessary to reduce an initial problem to a problem
with zero initial data. We can not use for that the known results on this
subject since such results are absent for the case of an irrational anisotropy
of H\"{o}lder spaces.

At last, section \ref{s10} concludes the proofs of the main theorems
\ref{T2.1}, \ref{T2.2} and \ref{T2.3} below.

\section{Functional spaces and formulation of the main results.}

\label{s2}

In this paper we use some natural for equation \eqref{1.1} anisotropic
H\"{o}lder spaces of functions with different smoothness with respect to
different variables. Let $\overline{l}=(l_{1},l_{2},...,l_{N})$, where $l_{i}$
are arbitrary positive non-integer numbers. Denote by $C^{\overline{l}}%
(R^{N})$ the Banach space of functions $u(x)$, $x\in R^{N}$, with the finite
norm
\begin{equation}
\left\Vert u\right\Vert _{C^{\overline{l}}(R^{N})}\equiv|u|_{R^{N}%
}^{(\overline{l})}=|u|_{R^{N}}^{(0)}+{\sum\limits_{i=1}^{N}}\left\langle
u\right\rangle _{x_{i},R^{N}}^{(l_{i})},\label{2.1}%
\end{equation}%
\begin{equation}
|u|_{R^{N}}^{(0)}=\sup_{x\in R^{N}}|u(x)|,\label{2.2}%
\end{equation}%
\begin{equation}
\left\langle u\right\rangle _{x_{i},R^{N}}^{(l_{i})}=\sup_{x\in R^{N}%
,h>0}\frac{|D_{x_{i}}^{[l_{i}]}u(x_{1},x_{2},...,x_{i}+h,...,x_{N})-D_{x_{i}%
}^{[l_{i}]}u(x)|}{h^{l_{i}-[l_{i}]}}.\label{2.3}%
\end{equation}
Here $[l_{i}]$ the integer part of a number $l_{i}$, $D_{x_{i}}^{[l_{i}]}u(x)$
is the derivative of a function $u(x)$ of order $[l_{i}]$ with respect to a
variable $x_{i}$. Seminorm \eqref{2.3} can be equivalently defined as
(\cite{19} - \cite{21})
\begin{equation}
\left\langle u\right\rangle _{x_{i},R^{N}}^{(l_{i})}\simeq\sup_{x\in
R^{N},h>0}\frac{|\delta_{h,x_{i}}^{k}u(x)|}{h^{l_{i}}},\label{2.4}%
\end{equation}
where $\delta_{h,x_{i}}=u(x_{1},x_{2},...,x_{i}+h,...,x_{N})-u(x)$ represents
the difference of a function $u(x)$ with respect to a variable $x_{i}$ with a
step $h$, $\delta_{h,x_{i}}^{k}u(x)=\delta_{h,x_{i}}\left(  \delta_{h,x_{i}%
}^{k-1}u(x)\right)  =\left(  \delta_{h,x_{i}}\right)  ^{k}u(x)$ is the
difference of an arbitrary fixed order $k>l_{i}$. It is known (see, for
example, \cite{20}), that functions from the space $C^{\overline{l}}(R^{N})$
admit also some mixed derivatives up to a definite order depending on the set
of the exponents $l_{i}$. At that all the mixed and "pure" derivatives
$D_{x_{i}}^{[l_{i}]}$ have finite H\"{o}lder seminorms with some exponents
with respect to all the variables. Namely, let $\overline{k}=(k_{1}%
,k_{2},...,k_{N})$, $k_{i}\leq\lbrack l_{i}]$ and
\begin{equation}
\omega=1-{\sum\limits_{i=1}^{N}}\frac{k_{i}}{l_{i}}>0,\label{2.5}%
\end{equation}
and $\overline{d}=(d_{1},d_{2},...,d_{N})$, where $d_{i}=\omega l_{i}$. Then
\begin{equation}
D_{x}^{\overline{k}}u(x)\in C^{\overline{d}}(R^{N}),\quad\left\Vert
D_{x}^{\overline{k}}u(x)\right\Vert _{C^{\overline{d}}(R^{N})}\leq
C(N,\overline{l},\overline{k})\left\Vert u(x)\right\Vert _{C^{\overline{l}%
}(R^{N})}.\label{2.6}%
\end{equation}

In the present paper we are going to consider solutions to problem
\eqref{1.1}, \eqref{1.2}, that is functions $u(x,t)$ defined in the domain
$R_{T}^{N}\equiv R^{N}\times(0,T)$ (including the case $R_{\infty}^{N}%
=R^{N}\times(0,\infty)$). For such domains in the space $R^{N+1}$ all
definitions and properties in \eqref{2.1} - \eqref{2.6} (for the space $R^{N}%
$) are also valid with respect to all variables $(x,t)\in R^{N+1}$. In view of
the character of equation \eqref{1.1} and in view of our splitting of the
whole set of the space variables into $r$ groups $z_{k}$ of length $N_{k}$
(see definition \eqref{1.0}), we introduce now some additional notations. Let
$\alpha\in(0,1)$ be chosen in a way that numbers $\theta\alpha$ and
$\theta+\theta\alpha$ are positive non-integers, where $\theta$ is the order
of the derivative in $t$ from equation \eqref{1.1}. We suppose that functions
$u(x,t)$ under consideration have smoothness in $t$ of order $\theta
+\theta\alpha$ in the sense of definition \eqref{2.3} that is the value of
$\left\langle u\right\rangle _{t,\overline{R_{T}^{N}}}^{(\theta+\theta\alpha
)}$ is finite. We suppose also that for each group $z_{k}$ from \eqref{1.0}
the smoothness order of $u(x,t)$ with respect to each space variable from
$z_{k}$ is $\sigma_{k}(1+\alpha)$, where $\sigma_{k}\alpha$ and $\sigma
_{k}+\sigma_{k}\alpha$ are non-integers. That is for each space variable
$x_{i}$ inside $z_{k}$ the seminorm $\left\langle u\right\rangle
_{x_{i},\overline{R_{T}^{N}}}^{\sigma_{k}(1+\alpha)}$ is finite. Denote the
total H\"{o}lder seminorm with respect to the group $z_{k}$ by $\left\langle
u\right\rangle _{z_{k},\overline{R_{T}^{N}}}^{\sigma_{k}(1+\alpha)}$, that is
\begin{equation}
\left\langle u\right\rangle _{z_{k},\overline{R_{T}^{N}}}^{\sigma_{k}%
(1+\alpha)}\equiv{\sum\limits_{x_{i}\in z_{k}}}\left\langle u\right\rangle
_{x_{i},\overline{R_{T}^{N}}}^{\left(  \sigma_{k}(1+\alpha)\right)
}.\label{2.7}%
\end{equation}
Denote, besides, the set of the orders of fractional differentiation with
respect to different groups of space variables in equation \eqref{1.1} by
$\overline{\sigma}=(\sigma_{1},\sigma_{2},...,\sigma_{r})$ and the set of the
smoothness exponents with respect to different groups by $\overline{\sigma
}(1+\alpha)=(\sigma_{1}(1+\alpha),\sigma_{2}(1+\alpha),...,\sigma_{r}%
(1+\alpha))$. Denote, at last, the H\"{o}lder space of functions $u(x,t)$ with
described anisotropic smoothness by $C^{\overline{\sigma}(1+\alpha
),\theta+\theta\alpha}(\overline{R_{T}^{N}})$. That is $C^{\overline{\sigma
}(1+\alpha),\theta+\theta\alpha}(\overline{R_{T}^{N}})$ is the H\"{o}lder
space of bounded and continuous in the closed domain $\overline{R_{T}^{N}}$
functions with the finite over $\overline{R_{T}^{N}}$ norm
\[
\left\Vert u\right\Vert _{C^{\overline{\sigma}(1+\alpha),\theta+\theta\alpha
}(\overline{R_{T}^{N}})}\equiv
\]%
\begin{equation}
\equiv|u|_{\overline{R_{T}^{N}}}^{(\overline{\sigma}(1+\alpha),\theta
+\theta\alpha)}=|u|_{\overline{R_{T}^{N}}}^{(0)}+{\sum\limits_{k=1}^{r}%
}\left\langle u\right\rangle _{z_{k},\overline{R_{T}^{N}}}^{(\sigma
_{k}(1+\alpha))}+\left\langle u\right\rangle _{t,\overline{R_{T}^{N}}%
}^{(\theta+\theta\alpha)},\label{2.8}%
\end{equation}
where $\left\langle u\right\rangle _{z_{k},\overline{R_{T}^{N}}}^{(\sigma
_{k}(1+\alpha))}$ are defined in \eqref{2.7} and
\begin{equation}
|u|_{\overline{R_{T}^{N}}}^{(0)}=\sup_{(x,t)\in\overline{R_{T}^{N}}%
}|u(x,t)|.\label{2.9}%
\end{equation}
Besides, we denote by $C^{\overline{\sigma}(1+\alpha)}(R^{N})$ the space of
functions $u(x)$ with dependance only on the space variables, with $R^{N}$ as
the domain of their definition, and with defined above smoothness in the space
variables that is
\begin{equation}
\left\Vert u\right\Vert _{C^{\overline{\sigma}(1+\alpha)}(R^{N})}%
\equiv|u|_{R^{N}}^{(\overline{\sigma}(1+\alpha))}=|u|_{R^{N}}^{(0)}%
+{\sum\limits_{k=1}^{r}}\left\langle u\right\rangle _{z_{k},R^{N}}%
^{(\sigma_{k}(1+\alpha))}.\label{2.10}%
\end{equation}
At the same time, for the right hand side $f(x,t)$ of equation \eqref{1.1} we
use a H\"{o}lder space with a lower smoothness - according to the orders of
differentiation in \eqref{1.1} in $t$ and $x$. Namely, we use the space
$C^{\overline{\sigma}\alpha,\theta\alpha}(\overline{R_{T}^{N}})$ with the
norm
\begin{equation}
\left\Vert f\right\Vert _{C^{\overline{\sigma}\alpha,\theta\alpha}%
(\overline{R_{T}^{N}})}\equiv|f|_{\overline{R_{T}^{N}}}^{(\overline{\sigma
}\alpha,\theta\alpha)}=|f|_{\overline{R_{T}^{N}}}^{(0)}+{\sum\limits_{k=1}%
^{r}}\left\langle f\right\rangle _{z_{k},\overline{R_{T}^{N}}}^{(\sigma
_{k}\alpha)}+\left\langle f\right\rangle _{t,\overline{R_{T}^{N}}}%
^{(\theta\alpha)}.\label{2.12}%
\end{equation}

Note that all the above definitions of functional spaces are preserved and in
the case of bounded domains in $R^{N}$ and $R_{T}^{N}$.

In what follows we will use also the closed subspace of the space
$C^{\overline{\sigma}\alpha,\theta\alpha}(\overline{R_{T}^{N}})$ with elements
$f(x,t)$, that are identically equal to zero at $t=0$, $f(x,0)\equiv0$. We
denote this subspace by $\underline{C}^{\overline{\sigma} \alpha,\theta\alpha
}(\overline{R_{T}^{N}})$. And analogously we denote by $\underline
{C}^{\overline{\sigma} (1+\alpha),\theta+\theta\alpha}(\overline{R_{T}^{N}})$
and $\underline{C}^{\overline{\sigma}(1+\alpha),\theta;\overline{\sigma}%
\alpha}(\overline{R_{T}^{N}})$ (underlined) the closed subspaces of the
corresponding spaces consisting of functions that equal to zero at $t=0$
together with all their derivatives in $t$ up to the order $[\theta]$.

Turning now to the original problem \eqref{1.1}, \eqref{1.2}, we consider it
as a linear operator $\emph{L}$ in the introduced functional spaces. That is
for an integer $\theta=n$
\[
\emph{L:\,}C^{\overline{\sigma}(1+\alpha),\theta+\theta\alpha}(\overline
{R_{T}^{N}})\rightarrow
\]
\begin{equation}
\rightarrow C^{\overline{\sigma}\alpha,\theta\alpha}(\overline{R_{T}^{N}%
})\times C^{\overline{\sigma}(1+\alpha)}(R^{N})\times C^{\overline{\sigma
}(1+\alpha)-\frac{1}{n}\overline{\sigma}}\times...\times C^{\overline{\sigma
}(1+\alpha)-\frac{n-1}{n}\overline{\sigma}}, \label{2.14}%
\end{equation}
and for a non-integer $\theta$ (under some conditions)
\[
\emph{L:\,}C^{\overline{\sigma}(1+\alpha),\theta+\theta\alpha}(\overline
{R_{T}^{N}})\rightarrow
\]
\begin{equation}
\rightarrow C^{\overline{\sigma}\alpha,\theta\alpha}(\overline{R_{T}^{N}%
})\times C^{\overline{\sigma}(1+\alpha)}(R^{N})\times C^{\overline{\sigma
}(1+\alpha)-\frac{1}{\theta}\overline{\sigma}}(R^{N})\times...\times
C^{\overline{\sigma}(1+\alpha)-\frac{[\theta]}{\theta}\overline{\sigma}}.
\label{2.15}%
\end{equation}
This operator maps a function $u(x,t)$ from the space $C^{\overline{\sigma
}(1+\alpha),\theta+\theta\alpha}(\overline{R_{T}^{N}})$ to it's image under
applying the whole differential operator from \eqref{1.1} (that is the
corresponding function $f(x,t)$ from the space $C^{\overline{\sigma}%
\alpha,\theta\alpha}(\overline{R_{T}^{N}})$) and it's initial traces
$u_{0}(x)=u(x,0)$, $u_{1}(x)=u_{t}(x,0)$, ... , $u_{n-1}(x)=u_{t}%
^{(n-1)}(x,0)$ ( ...,$u_{[\theta]}(x)=u_{t}^{([\theta])}(x,0)$ ) up to the
order $n-1$ ( in the case of an integer $\theta=n$) or up to the order
$[\theta]$ in the case of a non-integer $\theta$. The correctness of this
definitions is stated by the following main theorems of the present papers.

\begin{theorem}
\label{T2.1} Let $\theta=n>0$ be an integer and not a number of the form
$n\neq4k+2$, $k=0,1,...$. Let further $\alpha$ be a positive non-integer such
that $n\alpha$ is a non-integer. Let, at last, $\sigma_{k}$, $k=1,...,r$, are
positive numbers such that $\sigma_{k}\alpha$ and $\sigma_{k}+\sigma_{k}%
\alpha$, $k=1,...,r$, are non-integers. Then the operator $\emph{L}$ is a
linear isomorphism from the space $C^{\overline{\sigma}(1+\alpha),n+n\alpha
}(\overline{R_{T}^{N}})$ to the space $C^{\overline{\sigma}\alpha,n\alpha
}(\overline{R_{T}^{N}})\times C^{\overline{\sigma}(1+\alpha)}(R^{N})\times
C^{\overline{\sigma}(1+\alpha)-\frac{1}{n}\overline{\sigma}}\times...\times
C^{\overline{\sigma}(1+\alpha)-\frac{n-1}{n}\overline{\sigma}}$ as it is
indicated in \eqref{2.14}.

If the function $f(x,t)$ in \eqref{1.1} is defined for all $t>0$ in the domain
$\overline{R_{\infty}^{N}}=R^{N}\times\lbrack0,\infty)$ and it's norm in the
space $C^{\overline{\sigma}\alpha,n\alpha}(\overline{R_{\infty}^{N}})$ is
finite (that is if $\left\vert f(x,t)\right\vert _{\overline{R_{\infty}^{N}}%
}^{(\overline{\sigma}\alpha,n\alpha)}<\infty$), then Cauchy problem
\eqref{1.1}, \eqref{1.2} has the unique solution $u(x,t)$, which belongs to
the space $C^{\overline{\sigma}(1+\alpha),n+n\alpha}(\overline{R_{\infty}^{N}%
})$ locally in time $t$ and the following estimates are valid
\begin{equation}
\left\langle u\right\rangle _{\overline{R_{\infty}^{N}}}^{(\overline{\sigma
}(1+\alpha),n+n\alpha)}\leq C(\overline{\sigma},\alpha)\left(  |f|_{\overline
{R_{\infty}^{N}}}^{(\overline{\sigma}\alpha,n\alpha)}+{\sum\limits_{i=0}%
^{n-1}}|u_{i}|_{R^{N}}^{(\overline{\sigma}(1+\alpha)-\frac{i}{n}%
\overline{\sigma})}\right)  ,\label{2.16}%
\end{equation}%
\begin{equation}
\left\vert u\right\vert _{\overline{R_{\widetilde{T}}^{N}}}^{(0)}\leq
C(\overline{\sigma},\alpha)\left(  |f|_{\overline{R_{\infty}^{N}}}%
^{(\overline{\sigma}\alpha,n\alpha)}+{\sum\limits_{i=0}^{n-1}}|u_{i}|_{R^{N}%
}^{(\overline{\sigma}(1+\alpha)-\frac{i}{n}\overline{\sigma})}\right)
(1+\widetilde{T}^{n+\alpha})+|u_{0}|_{R^{N}}^{(0)},\quad\widetilde{T}%
\leq\infty.\label{2.17}%
\end{equation}

\end{theorem}

\begin{remark}
\label{R00} For a non-integer $\theta$ operator $\emph{L}$ not always is an
isomorphism of the indicated above spaces but only under some conditions (see
Theorem \ref{T2.3} below). For example, the function with no dependance on
$x$, namely $u(x,t)=t$, belongs in fact to $C^{\infty}(\overline{R_{T}^{N}})$
and the more $u(x,t)\in C^{\overline{\sigma}(1+\alpha),\theta+\theta\alpha
}(\overline{R_{T}^{N}})$ with arbitrary large exponents. Let $\theta\in(0,1)$
and $\alpha$ is sufficiently big so that $\theta\alpha>1$. The value of the
differential expression from \eqref{1.1} for this function is equal to
$f(x,t)=D_{\ast t}^{\theta} u=C(\theta)t^{1-\theta}$, so that the maximal
smoothness of $f(x,t)$ in $t$ up to $t=0$ is equal to $1-\theta$, which is
less than needed smoothness $\theta\alpha$, $1-\theta<\theta\alpha$. Likewise,
the function $u(x,t)=t^{\theta}$ for $\theta\in(0,1)$ is a solution to
\eqref{1.1}, \eqref{1.2} with the corresponding constant function $f \equiv
const$ and with zero initial condition. And this $u(x,t)$ has up to $t=0$ the
order of smoothness in $t$ exactly $\theta$ and not $\theta+ \theta\alpha$.
Sharp consideration of such cases are out of the scope of the present paper.
And our goal is to determine conditions for problem \eqref{1.1}, \eqref{1.2}
to behave likewise Cauchy problems for usual parabolic equations (comp.
\cite{Sol1}) - see Remark \ref{R0}.

However, the following theorem on solvability of problem \eqref{1.1},
\eqref{1.2} is valid.
\end{remark}

\begin{theorem}
\label{T2.2} Let $\theta>0$ be a non-integer, and let $\alpha>0$ be such that
$\theta\alpha$ and $\theta+\theta\alpha$ are non-integers. Let further
$\sigma_{k}$, $k=1,...,r$, be positive non-integers such that $\sigma
_{k}\alpha$ and $\sigma_{k}+\sigma_{k}\alpha$, $k=1,...,r$, are
non-integers. Let also be fulfilled the following compatibility condition
\begin{equation}
f(x,0)={\sum_{k=1}^{r}}(-\Delta_{z_{k}})^{\frac{\sigma_{k}}{2}}u_{0}(x),\quad
x\in R^{N}\label{1dop.1}%
\end{equation}
and for $\theta\alpha>1$ the condition
\[
\frac{d^{i}f(x,0)}{dt^{i}}\equiv0,\quad i=1,...,[\theta\alpha].
\]

If $\theta\in(0,2)$ then problem \eqref{1.1}, \eqref{1.2} has the unique
solution from the space $C^{\overline{\sigma}(1+\alpha),\theta+\theta\alpha
}(\overline{R_{T}^{N}})$ with the estimate
\begin{equation}
|u|_{\overline{R_{T}^{N}}}^{(\overline{\sigma}(1+\alpha),\theta+\theta\alpha
)}\leq C(\overline{\sigma},\theta,\alpha,T)\left(  |f|_{\overline{R_{T}^{N}}%
}^{(\overline{\sigma}\alpha,\theta\alpha)}+{\sum\limits_{i=0}^{[\theta]}%
}|u_{i}|_{R^{N}}^{(\overline{\sigma}(1+\alpha)-\frac{i}{\theta}\overline
{\sigma})}\right)  .\label{2.19}%
\end{equation}
In the case $T=\infty$
\begin{equation}
\left\langle u\right\rangle _{\overline{R_{\infty}^{N}}}^{(\overline{\sigma
}(1+\alpha),\theta+\theta\alpha)}\leq C(\overline{\sigma},\theta
,\alpha)\left(  |f|_{\overline{R_{\infty}^{N}}}^{(\overline{\sigma}%
\alpha,\theta\alpha)}+{\sum\limits_{i=0}^{[\theta]}}|u_{i}|_{R^{N}%
}^{(\overline{\sigma}(1+\alpha)-\frac{i}{\theta}\overline{\sigma})}\right)
,\label{2.20}%
\end{equation}%
\begin{equation}
\left\vert u\right\vert _{\overline{R_{\widetilde{T}}^{N}}}^{(0)}\leq
C(\overline{\sigma},\theta,\alpha)\left(  |f|_{\overline{R_{\infty}^{N}}%
}^{(\overline{\sigma}\alpha,\theta\alpha)}+{\sum\limits_{i=0}^{[\theta]}%
}|u_{i}|_{R^{N}}^{(\overline{\sigma}(1+\alpha)-\frac{i}{\theta}\overline
{\sigma})}\right)  \times\label{2.21}%
\end{equation}%
\[
\times\left(  1+\widetilde{T}^{\theta+\theta\alpha}\right)  +|u_{0}|_{R^{N}%
}^{(0)},\quad\widetilde{T}\leq T.
\]

If now $\theta>2$ then the previous statement including the corresponding
above estimates is still valid under the additional assumption $\{\theta
\}+\theta\alpha>1$.
\end{theorem}

\begin{remark}
\label{1dop.000} We stress that condition \eqref{1dop.1} is necessary for the
smoothness up to $t=0$ since the fractional derivative $D^{\theta}_{ \ast
t}u(x,t)$ of a function $u(x,t)$ from the space $C^{\overline{\sigma}%
(1+\alpha),\theta+\theta\alpha}(\overline{R_{T}^{N}})$ is always equal
identically to zero at $t=0$ - see Proposition \ref{P3.1.00} below.
\end{remark}

As a sufficient condition for the operator $\emph{L}$ in the case of
non-integer $\theta$ to be an isomorphism of the corresponding spaces is the
simple condition $\{\theta\}+\theta\alpha<1$. That is the following theorem is valid.

\begin{theorem}
\label{T2.3} Let $\theta>0$ be a non-integer and $\alpha> 0$ be such that
$\theta\alpha$ and $\theta+\theta\alpha$ are non-integers. Let further
$\sigma_{k}$, $k=1,...,r$ be positive numbers such that $\sigma_{k}\alpha$
and $\sigma_{k}+\sigma_{k}\alpha$, $k=1,...,r$, are nonintegers. If
$\{\theta\}+\theta\alpha<1$, then the operator $\emph{L}$ is an isomorphism of
the space $C^{\overline{\sigma}(1+\alpha),\theta+\theta\alpha}(\overline
{R_{T}^{N}})$ and the closed subspace of the space $C^{\overline{\sigma}%
\alpha,\theta\alpha}(\overline{R_{T}^{N}})\times C^{\overline{\sigma}%
(1+\alpha)}(R^{N})\times C^{\overline{\sigma}(1+\alpha)-\frac{1}{\theta
}\overline{\sigma}}(R^{N})\times...\times C^{\overline{\sigma}(1+\alpha
)-\frac{[\theta]}{\theta}\overline{\sigma}}$, which consists of the sets
$(f,u_{0},u_{1},...,u_{[\theta]})$ were $f$ and $u_{0}$ obey condition \eqref{1dop.1}.
\end{theorem}

\section{Some preliminaries.}

\label{s3}

This section is included for the completeness and it contains some known
properties of fractional differential operators of the present paper. All the
statements of this section are formulated in the form we need below.

Firstly, on the ground of Section 3, Ch.1, in \cite{Samko} (see, in
particular, Section 3.1, Corollary 1), directly from definitions \eqref{1.3},
\eqref{1.4.0}, \eqref{1.5} we infer the following proposition.

\begin{proposition}
\label{P3.1} The fractional differential operator $D_{\ast t}^{\theta}$ for
noninteger $\theta+\theta\alpha$ ($\theta\alpha\in(0,1)$) and $\{\theta
\}+\theta\alpha<1$ is a bounded linear operator from $C^{\theta+\theta\alpha
}([0,T])$ to $C^{\theta\alpha}([0,T])$. That is if $u(t)\in C^{\theta
+\theta\alpha}([0,T])$, then
\begin{equation}
|D_{\ast t}^{\theta}u(t)|_{[0,T]}^{(\theta\alpha)}\leq C(\alpha,\theta
,T)|u(t)|_{[0,T]}^{(\theta+\theta\alpha)}, \label{3.1}%
\end{equation}
where $|u(t)|_{[0,T]}^{(\lambda)}$ is the norm in the space $C^{\lambda
}([0,T])$, $\lambda>0$.

If now $u(t)\in C^{\theta+\theta\alpha}([0,T])$ and
\[
u^{(i)}(0)=0,\quad i=0,1,...,[\theta+\theta\alpha],
\]
then \eqref{3.1} is valid without the restriction $\{\theta\}+\theta\alpha<1$.
\end{proposition}

%=============================
Further, the Caputo - Jrbashyan derivative possesses the following property
(compare \cite{Kilbas}, Theorem 2.2).

\begin{proposition}
\label{P3.1.00} If a function $u(t)\in C^{\theta+ \theta\alpha}([0,T])$ with a
positive noninteger $\theta$, then it's derivative $D_{\ast t}^{\theta}u(t)$
vanishes at $t=0$ that is
\begin{equation}
D_{\ast t}^{\theta}u(t)|_{t=0}=0. \label{3.1.00}%
\end{equation}
In other words, if a function $u(t)$ has some higher smoothness, than a
fractional $\theta> 0$, then it's Caputo - Jrbashyan derivative of order
$\theta$ necessarily vanishes at $t=0$.
\end{proposition}

We include a short proof for the completeness.

\begin{proof}
Note first that, in view of the second equality in representation \eqref{1.5},
we can without loss of generality assume $\theta\in(0,1)$.

Further, if $\theta+ \theta\alpha\geq1$, then property \eqref{3.1.00} follows
directly from the first equality in representation \eqref{1.5}, since in this
case the expression under the integral is integrable.

If now $\theta+ \theta\alpha< 1$, then after some smoothing of $u(t)$ (for
example, as it was done in Section \ref{s5.4} below) we obtain a sequence $\{
u_{n}(t) \}$ of the class $C^{\infty}([0,T])$ with
\begin{equation}
u_{n}(t) \rightarrow u(t) \ \ \text{in} \ \ C^{\theta+ \beta}([0,T]),
\ \ \beta\in(0, \theta\alpha).\label{3.1.01}%
\end{equation}
Since each member $u_{n}(t)$ of the sequence has the smoothness, which is
grater, than one, then all functions $u_{n}(t)$ possess property
\eqref{3.1.00}. Moreover, from \eqref{3.1} it follows, in particular, that the
derivatives $D_{\ast t}^{\theta}u_{n}(t)$ converge uniformly on $[0,T]$ (at
least for a subsequence) to the derivative $D_{\ast t}^{\theta}u(t)$ of the
original function. This means that the function $u(t)$ itself satisfies \eqref{3.1.00}.
\end{proof}

%============================

Besides, Example 3.9 in \cite{Umarov} contains an explicit solution to the
simplest Cauchy problem%

\[
D_{\ast t}^{\theta}u(t)=h(t),\quad t\in\lbrack0,T],\quad h(t)\in
C([0,T]),\quad n-1<\theta<n,
\]%
\[
u^{(k)}(0)=a_{k},\quad k=0,1,...,n-1,
\]
and the mentioned solution is expressed as
\[
u(t)=J^{\theta}h(t)+{\sum\limits_{k=0}^{n-1}}\frac{a_{k}}{k!}t^{k},
\]
where
\begin{equation}
J^{\theta}h(t)\equiv\frac{1}{\Gamma(\theta)}{\int\limits_{0}^{t}}%
(t-\tau)^{-1+\theta}h(\tau)d\tau.\label{3.2}%
\end{equation}

And from this, again on the base of Corollary 1 in Section 3, Ch.1 of
\cite{Samko} and on the base of Lemma 13.1 in the same monograph (with the
taking into account the Remark after this lemma), follows a statement, which
is the inverse to Proposition \ref{P3.1}.

\begin{proposition}
\label{P3.1.1} If $D_{\ast t}^{\theta}u(t)\in C^{\theta\alpha}([0,T])$,
$\theta\alpha\in(0,1)$ and $D_{\ast t}^{\theta}u(t)|_{t=0}=0$,
$u^{(k)}(0)=0$, $k=0,...,[\theta]$, then for a noninteger $\theta+\theta
\alpha$ the function $u(t)\in\underline{C}^{\theta+\theta\alpha}([0,T])$ and
\begin{equation}
|u(t)|_{[0,T]}^{(\theta+\theta\alpha)}\leq C(\alpha,\theta,T)|D_{\ast
t}^{\theta}u(t)|_{[0,T]}^{(\theta\alpha)}. \label{3.4}%
\end{equation}
And if $\theta+\theta\alpha>[\theta]+1$, then also
\begin{equation}
u^{([\theta]+1)}(0)=0. \label{3.4.0}%
\end{equation}

\end{proposition}

We formulate now in the form we need below some known properties of the
fractional Laplace operator in $R^{M}$.

Consider the operator $(-\Delta)^{\frac{\sigma}{2}}$ with a noninteger
$\sigma>0$ on functions $u(x)$, $x\in R^{M}$, where $R^{M}$, $M=1,2,...$, is
Euclidian space. Representation \eqref{1.9} for this operator has the form
\begin{equation}
(-\Delta)^{\frac{\sigma}{2}}u(x)=C_{M,\sigma,m}{\int\limits_{\eta\in R^{M}}%
}\frac{\delta_{\eta,x}^{m}u(x)}{|\eta|^{M+\sigma}}d\eta,\quad m>\sigma
,\label{3.4.1}%
\end{equation}
where, remind, $\delta_{\eta,x}u(x)=u(x+\eta)-u(x)$, $\delta_{\eta,x}%
^{m}u(x)=\delta_{\eta,x}(\delta_{\eta,x}^{m-1}u(x))$ is the difference of
order $m$. Directly from this representation, on the ground of classical
estimates for elliptic potentials (see, for example, \cite{Gilbarg}) it
follows that the operator $(-\Delta)^{\frac{\sigma}{2}}$ is correctly defined
on the space $C^{\sigma+\beta}(R^{M})$, $\beta>0$, with some noninteger
$\beta$ and $\sigma+\beta$. And in addition
\begin{equation}
|(-\Delta)^{\frac{\sigma}{2}}u|_{R^{M}}^{(\beta)}\leq C(M,\sigma
,\beta)|u|_{R^{M}}^{(\sigma+\beta)}\label{3.5}%
\end{equation}
(it is convenient to choose $m>\sigma+\beta$ in \eqref{3.4.1}). Note here,
that the analogous estimate for the corresponding highest H\"{o}lder seminorms
of the functions $u(x)$ and $(-\Delta)^{\frac{\sigma}{2}}u(x)$ is, in fact, a
particular case of Theorem 1 in Section 5.2.3 of \cite{21}. This theorem is
proved for the general homogeneous Besov spaces $B_{\beta}^{a,b}(R^{M})$, and
in the particular case $B_{\beta}^{\infty,\infty}(R^{M})=C^{\beta}(R^{M})$ the
assertion of the theorem gives for the H\"{o}lder seminorms
\begin{equation}
\left\langle (-\Delta)^{\frac{\sigma}{2}}u\right\rangle _{R^{M}}^{(\beta)}\leq
C(M,\sigma,\beta)\left\langle u\right\rangle _{R^{M}}^{(\sigma+\beta
)}.\label{3.5.1}%
\end{equation}

Without touching on here the question about the image of $C^{\sigma+\beta
}(R^{M})$ under the action of the operator $(-\Delta)^{\frac{\sigma}{2}}$ (and
this image does coincide with the whole $C^{\beta}(R^{M})$ even in the case of
the classical Laplace operator), we formulate an assertion, which is an
inverse to inequality \eqref{3.5}. Namely, the following estimate is valid
(see \cite{22}, Theorem 1.2)
\begin{equation}
|u|_{R^{M}}^{(\sigma+\beta)}\leq C(M,\sigma,\beta)\left(  |(-\Delta
)^{\frac{\sigma}{2}}u|_{R^{M}}^{(\beta)}+|u|_{R^{M}}^{(0)}\right)  ,
\label{3.6}%
\end{equation}
and also for lonely the highest seminorm (\cite{21}, Section 5.2.3, Theorem
1)
\begin{equation}
\left\langle u\right\rangle _{R^{M}}^{(\sigma+\beta)}\leq C(M,\sigma
,\beta)\left\langle (-\Delta)^{\frac{\sigma}{2}}u\right\rangle _{R^{M}%
}^{(\beta)}. \label{3.6.1}%
\end{equation}

\section{Operators $(-\Delta)^{\frac{\sigma}{2}}$ on the Schwartz spaces
$S(R^{M})$ and $S^{\prime}(R^{M})$}

\label{sS}

Below we need some assertion about continuity of the operator $(-\Delta
)^{\frac{\sigma}{2}}$ on the Schwartz spaces $S(R^{M})$ and $S^{\prime}%
(R^{M})$. It is known that the operator $(-\Delta)^{\frac{\sigma}{2}}$ does
not map the space $S(R^{M})$ to itself. Therefore this operator is not defined
on the whole space $S^{\prime}(R^{M})$. However, it is possible to highlight
some subclasses of the space $S^{\prime}$, where the mentioned operator is
continuously defined $S^{\prime}\rightarrow S^{\prime}$. We confine ourself
only to a subclass we need below. To a pity, the author could not find in
literature some appropriate for us assertions or proofs. Therefore, for the
completeness, we give here some assertions needed and some simple proofs.

We start with the action of $(-\Delta)^{\frac{\sigma}{2}}$ on $S(R^{M})$.
Denote a usual partial derivative of a function $u(x)\in S(R^{M})$ by
\[
D_{x}^{\overline{\omega}}u(x)\equiv\frac{\partial^{\omega_{1}}\partial
^{\omega_{2}}...\partial^{\omega_{M}}u(x)}{\partial x_{1}^{\omega_{1}}\partial
x_{2}^{\omega_{2}}...\partial x_{M}^{\omega_{M}}},\quad\overline{\omega
}=(\omega_{1},\omega_{2},...,\omega_{M})
\]
and for a function $u(x)\in S(R^{M})$ denote it's (semi)norm
\begin{equation}
\left\langle u\right\rangle _{R^{M}}^{n,p}={\sum_{|\overline{\omega}|\leq n}%
}\sup_{x\in R^{M}}|D_{x}^{\overline{\omega}}u(x)(1+|x|)^{p}|,\quad
n=0,1,2,...,\quad p>0.\label{3.10}%
\end{equation}

\begin{lemma}
\label{LS.1} Let $u(x)\in S(R^{M})$. Then $(-\Delta)^{\frac{\sigma}{2}}u(x)\in
C^{\infty}(R^{M})$ and for any $\varepsilon>0$
\begin{equation}
\left\vert D_{x}^{\overline{\omega}}\left[  (-\Delta)^{\frac{\sigma}{2}%
}u(x)\right]  \right\vert \leq C(M,\sigma,\varepsilon)\left\langle
u\right\rangle _{R^{M}}^{m+|\overline{\omega}|,p(\varepsilon)}%
(1+|x|)^{-M-\sigma+\varepsilon}.\label{S.1}%
\end{equation}

\end{lemma}

\begin{proof}
Note that for an arbitrary $\eta\in R^{M}$
\begin{equation}
D_{x}^{\overline{\omega}}\left(  \delta_{\eta,x}^{m}u(x)\right)  =\delta
_{\eta,x}^{m}\left(  D_{x}^{\overline{\omega}}u(x)\right)  \label{3.9}%
\end{equation}
and consequently, on the ground of the mean value theorem, for any $x,\eta\in
R^{M}$
\begin{equation}
\left\vert D_{x}^{\overline{\omega}}\left(  \delta_{\eta,x}^{m}u(x)\right)
\right\vert =\left\vert \delta_{\eta,x}^{m}\left(  D_{x}^{\overline{\omega}%
}u(x)\right)  \right\vert \leq C(M,m)|D_{x}^{m+|\overline{\omega}%
|}u|_{B_{m|\eta|}(x)}^{(0)}|\eta|^{m},\label{3.8}%
\end{equation}
where $B_{m|\eta|}(x)$ is the centered in $x$ ball of radius $m|\eta|$,
$\left\vert \overline{\omega}\right\vert =\omega_{1}+\omega_{2}+...+\omega
_{M}$,
\[
|D_{x}^{m+|\overline{\omega}|}u|_{B_{m|\eta|}(x)}^{(0)}\equiv{\sum
\limits_{|\rho|=m+|\overline{\omega}|}}\left\vert D_{x}^{\overline{\rho}%
}u(x)\right\vert _{B_{m|\eta|}(x)}^{(0)},\quad\overline{\rho}=(\rho_{1}%
,\rho_{2},...,\rho_{M}).
\]
Note also that for $|\eta|\leq|x|/2m$ and for any $p>0$ the value of
$|D_{x}^{m+|\overline{\omega}|}u|_{B_{m|\eta|}(x)}^{(0)}$ in \eqref{3.8} can
be estimated as follows
\begin{equation}
|D_{x}^{m+|\overline{\omega}|}u|_{B_{m|\eta|}(x)}^{(0)}\leq
C(M,m,p)\left\langle u\right\rangle _{R^{M}}^{m+|\overline{\omega}%
|,p}(1+|x|)^{-p},\quad|\eta|\leq|x|/2m.\label{3.11}%
\end{equation}
Denote, besides, for $\delta\in(0,1/10)$ and for $x\in R^{M}$
\begin{equation}
U_{\delta}(x)\equiv{\bigcup\limits_{j=1}^{m}}B_{|x|^{\delta}}(-\frac{x}%
{j}),\label{S.2}%
\end{equation}
where $B_{|x|^{\delta}}(-x/j)$ is the the centered in $-x/j$ ball of radius
$|x|^{\delta}$. Considering $x$ with $|x|>(2m)^{1/(1-\delta)}$, we see that on
the set $U_{\delta}(x)$ we have
\begin{equation}
\eta\in U_{\delta}(x)\Rightarrow|\eta|>\frac{|x|}{2m}.\label{S.3}%
\end{equation}
Consider now a derivative $D_{x}^{\overline{\omega}}$ of $\ (-\Delta
)^{\frac{\sigma}{2}}u(x)$. Use representation \eqref{3.4.1} and split the
integral over $R^{M}$ into the four domains
\begin{equation}
D_{x}^{\overline{\omega}}\left[  (-\Delta)^{\frac{\sigma}{2}}u(x)\right]
=C_{M,\sigma,m}{\int\limits_{|\eta|\leq\frac{1}{2m}|x|}}\frac{\delta_{\eta
,x}^{m}\left(  D_{x}^{\overline{\omega}}u(x)\right)  }{|\eta|^{M+\sigma}}%
d\eta+\label{3.12}%
\end{equation}%
\[
+C_{M,\sigma,m}{\int\limits_{\left\{  \frac{1}{2m}|x|<|\eta|\leq2m|x|\right\}
\cap U_{\delta}(x)}}\frac{\delta_{\eta,x}^{m}\left(  D_{x}^{\overline{\omega}%
}u(x)\right)  }{|\eta|^{M+\sigma}}d\eta+
\]%
\[
+C_{M,\sigma,m}{\int\limits_{\left\{  \frac{1}{2m}|x|<|\eta|\leq2m|x|\right\}
\backslash U_{\delta}(x)}}\frac{\delta_{\eta,x}^{m}\left(  D_{x}%
^{\overline{\omega}}u(x)\right)  }{|\eta|^{M+\sigma}}d\eta+
\]%
\[
+C_{M,\sigma,m}{\int\limits_{2m|x|<|\eta|}}\frac{\delta_{\eta,x}^{m}\left(
D_{x}^{\overline{\omega}}u(x)\right)  }{|\eta|^{M+\sigma}}d\eta\equiv
I_{1}+I_{2}+I_{3}+I_{4}.
\]
Taking advantage of \eqref{3.8}, \eqref{3.10}, estimate the integral $I_{1}$
as follows, bearing in mind that $m>\sigma$ and that on the domain of
integration $|\eta|\leq\frac{1}{2m}|x|$,
\[
\left\vert I_{1}\right\vert \leq C_{M,\sigma,m,k}\left\langle u\right\rangle
_{R^{M}}^{m+|\overline{\omega}|,p}{\int\limits_{|\eta|\leq\frac{1}{2m}|x|}%
}\frac{|\eta|^{m}(1+|x|)^{-p}}{|\eta|^{M+\sigma}}d\eta\leq
\]%
\[
\leq C_{M,\sigma,m,k}\left\langle u\right\rangle _{R^{M}}^{m+|\overline
{\omega}|,p}(1+|x|)^{-p}{\int\limits_{|\eta|\leq\frac{1}{2m}|x|}}\frac
{1}{|\eta|^{M-(m-\sigma)}}d\eta\leq
\]%
\begin{equation}
\leq C_{M,\sigma,m,k}\left\langle u\right\rangle _{R^{M}}^{m+|\overline
{\omega}|,p}(1+|x|)^{-p+m-\sigma},\label{3.13}%
\end{equation}
where $p>m-\sigma$.

Note now that
\begin{equation}
\delta_{\eta}^{m}\left(  D_{x}^{\overline{\omega}}u(x)\right)  ={\sum
\limits_{j=0}^{m}}C(i,m)D_{x}^{\overline{\omega}}u(x+j\eta).\label{S.4}%
\end{equation}
Taking into account \eqref{S.3} and the fact that the total measure of
$U_{\delta}(x)$ does not exceed $\left\vert U_{\delta}(x)\right\vert \leq
C|x|^{M\delta}$, we have for the integral $I_{2}$ in \eqref{3.12}
\[
\left\vert I_{2}\right\vert \leq C_{M,\sigma,m}\left\langle u\right\rangle
_{R^{M}}^{|\overline{\omega}|,0}{\int\limits_{U_{\delta}(x)}}\frac{1}%
{|\eta|^{M+\sigma}}d\eta\leq
\]%
\[
\leq C_{M,\sigma,m}\left\langle u\right\rangle _{R^{M}}^{|\overline{\omega
}|,0}{\int\limits_{U_{\delta}(x)}}\frac{1}{|x|^{M+\sigma}}d\eta,
\]
which gives under the condition $|x|>(2m)^{1/(1-\delta)}$ the estimate
\begin{equation}
\left\vert I_{2}\right\vert \leq C_{M,\sigma,m}\left\langle u\right\rangle
_{R^{M}}^{|\overline{\omega}|,0}|x|^{M\delta}|x|^{-M-\sigma}\leq
C_{M,\sigma,m}\left\langle u\right\rangle _{R^{M}}^{|\overline{\omega}%
|,0}(1+|x|)^{-M-\sigma+M\delta}.\label{3.15}%
\end{equation}

To estimate $I_{3}$ make again use of \eqref{S.4} and note that on the domain
of integration all the arguments of all the functions in \eqref{S.4} satisfy
the condition $\left\vert x+j\eta\right\vert \geq\left\vert x\right\vert
^{\delta}/m$ and thus, with the help of \eqref{3.11},
\begin{equation}
\left\vert I_{3}\right\vert \leq C_{M,\sigma,m,p}\left\langle u\right\rangle
_{R^{M}}^{|\overline{\omega}|,p}\left(  1+\left\vert x\right\vert \right)
^{-p\delta}{\int\limits_{\frac{1}{2m}|x|<|\eta|\leq2m|x|}}\frac{1}%
{|\eta|^{M+\sigma}}d\eta\leq\label{S.5}%
\end{equation}%
\[
\leq C_{M,\sigma,m,p}\left\langle u\right\rangle _{R^{M}}^{|\overline{\omega
}|,p}\left(  1+\left\vert x\right\vert \right)  ^{-p\delta}|x|^{-M-\sigma
}|x|^{M}\leq C_{M,\sigma,m,p}\left\langle u\right\rangle _{R^{M}}%
^{|\overline{\omega}|,p}\left(  1+\left\vert x\right\vert \right)
^{-p\delta-\sigma}.
\]

Turning at last to $I_{4}$, we avail ourselves of the fact that on the domain
of integration in \eqref{S.4}
\[
\left\vert x+j\eta\right\vert \geq%
%TCIMACRO{\QDATOPD{\{}{.}{|x|,\quad j=0,}{|\eta|/2m,\quad j>0}}%
%BeginExpansion
\genfrac{\{}{.}{0pt}{0}{|x|,\quad j=0,}{|\eta|/2m,\quad j>0}%
%EndExpansion
\]
and therefore
\[
\left\vert \delta_{\eta}^{m}\left(  D_{x}^{\overline{\omega}}u(x)\right)
\right\vert \leq C_{m,p}\left\langle u\right\rangle _{R^{M}}^{|\overline
{\omega}|,p}\left(  1+\left\vert x\right\vert \right)  ^{-p}+C_{m,p}%
\left\langle u\right\rangle _{R^{M}}^{|\overline{\omega}|,p}\left(
1+\left\vert \eta\right\vert \right)  ^{-p}.
\]
Consequently,
\[
\left\vert I_{4}\right\vert \leq C_{M,\sigma,m,p}\left\langle u\right\rangle
_{R^{M}}^{|\overline{\omega}|,p}(1+|x|)^{-p}{\int\limits_{|\eta|>2m|x|}}%
\frac{1}{|\eta|^{M+\sigma}}d\eta+
\]%
\begin{equation}
+C_{M,\sigma,m,p}\left\langle u\right\rangle _{R^{M}}^{|\overline{\omega}%
|,p}{\int\limits_{|\eta|>2m|x|}}\frac{(1+|\eta|)^{-p}}{|\eta|^{M+\sigma}}%
d\eta\leq C_{M,\sigma,m,k}\left\langle u\right\rangle _{R^{M}}^{|\overline
{\omega}|,k}(1+|x|)^{-p-\sigma}\label{S.6}%
\end{equation}

The above estimates for the integrals $I_{1}$ - $I_{2}$ show that, first, we
can differentiate under the sign of the integral in \eqref{3.4.1} and for
$u(x)\in S(R^{M})$ the function $(-\Delta)^{\frac{\sigma}{2}}u(x)$ belongs to
the class $C^{\infty}(R^{M})$. And the second, from \eqref{3.13},
\eqref{3.15}, \eqref{S.5} and \eqref{S.6} it follows that for
$|x|>(2m)^{1/(1-\delta)}$ we have
\[
\left\vert D_{x}^{\overline{\omega}}\left[  (-\Delta)^{\frac{\sigma}{2}%
}u(x)\right]  \right\vert \leq C_{M,\sigma,m,p} \left\langle u\right\rangle
_{R^{M}}^{m+|\overline{\omega}|,p} \times
\]%
\[
\times\left[  (1+|x|)^{-p+m-\sigma}+(1+|x|)^{-M-\sigma+M\delta}+\left(
1+\left\vert x\right\vert \right)  ^{-p\delta-\sigma}+(1+|x|)^{-p-\sigma
}\right]  .
\]
First choosing here sufficiently small $\delta$, $M\delta=\varepsilon$, and
then choosing $p$ sufficiently large, $p=M/\delta+m$, we arrive at the lemma statement.
\end{proof}

The proved lemma permits to define the operator $(-\Delta)^{\frac{\sigma}{2}}$
on a subclass $S_{\sigma}^{\prime}(R^{M})\subset S^{\prime}(R^{M})$ of the
class $S^{\prime}(R^{M})$, where
\begin{equation}
S_{\sigma}^{\prime}(R^{M})=\left\{  u(x)\in S^{\prime}(R^{M}):|u(x)|\leq
C(1+|x|)^{b},\quad C>0,\quad b<\sigma\right\} . \label{S.7}%
\end{equation}

\begin{lemma}
\label{LS.2} Operator $(-\Delta)^{\frac{\sigma}{2}}$ is correctly defined on
$S_{\sigma}^{\prime}(R^{M})\subset S^{\prime}(R^{M})$ in the usual sense of
the space $S^{\prime}(R^{M})$.
\end{lemma}

This statement follows directly from Lemma \ref{LS.1}, since for $u(x)\in
S_{\sigma}^{\prime}(R^{M})$ and $\varphi(x)\in S(R^{M})$ in the sense of
duality
\[
\left\langle (-\Delta)^{\frac{\sigma}{2}}u(x),\varphi(x)\right\rangle
\equiv\left\langle u(x),(-\Delta)^{\frac{\sigma}{2}}\varphi(x)\right\rangle ,
\]
where the last operation is correctly defined due to Lemma \ref{LS.1} and the
definition of $S_{\sigma}^{\prime}(R^{M})$. As an another simple consequence
of Lemma \ref{LS.1} we have the following assertion, which we present here
without a proof.

\begin{lemma}
\label{LS.3.1} Let for a sequence of measurable locally bounded functions
$\{u_{n}(x)\}\subset S_{\sigma}^{\prime}(R^{M})$ there exist such independent
on $n$ constants $C>0$ and $b<\sigma$ that
\[
|u_{n}(x)|\leq C(1+|x|)^{b},\quad n=1,2,...
\]
Let also there exits such a function $u(x)$, that $u_{n}(x)$ converges to
$u(x)$ on each ball $B_{R}=\{x\in R^{M}:|x|\leq R\}$. Then
\[
u(x)\in S_{\sigma}^{\prime}(R^{M}); \quad u_{n}(x)\rightarrow_{S^{\prime}
(R^{M})}u(x),
\]
\begin{equation}
(-\Delta)^{\frac{\sigma}{2}}u_{n}(x)\rightarrow_{S^{\prime}(R^{M})}%
(-\Delta)^{\frac{\sigma}{2}}u(x),\quad n\rightarrow\infty. \label{S.7.1}%
\end{equation}

\end{lemma}

Define now an anisotropic analog of the class $S_{\sigma}^{\prime}(R^{M})$,
which is related to the studying of the operator
\[
\emph{M}u\equiv{\sum_{k=1}^{r}}(-\Delta_{z_{k}})^{\frac{\sigma_{k}}{2}}u
\]
that is to the space operator in equation \eqref{1.1}. Namely, we denote
\[
S_{\overline{\sigma}}^{\prime}(R^{N})\equiv
\]%
\begin{equation}
\left\{  u(x)\in S^{\prime}(R^{N}):|u(x)|\leq C{\prod\limits_{k=1}^{r}%
}(1+|z_{k}|)^{b_{k}},C>0,b_{k}<\sigma_{k},k=\overline{1,r}\right\}
,\label{S.7.2}%
\end{equation}
where $z_{k}$ are defined in \eqref{1.0}. Directly from lemmas \ref{LS.1} -
\ref{LS.3.1} we infer the following assertion.

\begin{proposition}
\label{PS.1} Let for a sequence of measurable locally bounded functions
$\{u_{n}(x)\}\subset S_{\overline{\sigma}}^{\prime}(R^{N})$ there exist such
independent on $n$ constants $C>0$ and $b_{k}<\sigma_{k}$, $k=1,...,r$ that
\[
|u_{n}(x)|\leq C{\prod\limits_{k=1}^{r}}(1+|z_{k}|)^{b_{k}},\quad n=1,2,...
\]
If there exists such a function $u(x)$ that $u_{n}(x)$ uniformly converges to
$u(x)$ on each ball $B_{R}=\{x\in R^{N}:|x|\leq R\}$, then
\[
u(x)\in S_{\overline{\sigma}}^{\prime}(R^{N});\quad u_{n}(x)\rightarrow
_{S^{\prime}(R^{N})}u(x),
\]%
\begin{equation}
{\sum_{k=1}^{r}}(-\Delta_{z_{k}})^{\frac{\sigma_{k}}{2}}u_{n}(x)\rightarrow
_{S^{\prime}(R^{N})}{\sum_{k=1}^{r}}(-\Delta_{z_{k}})^{\frac{\sigma_{k}}{2}%
}u(x),\quad n\rightarrow\infty.\label{S.7.3}%
\end{equation}

\end{proposition}

\section{Operators $D_{\ast t}^{\theta}$ and $(-\Delta)^{\frac{\sigma}{2}}$ on
the Lizorkin spaces $\Phi(R^{M})$ and $\Phi^{\prime}(R^{M})$}

\label{sfi}

In this section we present, for the completeness, some known facts about the
acting of fractional differential operators on the Lizorkin spaces $\Phi
(R^{M})$ and $\Phi^{\prime}(R^{M})$. These spaces are a subspace and a
quotient space correspondingly of the spaces $S(R^{M})$ and $S^{\prime}%
(R^{M})$. They permit to generalize known properties of usual differential
operators on the spaces $S(R^{M})$ and $S^{\prime}(R^{M})$ to fractional
differential operators. We present here briefly some necessary for us
definitions and facts according to the corresponding presentation in
\cite{Samko}, section 8.2, where the reader can find more detailed information.

Define first the space $\Psi(R^{M})$, This space is the closed subspace of
$S(R^{M})$, consisting of function from $S(R^{M})$ that vanish at $x=0$
together with all their derivatives. That is
\begin{equation}
\Psi(R^{M})\equiv\left\{  \psi(x)\in S(R^{M}):\quad D_{x}^{\overline{\omega}%
}\psi(0)=0,\quad|\overline{\omega}|=0,1,2,...\right\}  . \label{fi.1}%
\end{equation}

The Lizorkin space $\Phi(R^{M})$ is the closed subspace of $S(R^{M})$,
consisting of functions from $S(R^{M})$ with the Fourier images from the space
$\Psi(R^{M})$. That is
\begin{equation}
\Phi(R^{M})\equiv\left\{  \varphi(x)\in S(R^{M}):\quad\widehat{\varphi}%
(\xi)\in\Psi(R^{M})\right\}  .\label{fi.2}%
\end{equation}
Besides, the space $\Phi(R^{M})$ can be equivalently defined as the closed
subspace of $S(R^{M})$, consisting of function that are orthogonal to all
polynomials. That is
\[
\varphi(x)\in\Phi(R^{M})\Leftrightarrow{\int\limits_{R^{M}}}\varphi
(x)P(x)dx=0,\forall P(x),
\]
where $P(x)$ is an arbitrary polynomial.

Since $\Psi(R^{M})$ and $\Phi(R^{M})$ closed subspaces of $S(R^{M})$, their
topologies are induced by the topology of $S(R^{M})$.

The dual to the space $\Psi(R^{M})$ is denoted by $\Psi^{\prime}(R^{M})$ and
it is the quotient space of $S^{\prime}(R^{M})$ along the closed subspace of
$S^{\prime}(R^{M})$, consisting of distributions with the supports at $x=0$.
It is well known that such distributions are exactly those that are finite
linear combinations of the Dirac function $\delta(x)$ with the support at
$x=0$ and it's derivatives.

At the same time, the dual to $\Phi(R^{M})$ space $\Phi^{\prime}(R^{M})$, is
the quotient space of $S^{\prime}(R^{M})$ along the space of polynomials
$P(x)$ (which is a closed subspace of $S^{\prime}(R^{M})$). That is the
elements of $\Phi^{\prime}(R^{M})$ are exactly the equivalence classes of the
distributions from $S^{\prime}(R^{M})$ modulo polynomials.

The topologies in the spaces $\Psi^{\prime}(R^{M})$ and $\Phi^{\prime}(R^{M})$
are the corresponding quotient topologies.

Fractional differential operators $D_{\ast t}^{\theta}$ and $(-\Delta
)^{\frac{\sigma}{2}}$ are defined and continuous from the space $\Phi(R^{M})$
to itself and from $\Phi^{\prime}(R^{M})$ to itself. For the operator
$(-\Delta)^{\frac{\sigma}{2}}$ this fact follows directly from the definition
of the space $\Phi(R^{M})$ and from definition \eqref{1.8}. The same is also
valid for the operator $D_{\ast t}^{\theta}$, when we consider it on functions
that vanish at $t=0$ together with all their derivatives with respect to $t$
up to the order $[\theta]$. At that one needs to take into account relation
\eqref{4.27} and Remark \ref{R4.1} below in section \ref{s4}.

According to the definitions of $\Psi(R^{M})$ and $\Phi(R^{M})$, the Fourier
transform is a linear homeomorphism from $\Phi(R^{M})$ to $\Psi(R^{M})$ and,
correspondingly, from $\Phi^{\prime}(R^{M})$ to $\Psi^{\prime}(R^{M})$.

Moreover, for the distributions from $\Phi^{\prime}(R^{N+1})$ all formulas of
the Fourier transform for fractional derivatives are preserved. That is if
$u(x,t)\in\Phi^{\prime}(R^{N+1})$, then, analogously to \eqref{1.8},
\begin{equation}
\widehat{(-\Delta_{z_{k}})^{\frac{\sigma_{k}}{2}}u(x,t)}(\xi,\xi_{0}%
)=|\eta_{k}|^{\sigma_{k}}\widehat{u}(\xi,\xi_{0}). \label{fi.3}%
\end{equation}
And if $u(x,t)\in\Phi^{\prime}(R^{N+1})$ and it's support along with the
supports of all it's derivatives in $t$ up to the order $[\theta]$ are
included in the set $\{t\geq0\}$, then
\begin{equation}
\widehat{D_{\ast t}^{\theta}u(x,t)}(\xi,\xi_{0})=(i\xi_{0})^{[\theta]}%
(i\xi_{0})^{\{\theta\}}\widehat{u}(\xi,\xi_{0}). \label{fi.4}%
\end{equation}

We are going to make use of these circumstances below at the proof of the
uniqueness for problem \eqref{1.1}, \eqref{1.2}.

\section{Operators $D_{\ast t}^{\theta}$ and $(-\Delta)^{\frac{\sigma}{2}}$ on
H\"{o}lder spaces $C^{\overline{\sigma}(1+\alpha),\theta+\theta\alpha
}(\overline{R_{T}^{N}})$ and $C^{\overline{\sigma}(1+\alpha)}(R^{N})$}

\label{Hold}

Turning to the H\"{o}lder spaces, introduced in \eqref{2.8} and \eqref{2.10},
we note that these spaces are defined in terms of the corresponding
coordinate-wise smoothness. Therefore on the grounds of \eqref{2.4},
\eqref{3.5} and \eqref{3.6} the following assertion can be obtained.

\begin{proposition}
\label{P3.5} The operator $(-\Delta_{z_{k}})^{\frac{\sigma_{k}}{2}}$,
$k=1,...,r$ is a bounded linear operator from $C^{\overline{\sigma}(1+\alpha
)}(R^{N})$ to $C^{\overline{\sigma}\alpha}(R^{N})$. That is for $u(x)\in
C^{\overline{\sigma}(1+\alpha)}(R^{N})$%
\begin{equation}
|(-\Delta_{z_{k}})^{\frac{\sigma_{k}}{2}}u|_{R^{N}}^{(\overline{\sigma}%
\alpha)}\leq C(N,N_{k},\overline{\sigma},\alpha)|u|_{R^{N}}^{\left(
\overline{\sigma}(1+\alpha)\right)  } \label{3.18}%
\end{equation}
and for the H\"{o}lder seminorm of the function $(-\Delta_{z_{k}}
)^{\frac{\sigma_{k}}{2}}u$ with respect to another group of variables $z_{l}$,
$l \neq k$, the following interpolation inequality is valid with an arbitrary
$\varepsilon>0$
\begin{equation}
\left\langle (-\Delta_{z_{k}})^{\frac{\sigma_{k}}{2}}u\right\rangle
_{z_{l},R^{N}}^{(\sigma_{l}\alpha)}\leq C\varepsilon^{\alpha}\left\langle
u\right\rangle _{z_{k},R^{N}}^{(\sigma_{k}+\sigma_{k}\alpha)}+\frac
{C}{\varepsilon}\left\langle u\right\rangle _{z_{l},R^{N}}^{(\sigma_{l}%
+\sigma_{l}\alpha)}. \label{3.18.1}%
\end{equation}

\end{proposition}

\begin{proof}
Firstly, as for the smoothness with respect to the group $z_{k}$, from
\eqref{3.5} it follows that uniformly in the other coordinate groups
\begin{equation}
|(-\Delta_{z_{k}})^{\frac{\sigma_{k}}{2}}u|_{R^{N}}^{(0)}+\left\langle
(-\Delta_{z_{k}})^{\frac{\sigma_{k}}{2}}u\right\rangle _{z_{k},R^{N}}%
^{(\sigma_{k}\alpha)}\leq C(N_{k},\sigma_{k},\alpha)|u|_{R^{N}}^{\left(
\overline{\sigma}(1+\alpha)\right)  }.\label{3.19}%
\end{equation}
To show the estimates of the H\"{o}lder seminorms for the function
$(-\Delta_{z_{k}})^{\frac{\sigma_{k}}{2}}u(x)$ with respect to other groups
$z_{l}$, $l\neq k$, we again use representation \eqref{3.4.1}%
\begin{equation}
(-\Delta_{z_{k}})^{\frac{\sigma_{k}}{2}}u(x)=C_{N_{k},\sigma_{k},m}%
{\int\limits_{R^{N_{k}}}}\frac{\delta_{\eta,z_{k}}^{m}u(x)}{|\eta
|^{N_{k}+\sigma_{k}}}d\eta,\quad m>\sigma_{k}+\sigma_{k}\alpha.\label{3.20}%
\end{equation}
To estimate the seminorm $\left\langle (-\Delta_{z_{k}})^{\frac{\sigma_{k}}%
{2}}u\right\rangle _{z_{l},R^{N}}^{(\sigma_{l}\alpha)}$ according to
\eqref{2.4} we carry the difference with respect to $z_{l}$ inside the
integral in \eqref{3.20},
\[
\left\langle (-\Delta_{z_{k}})^{\frac{\sigma_{k}}{2}}u\right\rangle
_{z_{l},R^{N}}^{(\sigma_{l}\alpha)}=
\]%
\begin{equation}
=\sup_{x\in R^{N},h\in R^{N_{l}},h\neq0}\frac{C_{N_{k},\sigma_{k},m}%
}{|h|^{\sigma_{l}\alpha}}{\int\limits_{R^{N_{k}}}}\frac{\delta_{h,z_{l}}%
^{p}\delta_{\eta,z_{k}}^{m}u(x)}{|\eta|^{N_{k}+\sigma_{k}}}d\eta,\quad
p>\sigma_{l}+\sigma_{l}\alpha.\label{3.21}%
\end{equation}
Estimate the expression under the $\sup$-sign in \eqref{3.21} by splitting the
integral into the two parts
\[
{\int\limits_{R^{N_{k}}}}\frac{\delta_{h,z_{l}}^{m}\delta_{\eta,z_{k}}%
^{m}u(x)}{|\eta|^{N_{k}+\sigma_{k}}|h|^{\sigma_{l}\alpha}}d\eta={\int
\limits_{|\eta|\leq\varepsilon|h|^{\frac{\sigma_{l}}{\sigma_{k}}}}}%
\frac{\delta_{h,z_{l}}^{m}\delta_{\eta,z_{k}}^{m}u(x)}{|\eta|^{N_{k}%
+\sigma_{k}}|h|^{\sigma_{l}\alpha}}d\eta+
\]%
\[
+{\int\limits_{|\eta|\geq\varepsilon|h|^{\frac{\sigma_{l}}{\sigma_{k}}}}}%
\frac{\delta_{h,z_{l}}^{m}\delta_{\eta,z_{k}}^{m}u(x)}{|\eta|^{N_{k}%
+\sigma_{k}}|h|^{\sigma_{l}\alpha}}d\eta\equiv I_{1}+I_{2},
\]
where $\varepsilon>0$ is arbitrary but fixed. Note that according to
\eqref{2.4} and according to the definition of the finite differences
\[
|\delta_{h,z_{l}}^{m}\delta_{\eta,z_{k}}^{m}u(x)|\leq C\left\langle
u\right\rangle _{z_{l},R^{N}}^{(\sigma_{l}+\sigma_{l}\alpha)}|h|^{\sigma
_{l}+\sigma_{l}\alpha},
\]%
\begin{equation}
|\delta_{h,z_{l}}^{m}\delta_{\eta,z_{k}}^{m}u(x)|\leq C\left\langle
u\right\rangle _{z_{k},R^{N}}^{(\sigma_{k}+\sigma_{k}\alpha)}|\eta
|^{\sigma_{k}+\sigma_{k}\alpha}.\label{3.22}%
\end{equation}
Making use of the second of these inequalities, we estimate the integral
$I_{1}$ as follows
\[
|I_{1}|\leq C\left\langle u\right\rangle _{z_{k},R^{N}}^{(\sigma_{k}%
+\sigma_{k}\alpha)}{\int\limits_{|\eta|\leq\varepsilon|h|^{\frac{\sigma_{l}%
}{\sigma_{k}}}}}\frac{|\eta|^{\sigma_{k}+\sigma_{k}\alpha}}{|\eta
|^{N_{k}+\sigma_{k}}|h|^{\sigma_{l}\alpha}}d\eta=
\]%
\[
=\frac{C\left\langle u\right\rangle _{z_{k},R^{N}}^{(\sigma_{k}+\sigma
_{k}\alpha)}}{|h|^{\sigma_{l}\alpha}}{\int\limits_{|\eta|\leq\varepsilon
|h|^{\frac{\sigma_{l}}{\sigma_{k}}}}}\frac{d\eta}{|\eta|^{N_{k}-\sigma
_{k}\alpha}}=
\]%
\begin{equation}
=C\left\langle u\right\rangle _{z_{k},R^{N}}^{(\sigma_{k}+\sigma_{k}\alpha
)}|h|^{-\sigma_{l}\alpha}\left(  \varepsilon|h|^{\frac{\sigma_{l}}{\sigma_{k}%
}}\right)  ^{\sigma_{k}\alpha}=C\varepsilon^{\sigma_{k}\alpha}\left\langle
u\right\rangle _{z_{k},R^{N}}^{(\sigma_{k}+\sigma_{k}\alpha)}.\label{3.23}%
\end{equation}
Analogously, taking advantage of the first inequality in \eqref{3.22}, we have
for the integral $I_{2}$,
\[
|I_{2}|\leq C\left\langle u\right\rangle _{z_{l},R^{N}}^{(\sigma_{l}%
+\sigma_{l}\alpha)}{\int\limits_{|\eta|\geq\varepsilon|h|^{\frac{\sigma_{l}%
}{\sigma_{k}}}}}\frac{|h|^{\sigma_{l}+\sigma_{l}\alpha}}{|\eta|^{N_{k}%
+\sigma_{k}}|h|^{\sigma_{l}\alpha}}d\eta=
\]%
\[
=C\left\langle u\right\rangle _{z_{l},R^{N}}^{(\sigma_{l}+\sigma_{l}\alpha
)}|h|^{\sigma_{l}}{\int\limits_{|\eta|\geq\varepsilon|h|^{\frac{\sigma_{l}%
}{\sigma_{k}}}}}\frac{d\eta}{|\eta|^{N_{k}+\sigma_{k}}}=
\]%
\begin{equation}
=C\left\langle u\right\rangle _{z_{l},R^{N}}^{(\sigma_{l}+\sigma_{l}\alpha
)}|h|^{\sigma_{l}}\left(  \varepsilon|h|^{\frac{\sigma_{l}}{\sigma_{k}}%
}\right)  ^{-\sigma_{k}}=\frac{C}{\varepsilon^{\sigma_{k}}}\left\langle
u\right\rangle _{z_{l},R^{N}}^{(\sigma_{l}+\sigma_{l}\alpha)}.\label{3.24}%
\end{equation}
Estimate \eqref{3.18.1} follows now from \eqref{3.21}, \eqref{3.23} and
\eqref{3.24} after the change of variables $\varepsilon^{\sigma_{k}%
}\rightarrow\varepsilon$. And this, together with \eqref{3.19}, completes the proof.
\end{proof}

We formulate now some more general assertion as a theorem.

\begin{theorem}
\label{Tdop1} Let us be given a number $\rho_{k}\in(0,\sigma_{k}+\sigma
_{k}\alpha)$ for some $k\in\{1,2,...,r\}$ so that it can be expressed as
$\rho_{k}=(1-\omega)\sigma_{k}(1+\alpha)$, $\omega\in(0,1)$. Denote
\begin{equation}
\overline{\rho}\equiv(1-\omega)\overline{\sigma}(1+\alpha). \label{3dop1}%
\end{equation}
The operator $(-\Delta_{z_{k} })^{\frac{\rho_{k}}{2}}$, $k=1,...,r$, is a
bounded linear operator from $C^{\overline{\sigma}(1+\alpha)}(R^{N})$ to
$C^{\omega\overline{\sigma}(1+\alpha)}(R^{N})$ that is for $u(x)\in
C^{\overline{\sigma}(1+\alpha)}(R^{N})$
\begin{equation}
|(-\Delta_{z_{k}})^{\frac{\rho_{k}}{2}}u|_{R^{N}}^{(\omega\overline{\sigma
}(1+\alpha))}\leq C(N,N_{k},\omega,\overline{\sigma},\alpha)|u|_{R^{N}%
}^{\left(  \overline{\sigma}(1+\alpha)\right)  }. \label{3dop2}%
\end{equation}
Moreover, for the H\"{o}lder seminorm of the function $(-\Delta_{z_{k}}
)^{\frac{\rho_{k}}{2}}u$ with respect to some another group of the variables
$z_{l}$, $l \neq k$, we have the following interpolation inequality with an
arbitrary $\varepsilon>0$
\begin{equation}
\left\langle (-\Delta_{z_{k}})^{\frac{\rho_{k}}{2}}u\right\rangle
_{z_{l},R^{N}}^{(\omega\sigma_{l}(1+\alpha))}\leq C\varepsilon^{\frac{\omega
}{1-\omega}}\left\langle u\right\rangle _{z_{k},R^{N}}^{(\sigma_{k}+\sigma
_{k}\alpha)}+\frac{C}{\varepsilon}\left\langle u\right\rangle _{z_{l},R^{N}%
}^{(\sigma_{l}+\sigma_{l}\alpha)}. \label{3dop3}%
\end{equation}

\end{theorem}

\begin{proof}
The theorem follows directly from the previous proposition. It is enough to
note that
\[
\overline{\sigma}(1+\alpha)=\overline{\rho}(1+\beta),\quad\beta\equiv\frac
{1}{1-\omega}-1=\frac{\omega}{1-\omega}>0,
\]
and thus $C^{\overline{\sigma}(1+\alpha)}(R^{N})=C^{\overline{\rho}(1+\beta
)}(R^{N})$. Consequently, the proof follows from Proposition \ref{P3.5}.
\end{proof}

Move now to considering of the space $C^{\overline{\sigma}(1+\alpha
),\theta+\theta\alpha}(R_{T}^{N})$ of the functions with the independent
variable $t$ besides the independent variables $x\in R^{N}$.

\begin{proposition}
\label{P3.6} The operator $(-\Delta_{z_{k}})^{\frac{\sigma_{k}}{2}}$,
$k=1,...,r$, is a bounded linear operator from $C^{\overline{\sigma}%
(1+\alpha),\theta+\theta\alpha}(\overline{R_{T}^{N}})$ to $C^{\overline
{\sigma}\alpha,\theta\alpha}(\overline{R_{T}^{N}})$ that is for $u(x,t)\in
C^{\overline{\sigma}(1+\alpha),\theta+\theta\alpha}(R_{T}^{N})$
\begin{equation}
|(-\Delta_{z_{k}})^{\frac{\sigma_{k}}{2}}u|_{\overline{R_{T}^{N}}}%
^{(\overline{\sigma}\alpha,\theta\alpha)}\leq C(N,N_{k},\overline{\sigma
},\theta,\alpha)|u|_{\overline{R_{T}^{N}}}^{\left(  \overline{\sigma}%
(1+\alpha),\theta+\theta\alpha\right)  }.\label{3.28}%
\end{equation}
Moreover, for the H\"{o}der seminorm of $(-\Delta_{z_{k}} )^{\frac{\sigma_{k}%
}{2}}u$ with respect to some another group of the variables $z_{l}$, $l \neq
k$, we have the following interpolation inequality with an arbitrary
$\varepsilon>0$
\begin{equation}
\left\langle (-\Delta_{z_{k}})^{\frac{\sigma_{k}}{2}}u\right\rangle
_{z_{l},\overline{R_{T}^{N}}}^{(\sigma_{l}\alpha)}\leq C\varepsilon^{\alpha
}\left\langle u\right\rangle _{z_{k},\overline{R_{T}^{N}}}^{(\sigma_{k}%
+\sigma_{k}\alpha)}+\frac{C}{\varepsilon}\left\langle u\right\rangle
_{z_{l},\overline{R_{T}^{N}}}^{(\sigma_{l}+\sigma_{l}\alpha)}, \label{3.29}%
\end{equation}
and for the H\"{o}der seminorm of $(-\Delta_{z_{k}} )^{\frac{\sigma_{k}}{2}}u$
with respect to $t$ we have
\begin{equation}
\left\langle (-\Delta_{z_{k}})^{\frac{\sigma_{k}}{2}}u\right\rangle
_{t,\overline{R_{T}^{N}}}^{(\theta\alpha)}\leq C\varepsilon^{\alpha
}\left\langle u\right\rangle _{z_{k},\overline{R_{T}^{N}}}^{(\sigma_{k}%
+\sigma_{k}\alpha)}+\frac{C}{\varepsilon}\left\langle u\right\rangle
_{t,\overline{R_{T}^{N}}}^{(\theta+\theta\alpha)}. \label{3.30}%
\end{equation}

\end{proposition}

The proof of this proposition is identical to the proof of Proposition
\ref{P3.5} because the proofs of inequalities \eqref{3.29} and \eqref{3.30}
are identical to the proof of \eqref{3.18.1}.

As a direct consequence of this proposition we have the following more general
assertion, whose proof replicates the proof of Theorem \ref{Tdop1}.

\begin{theorem}
\label{T3dop2} Let us be given a number $\rho_{k}\in(0,\sigma_{k}+\sigma
_{k}\alpha)$ for some $k\in\{1,2,...,r\}$ so that $\rho_{k}=(1-\omega
)\sigma_{k}(1+\alpha)$, $\omega\in(0,1)$. Denote
\begin{equation}
\overline{\rho}\equiv(1-\omega)\overline{\sigma}(1+\alpha). \label{3dop4}%
\end{equation}
The operator $(-\Delta_{z_{k} })^{\frac{\rho_{k}}{2}}$, $k=1,...,r$, is a
linear bounded operator from $C^{\overline{\sigma}(1+\alpha),\theta
+\theta\alpha}(\overline{R_{T}^{N}})$ to $C^{\omega\overline{\sigma}%
(1+\alpha),\omega\theta(1+\alpha)}(\overline{R_{T}^{N}})$ that is for
$u(x,t)\in C^{\overline{\sigma}(1+\alpha),\theta+\theta\alpha}(R_{T}^{N})$%
\begin{equation}
|(-\Delta_{z_{k}})^{\frac{\rho_{k}}{2}}u|_{\overline{R_{T}^{N}}}%
^{(\omega\overline{\sigma}(1+\alpha),\omega\theta(1+\alpha))}\leq
C(N,N_{k},\omega,\overline{\sigma},\theta,\alpha)|u|_{\overline{R_{T}^{N}}%
}^{\left(  \overline{\sigma}(1+\alpha),\theta+\theta\alpha\right)  }.
\label{3dop5}%
\end{equation}
Moreover for the H\"{o}lder seminorm of $(-\Delta_{z_{k}} )^{\frac{\rho_{k}%
}{2}}u$ with respect to some another group op the variables $z_{l}$, $l \neq
k$, the following interpolation inequality with an arbitrary $\varepsilon>0$
is valid
\begin{equation}
\left\langle (-\Delta_{z_{k}})^{\frac{\rho_{k}}{2}}u\right\rangle
_{z_{l},\overline{R_{T}^{N}}}^{(\omega\sigma_{l}(1+\alpha))}\leq
C\varepsilon^{\frac{\omega}{1-\omega}}\left\langle u\right\rangle
_{z_{k},\overline{R_{T}^{N}}}^{(\sigma_{k}+\sigma_{k}\alpha)}+\frac
{C}{\varepsilon}\left\langle u\right\rangle _{z_{l},\overline{R_{T}^{N}}%
}^{(\sigma_{l}+\sigma_{l}\alpha)}. \label{3dop6}%
\end{equation}
And for the H\"{o}lder seminorm of $(-\Delta_{z_{k}} )^{\frac{\rho_{k}}{2}}u$
with respect to $t$ the following inequality is valid
\begin{equation}
\left\langle (-\Delta_{z_{k}})^{\frac{\sigma_{k}}{2}}u\right\rangle
_{t,\overline{R_{T}^{N}}}^{(\omega\theta(1+\alpha))}\leq C\varepsilon
^{\frac{\omega}{1-\omega}}\left\langle u\right\rangle _{z_{k},\overline
{R_{T}^{N}}}^{(\sigma_{k}+\sigma_{k}\alpha)}+\frac{C}{\varepsilon}\left\langle
u\right\rangle _{t,\overline{R_{T}^{N}}}^{(\theta+\theta\alpha)}.
\label{3dop7}%
\end{equation}

\end{theorem}

In what follows we will need three more useful assertions.

\begin{lemma}
\label{L3.2} Let $l_{1}>0$, $l_{2}>0$ and let $\overline{\Omega}\subset
R^{N_{1}+N_{2}}$ be a domain in $R^{N_{1}+N_{2}}$ with the boundary
$\partial\Omega$ of the class $C^{\max\{l_{1},l_{2}\}}$. Let, further, a
function $u(z_{1},z_{2})\in C_{z_{1},z_{2}}^{l_{1},l_{2}}(\overline{\Omega})$,
$(z_{1},z_{2})\in\overline{\Omega}$, $z_{1}\in R^{N_{1}}$, $z_{2}\in R^{N_{2}%
}$. Then for any integers $m>l_{1}$, $k>l_{2}$, for any positive $a$ and $b$
with
\begin{equation}
\frac{a}{l_{1}}+\frac{b}{l_{2}}=1, \label{3.33}%
\end{equation}
and for any $\varepsilon>0$ the following estimate is valid
\[
\left\vert \delta_{\tau,z_{1}}^{m}\delta_{h,z_{2}}^{k}u(z_{1},z_{2}%
)\right\vert \leq C(\overline{\Omega},l_{1},l_{2}) \times
\]
\begin{equation}
\times\left[  \varepsilon^{l_{1}-a}\left(  \left\langle u\right\rangle
_{z_{1},\overline{\Omega}}^{(l_{1})}+|u|_{\overline{\Omega}}^{(0)}\right)
+\varepsilon^{-a}\left(  \left\langle u\right\rangle _{z_{2},\overline{\Omega
}}^{(l_{2})}+|u|_{\overline{\Omega}}^{(0)}\right)  \right]  \left\vert
\tau\right\vert ^{a}\left\vert h\right\vert ^{b}, \label{3.34}%
\end{equation}
where $\tau$ and $h$ are such that the arguments of the function $\delta
_{\tau,z_{1}} ^{m}\delta_{h,z_{2}}^{k}u(z_{1},z_{2})$ stay inside $\Omega$.

If now the domain $\Omega$ is the space $R^{N_{1}+N_{2}}$ or a half-space, or
a domain of the form
\begin{equation}
\Omega=\{x\in R^{N_{1}+N_{2}}:x_{i_{1}}>0,...,x_{i_{k}}>0\}, \label{3.34.1}%
\end{equation}
then the terms $|u|_{\overline{\Omega}}^{(0)}$ in \eqref{3.34} can be omitted.
\end{lemma}

\begin{proof}
Extend the function $u(z_{1},z_{2})$ from the domain $\overline{\Omega}$ to
the whole space $R^{N_{1}+N_{2}}$ with the preserving (up to a multiple
constant) it's norm in $C_{z_{1},z_{2}}^{l_{1},l_{2}}(\overline{\Omega})$ to a
finite function (see., for example, \cite{Sol1}, \cite{LadParab}). That is
\[
|u|_{R^{N_{1}+N_{2}}}^{(l_{1},l_{2})}\leq C(\overline{\Omega},l_{1}%
,l_{2})|u|_{\overline{\Omega}}^{(l_{1},l_{2})},
\]
but note that the corresponding seminorms satisfy (in general)%

\[
\left\langle u\right\rangle _{z_{i},R^{N_{1}+N_{2}}}^{(l_{i})}\leq C \left(
\left\langle u\right\rangle _{z_{i},\overline{\Omega}}^{(l_{i})}%
+|u|_{\overline{\Omega}}^{(0)}\right) .
\]
However, for a domain $\Omega$ of the form \eqref{3.34.1} the last inequality
is valid without the term $|u|_{\overline{\Omega}}^{(0)}$ and this term does
not plays a role in the further reasonings. Overall, such an extension permits
to suppose that $\overline{\Omega}$ coincides with the whole space
$R^{N_{1}+N_{2}}$.

Consider the two possible cases of relations between $\left\vert
\tau\right\vert $ and $\left\vert h\right\vert $. Let first $\left\vert
\tau\right\vert \leq\varepsilon\left\vert h\right\vert ^{l_{2}/l_{1}}$. Then,
making use of \eqref{2.4}, we have
\[
\left\vert \delta_{\tau,z_{1}}^{m}\delta_{h,z_{2}}^{k}u(z_{1},z_{2}%
)\right\vert \leq C (m,k)\left\langle u\right\rangle _{z_{1},\overline{\Omega
}}^{(l_{1})}\left\vert \tau\right\vert ^{l_{1}}=C (m,k)\left\langle
u\right\rangle _{z_{1},\overline{\Omega}}^{(l_{1})}\left\vert \tau\right\vert
^{a}\left\vert \tau\right\vert ^{l_{1}-a}\leq
\]%
\begin{equation}
\leq C (m,k)\left\langle u\right\rangle _{z_{1},\overline{\Omega}}^{(l_{1}%
)}\left\vert \tau\right\vert ^{a}\left(  \varepsilon\left\vert h\right\vert
^{\frac{l_{2}}{l_{1}}}\right)  ^{l_{1}-a}=C (m,k)\varepsilon^{l_{1}%
-a}\left\langle u\right\rangle _{z_{1},\overline{\Omega}}^{(l_{1})}\left\vert
\tau\right\vert ^{a}\left\vert h\right\vert ^{b}, \label{3.35}%
\end{equation}
since $b=l_{2}-al_{2}/l_{1}$ in view of \eqref{3.33}.

Let now $\left\vert \tau\right\vert \geq\varepsilon\left\vert h\right\vert
^{l_{2}/l_{1}}$ that is $\left\vert h\right\vert \leq\varepsilon^{-l_{1}%
/l_{2}}\left\vert \tau\right\vert ^{l_{1}/l_{2}}$. Then we have, analogously
to the previous case,
\[
\left\vert \delta_{\tau,z_{1}}^{m}\delta_{h,z_{2}}^{k}u(z_{1},z_{2}%
)\right\vert \leq C (m,k)\left\langle u\right\rangle _{z_{2},\overline{\Omega
}}^{(l_{2})}\left\vert h\right\vert ^{l_{2}}=C (m,k)\left\langle
u\right\rangle _{z_{2},\overline{\Omega}}^{(l_{2})}\left\vert h\right\vert
^{b}\left\vert h\right\vert ^{l_{2}-b}\leq
\]%
\begin{equation}
\leq C (m,k)\left\langle u\right\rangle _{z_{2},\overline{\Omega}}^{(l_{2}%
)}\left\vert h\right\vert ^{b}\left(  \varepsilon^{-l_{1}/l_{2}}\left\vert
\tau\right\vert ^{\frac{l_{1}}{l_{2}}}\right)  ^{l_{2}-b}=C (m,k)\varepsilon
^{-a}\left\langle u\right\rangle _{z_{1},\overline{\Omega}}^{(l_{1}%
)}\left\vert \tau\right\vert ^{a}\left\vert h\right\vert ^{b}, \label{3.36}%
\end{equation}
since $a=l_{1}-bl_{1}/l_{2}$ in view of \eqref{3.33}.

The lemma follows now from \eqref{3.35} and \eqref{3.36}.
\end{proof}

The proved lemma is valid also in the case when the function under
consideration depends only on a single variable $z$ that is when $z_{1}=z_{2}%
$. The proof in this case is a simple replication of the proof of Lemma
\ref{L3.2} therefore we give the following assertion without a proof.

\begin{corollary}
\label{C3.1} Let $l>0$ and let $\overline{\Omega}\subset R^{N}$ be a domain in
$R^{N}$ with the boundary $\partial\Omega$ of the class $C^{l}$. Let us also
be given a function $u(z)\in C^{l} (\overline{\Omega})$. Let, finally, $m>l$
and $k>l$ be arbitrary integers and let positive numbers $a$ and $b$ be such
that
\begin{equation}
\frac{a}{l}+\frac{b}{l}=1\Leftrightarrow a+b=l . \label{3.36.1}%
\end{equation}
Then
\begin{equation}
\left\vert \delta_{\tau,z}^{m}\delta_{h,z}^{k}u(z)\right\vert \leq
C(\overline{\Omega},l)\left(  \left\langle u\right\rangle _{z,\overline
{\Omega}}^{(l)}+|u|_{\overline{\Omega}}^{(0)}\right)  \left\vert
\tau\right\vert ^{a}\left\vert h\right\vert ^{b}, \label{3.36.2}%
\end{equation}
where $\tau$ and $h$ are taken in the way that the arguments of $\delta
_{\tau,z}^{m}\delta_{h,z}^{k}u(z)$ stay in the domain $\Omega$.

If now the the domain $\Omega$ coincides with $R^{N}$ or with a half-space, or
with a domain of the form
\begin{equation}
\Omega=\{z\in R^{N}:z_{i_{1}}>0,...,z_{i_{k}}>0\}, \label{3.36.3}%
\end{equation}
then the terms $|u|_{\overline{\Omega}}^{(0)}$ in \eqref{3.36.2} can be omitted
\end{corollary}

From Lemma \ref{L3.2} and Corollary \ref{C3.1} one can infer a useful
assertion, which generalizes, in a sense, some analogous assertion for
H\"{o}lder spaces from \cite{Mushel}. We mean the assertion number 5 in
section 6, \cite{Mushel}, for the one-dimensional case.

\begin{proposition}
\label{DopMushel} Let a function $u(x)\in C\overline{^{l} }(\overline{\Omega
})$ in the sense of definitions \eqref{2.1} - \eqref{2.4}, $\overline
{l}=(l_{1},...,l_{N})$. Let, further, a domain $\Omega$ be the whole space
$R^{N}$, a half-space or a domain of the form \eqref{3.36.3}. Let also $k$ be
a fixed index, $k\in\{1,...,N\}$ and let $h>0$. Let, finally, $a\in(0,l_{k})$
be such that $l_{k}-a$ is a noninteger. Consider for an integer $m>l_{k}$ the
function
\begin{equation}
u_{a,h,k}(x)\equiv\frac{\delta_{h,x_{k}}^{m}u(x)}{h^{a}}. \label{3.36.4}%
\end{equation}

The function $u_{a,h,k}(x)$ belongs to the space $C^{\overline{l}\omega_{a}%
}(\overline{\Omega})$,
\begin{equation}
\omega_{a}\equiv(1-a/l_{k}),\quad\overline{l}\omega_{a}=(l_{1}\omega
_{a},...,l_{N}\omega_{a}), \label{3.36.5}%
\end{equation}
and uniformly in $h$
\begin{equation}
\left\Vert u_{a,h,k}(x)\right\Vert _{C^{\overline{l}\omega_{a}}(\overline
{\Omega})}\equiv|u_{a,h,k}(x)|_{\overline{\Omega}}^{(\overline{l}\omega_{a}%
)}\leq C|u(x)|_{\overline{\Omega}}^{(\overline{l})} \label{3.36.6}%
\end{equation}
with a constant $C$, which does not depend on $h$. In particular, for an
arbitrary $\varepsilon>0$
\begin{equation}
\left\langle u_{a,h,k}(x)\right\rangle _{x_{i},\overline{\Omega}}^{(\omega
_{a}l_{i})}\leq%
%TCIMACRO{\QDATOPD{\{}{.}{C(l_{i},l_{k},m)\left(  \varepsilon^{l_{i}%
%-a}\left\langle u\right\rangle _{x_{i},\overline{\Omega}}^{(l_{i}%
%)}+\varepsilon^{-a}\left\langle u\right\rangle _{x_{k},\overline{\Omega}%
%}^{(l_{k})}\right)  ,\quad i\neq k,}{C(l_{k},m)\left\langle u\right\rangle
%_{x_{k},\overline{\Omega}}^{(l_{k})},\quad i=k,} }%
%BeginExpansion
\genfrac{\{}{.}{0pt}{0}{C(l_{i},l_{k},m)\left(  \varepsilon^{l_{i}%
-a}\left\langle u\right\rangle _{x_{i},\overline{\Omega}}^{(l_{i}%
)}+\varepsilon^{-a}\left\langle u\right\rangle _{x_{k},\overline{\Omega}%
}^{(l_{k})}\right)  ,\quad i\neq k,}{C(l_{k},m)\left\langle u\right\rangle
_{x_{k},\overline{\Omega}}^{(l_{k})},\quad i=k,}
%EndExpansion
\label{3.36.7}%
\end{equation}%
\begin{equation}
\left\vert u_{a,h,k}(x)\right\vert _{\overline{\Omega}}^{(0)}\leq C(m)\left(
\varepsilon^{l_{k}-a}\left\langle u\right\rangle _{x_{k},\overline{\Omega}%
}^{(l_{k})}+\varepsilon^{-a}\left\vert u\right\vert _{\overline{\Omega}}%
^{(0)}\right)  . \label{3.36.8}%
\end{equation}

\end{proposition}

\begin{proof}
Estimates \eqref{3.36.7} are a direct consequence of Lemma \ref{L3.2} and
Corollary \ref{C3.1} with the taking into account the properties of H\"{o}lder
seminorms in \eqref{2.4}. To obtain estimate \eqref{3.36.8} we first suppose
that $h\leq\varepsilon$. Then
\[
\left\vert u_{a,h,k}(x)\right\vert _{\overline{\Omega}}^{(0)}=\left\vert
\frac{\delta_{h,x_{k}}^{m}u(x)}{h^{a}}\right\vert _{\overline{\Omega}}%
^{(0)}=\left\vert \frac{\delta_{h,x_{k}}^{m}u(x)}{h^{l_{k}}}\right\vert
_{\overline{\Omega}}^{(0)}h^{l_{k}-a}\leq\varepsilon^{l_{k}-a}\left\langle
u\right\rangle _{x_{k},\overline{\Omega}}^{(l_{k})}.
\]
If now $h>\varepsilon$, then
\[
\left\vert u_{a,h,k}(x)\right\vert _{\overline{\Omega}}^{(0)}=\left\vert
\frac{\delta_{h,x_{k}}^{m}u(x)}{h^{a}}\right\vert _{\overline{\Omega}}%
^{(0)}\leq\varepsilon^{-a}\left\vert \delta_{h,x_{k}}^{m}u(x)\right\vert
_{\overline{\Omega}}^{(0)}\leq C(m)\varepsilon^{-a}\left\vert u\right\vert
_{\overline{\Omega}}^{(0)},
\]
which completes the proof of the proposition.
\end{proof}

Consider, further, the acting of the Caputo - Jrbashyan derivative $D_{\ast
t}^{\theta}$ in the introduced anisotropic H\"{o}lder spaces.

\begin{proposition}
\label{P3.8} Let nonintegers $\theta>0$ and $\alpha>0$ be such that
\begin{equation}
\lbrack\theta+\theta\alpha]=[\theta],\text{ that is }\theta\alpha
<1-\{\theta\}. \label{3.37}%
\end{equation}
Then the operator of the Caputo - Jrbashyan derivative $D_{\ast t}^{\theta}$
is a bounded linear operator from $C^{\overline{\sigma}(1+\alpha
),\theta+\theta\alpha}(\overline{R_{T}^{N}})$ to $C^{\overline{\sigma}%
\alpha,\theta\alpha}(\overline{R_{T}^{N}})$ that is for a function $u(x,t)\in
C^{\overline{\sigma}(1+\alpha),\theta+\theta\alpha}(\overline{R_{T}^{N}})$ the
following estimate is valid
\begin{equation}
|D_{\ast t}^{\theta}u|_{\overline{R_{T}^{N}}}^{(\overline{\sigma}\alpha
,\theta\alpha)}\leq C(N,\overline{\sigma},\theta,\alpha)|u|_{\overline
{R_{T}^{N}}}^{\left(  \overline{\sigma}(1+\alpha),\theta+\theta\alpha\right)
}. \label{3.38}%
\end{equation}

\end{proposition}

\begin{proof}
Firstly, from \eqref{3.1} it follows that
\begin{equation}
|D_{\ast t}^{\theta}u|_{\overline{R_{T}^{N}}}^{(0)}+\left\langle D_{\ast
t}^{\theta}u\right\rangle _{t,\overline{R_{T}^{N}}}^{(\theta\alpha)}\leq
C(\alpha,\theta,T)|u|_{\overline{R_{T}^{N}}}^{(\overline{\sigma}%
(1+\alpha),\theta+\theta\alpha)}.\label{3.39.1}%
\end{equation}
Further, let $p-1<\theta<p$ with some positive integer $p$. Represent the
fractional derivative $D_{\ast t}^{\theta}u$ in the form (see \eqref{1.5},
$\theta-p+1=\{\theta\}$, $p-1=[\theta]$) by applying integration by parts
\[
D_{\ast t}^{\theta}u(x,t)=
\]%
\[
=C(\theta){\int\limits_{0}^{t}}\frac{u_{\tau}^{(p)}(x,\tau)d\tau}%
{(t-\tau)^{\{\theta\}}}=C(\theta){\int\limits_{0}^{t}}\frac{\left[  u_{\tau
}^{(p-1)}(x,\tau)-u_{t}^{(p-1)}(x,t)\right]  _{\tau}^{\prime}d\tau}%
{(t-\tau)^{\{\theta\}}}=
\]%
\[
=C(\theta)\frac{u_{t}^{([\theta])}(x,t)-u_{\tau}^{([\theta])}(x,0)}%
{(t-0)^{\{\theta\}}}+
\]%
\begin{equation}
+C(\theta){\int\limits_{0}^{t}}\frac{\left[  u_{\tau}^{([\theta])}%
(x,\tau)-u_{t}^{([\theta])}(x,t)\right]  d\tau}{(t-\tau)^{1+\{\theta\}}}\equiv
I_{1}+I_{2}.\label{3.40}%
\end{equation}
Note further that on the ground of \eqref{2.6} $u_{t}^{([\theta])}(x,t)\in
C^{\overline{\sigma}(1+\alpha)(1-\frac{[\theta]}{\theta+\theta\alpha
}),\{\theta\}+\theta\alpha}(\overline{R_{T}^{N}})$ and
\begin{equation}
\left\langle u_{t}^{([\theta])}\right\rangle _{t,\overline{R_{T}^{N}}%
}^{(\{\theta\}+\theta\alpha)}+{\sum\limits_{k=1}^{r}}\left\langle
u_{t}^{([\theta])}\right\rangle _{z_{k},\overline{R_{T}^{N}}}^{(\sigma
_{k}(1+\alpha)(1-\frac{[\theta]}{\theta+\theta\alpha}))}\leq C\left\Vert
u\right\Vert _{C^{\overline{\sigma}(1+\alpha),\theta+\theta\alpha}%
(\overline{R_{T}^{N}})}.\label{3.41}%
\end{equation}
Making use of the last inequality we estimate H\"{o}lder seminorms of the
expressions $I_{1}$ and $I_{2}$ with respect to some group of the space
variables $z_{n}$. Consider first the ratio $I_{1}$. For this we make use of
Lemma \ref{L3.2} as applied to $u_{t}^{([\theta])}(x,t)$ and to the variables
$t$ and $z_{n}$. Since $\{\theta\}+\theta\alpha<1$, we can take in
\eqref{3.34} the order of the difference in $t$ to be one, $m=1$. And we fix
some integer $k>\sigma_{n}(1+\alpha)(1-\frac{[\theta]}{\theta+\theta\alpha})$
as the order of the difference in the space variables $z_{n}$. To apply the
mentioned lemma we choose the exponent $a$ as $a=\{\theta\}$ and then the
exponent $b$ is defined from the relation
\[
\frac{\{\theta\}}{\{\theta\}+\theta\alpha}+\frac{b}{\sigma_{n}(1+\alpha
)(1-\frac{[\theta]}{\theta+\theta\alpha})}=1,
\]
that is
\[
b=\sigma_{n}(1+\alpha)\frac{\{\theta\}+\theta\alpha}{\theta+\theta\alpha}%
\frac{\theta\alpha}{\{\theta\}+\theta\alpha}=\sigma_{n}\alpha.
\]
Thus we have on the ground of \eqref{3.34} ($h\in R^{N_{n}}$)
\[
\left\vert \delta_{h,z_{n}}^{k}\left[  u_{t}^{([\theta])}(x,t)-u_{\tau
}^{([\theta])}(x,0)\right]  \right\vert \leq
\]%
\[
\leq C\left(  \left\langle u_{t}^{([\theta])}\right\rangle _{t,\overline
{R_{T}^{N}}}^{(\{\theta\}+\theta\alpha)}+\left\langle u_{t}^{([\theta
])}\right\rangle _{z_{n},\overline{R_{T}^{N}}}^{(\sigma_{n}(1+\alpha
)(1-\frac{[\theta]}{\theta+\theta\alpha}))}\right)  t^{\{\theta\}}%
|h|^{\sigma_{n}\alpha}%
\]
or, dividing both parts by $t^{\{\theta\}}$,
\[
\left\vert \delta_{h,z_{n}}^{k}I_{1}\right\vert \leq C\left(  \left\langle
u_{t}^{([\theta])}\right\rangle _{t,\overline{R_{T}^{N}}}^{(\{\theta
\}+\theta\alpha)}+\left\langle u_{t}^{([\theta])}\right\rangle _{z_{n}%
,\overline{R_{T}^{N}}}^{(\sigma_{n}(1+\alpha)(1-\frac{[\theta]}{\theta
+\theta\alpha}))}\right)  |h|^{\sigma_{n}\alpha}%
\]
that is
\begin{equation}
\left\langle I_{1}\right\rangle _{z_{n},\overline{R_{T}^{N}}}^{(\sigma
_{n}\alpha)}\leq C\left(  \left\langle u_{t}^{([\theta])}\right\rangle
_{t,\overline{R_{T}^{N}}}^{(\{\theta\}+\theta\alpha)}+\left\langle
u_{t}^{([\theta])}\right\rangle _{z_{n},\overline{R_{T}^{N}}}^{(\sigma
_{n}(1+\alpha)(1-\frac{[\theta]}{\theta+\theta\alpha}))}\right)  .\label{3.42}%
\end{equation}

Estimate now the seminorm of the integral $I_{2}$ in \eqref{3.40}. Let
 $h\in R^{N_{n}}$ be fixed. Suppose first
that $t>|h|^{\sigma_{n}/\theta}$ and decompose the integral in $I_{2}$ into
two parts as follows
\[
I_{2}={\int\limits_{t-|h|^{\sigma_{n}/\theta}}^{t}}\frac{\left[  u_{\tau
}^{([\theta])}(x,\tau)-u_{t}^{([\theta])}(x,t)\right]  d\tau}{(t-\tau
)^{1+\{\theta\}}}+
\]%
\[
+{\int\limits_{0}^{t-|h|^{\sigma_{n}/\theta}}}\frac{\left[  u_{\tau}%
^{([\theta])}(x,\tau)-u_{t}^{([\theta])}(x,t)\right]  d\tau}{(t-\tau
)^{1+\{\theta\}}}\equiv J_{1}+J_{2}.
\]
Estimate separately the finite differences $\delta_{h,z_{n}}^{k}J_{1}$ and
$\delta_{h,z_{n}}^{k}J_{2}$, $k>\sigma_{n}(1+\alpha)(1-\frac{[\theta]}%
{\theta+\theta\alpha})$. On the ground of the H\"{o}lder property of
$u_{t}^{([\theta])}(x,t)$ in $t$ we have for $J_{1}$
\[
|\delta_{h,z_{n}}^{k}J_{1}|\leq C\left\langle u_{t}^{([\theta])}\right\rangle
_{t,\overline{R_{T}^{N}}}^{(\{\theta\}+\theta\alpha)}{\int
\limits_{t-|h|^{\sigma_{n}/\theta}}^{t}}\frac{(t-\tau)^{\{\theta
\}+\theta\alpha}d\tau}{(t-\tau)^{1+\{\theta\}}}=
\]%
\begin{equation}
=C\left\langle u_{t}^{([\theta])}\right\rangle _{t,\overline{R_{T}^{N}}%
}^{(\{\theta\}+\theta\alpha)}\left(  |h|^{\sigma_{n}/\theta}\right)
^{\theta\alpha}=C\left\langle u_{t}^{([\theta])}\right\rangle _{t,\overline
{R_{T}^{N}}}^{(\{\theta\}+\theta\alpha)}|h|^{\sigma_{n}\alpha}.\label{3.43}%
\end{equation}
Use now the H\"{o}lder property of $u_{t}^{([\theta])}(x,t)$ in $z_{n}$ to
obtain
\[
|\delta_{h,z_{n}}^{k}J_{2}|\leq C\left\langle u_{t}^{([\theta])}\right\rangle
_{z_{n},\overline{R_{T}^{N}}}^{(\sigma_{n}(1+\alpha)(1-\frac{[\theta]}%
{\theta+\theta\alpha}))}|h|^{\sigma_{n}(1+\alpha)(1-\frac{[\theta]}%
{\theta+\theta\alpha})}{\int\limits_{0}^{t-|h|^{\sigma_{n}/\theta}}}%
\frac{d\tau}{(t-\tau)^{1+\{\theta\}}}=
\]%
\[
=C\left\langle u_{t}^{([\theta])}\right\rangle _{z_{n},\overline{R_{T}^{N}}%
}^{(\sigma_{n}(1+\alpha)(1-\frac{[\theta]}{\theta+\theta\alpha}))}%
|h|^{\sigma_{n}(1+\alpha)(1-\frac{[\theta]}{\theta+\theta\alpha})}\left(
|h|^{\frac{\sigma_{n}}{\theta}}\right)  ^{-\{\theta\}}=
\]%
\begin{equation}
=C\left\langle u_{t}^{([\theta])}\right\rangle _{z_{n},\overline{R_{T}^{N}}%
}^{(\sigma_{n}(1+\alpha)(1-\frac{[\theta]}{\theta+\theta\alpha}))}%
|h|^{\sigma_{n}\alpha}.\label{3.44}%
\end{equation}
If now $t\leq|h|^{\sigma_{n}/\theta}$, then analogously to the estimate for
$\delta_{h,z_{n}}^{k}J_{1}$ we have for $\delta_{h,z_{n}}^{k}I_{2}$
\begin{equation}
|\delta_{h,z_{n}}^{k}I_{2}|\leq C\left\langle u_{t}^{([\theta])}\right\rangle
_{t,\overline{R_{T}^{N}}}^{(\{\theta\}+\theta\alpha)}{\int\limits_{0}%
^{|h|^{\sigma_{n}/\theta}}}\frac{(t-\tau)^{\{\theta\}+\theta\alpha}d\tau
}{(t-\tau)^{1+\{\theta\}}}=C\left\langle u_{t}^{([\theta])}\right\rangle
_{t,\overline{R_{T}^{N}}}^{(\{\theta\}+\theta\alpha)}|h|^{\sigma_{n}\alpha
}.\label{3.45}%
\end{equation}
Collecting estimates \eqref{3.43} - \eqref{3.45}, we see that
\begin{equation}
\left\langle I_{2}\right\rangle _{z_{n},\overline{R_{T}^{N}}}^{(\sigma
_{n}\alpha)}\leq C\left(  \left\langle u_{t}^{([\theta])}\right\rangle
_{t,\overline{R_{T}^{N}}}^{(\{\theta\}+\theta\alpha)}+\left\langle
u_{t}^{([\theta])}\right\rangle _{z_{n},\overline{R_{T}^{N}}}^{(\sigma
_{n}(1+\alpha)(1-\frac{[\theta]}{\theta+\theta\alpha}))}\right)  .\label{3.46}%
\end{equation}

The assertion of the proposition and estimate \eqref{3.38} follow now from
\eqref{3.39.1}, \eqref{3.41}, \eqref{3.42} and \eqref{3.46}.
\end{proof}

We have also some more general assertion.

\begin{theorem}
\label{T3.2} Let $\{\theta\}+\theta\alpha<1$ and let a noninteger $\mu
\in(0,\theta+\theta\alpha)$ so that it can be represented as $\mu
=(1-\omega)\theta(1+\alpha)$, $\omega\in(0,1)$.

The operator $D_{\ast t}^{\mu}$ is a bonded linear operator from
$C^{\overline{\sigma}(1+\alpha),\theta+\theta\alpha}(\overline{R_{T}^{N}})$ to
$C^{\omega\overline{\sigma}(1+\alpha),\omega\theta(1+\alpha)}(\overline
{R_{T}^{N}})$ that is for $u(x,t)\in C^{\overline{\sigma}(1+\alpha
),\theta+\theta\alpha}(R_{T}^{N})$
\begin{equation}
|D_{\ast t}^{\mu}u|_{\overline{R_{T}^{N}}}^{(\omega\overline{\sigma}%
(1+\alpha),\omega\theta(1+\alpha))}\leq C(N,\omega,\overline{\sigma}%
,\theta,\alpha)|u|_{\overline{R_{T}^{N}}}^{\left(  \overline{\sigma}%
(1+\alpha),\theta+\theta\alpha\right)  }. \label{3.54}%
\end{equation}

\end{theorem}

\begin{proof}
The theorem is a direct consequence of the previous Proposition. It is enough
to note that
\[
\theta(1+\alpha)=\mu(1+\beta),\quad\overline{\sigma}(1+\alpha)=\overline{\rho
}(1+\beta),\quad\beta\equiv\frac{1}{1-\omega}-1=\frac{\omega}{1-\omega}>0,
\]
where
\[
\beta\equiv\frac{1}{1-\omega}-1=\frac{\omega}{1-\omega}>0,\quad\rho_{k}%
\equiv(1-\omega)\sigma_{k}(1+\alpha).
\]
Thus $C^{\overline{\sigma}(1+\alpha),\theta+\theta\alpha}(\overline{R_{T}^{N}%
})=C^{\overline{\rho}(1+\beta),\mu+\mu\beta}(\overline{R_{T}^{N}})$.
Consequently, the assertion of the theorem is obtained by the application of
Proposition \ref{P3.8}.
\end{proof}

It was shown in the example after Theorem \ref{T2.1} that the restriction
$\{\theta\}+\theta\alpha<1$ can not be omitted in general. However this
restriction is a possible sufficient condition for the assertions of
Proposition \ref{P3.8} and Theorem \ref{T3.2} to be valid. As a conclusion of
this section we present here some another sufficient condition. It is that all
the permitted by a class derivatives in $t$ vanish at $t=0$.

But first we prove an auxiliary lemma.

\begin{lemma}
\label{L3.1} Let a function $f(t)\in C^{\theta+\beta}([0,\infty))$, where
$\theta$ and $\theta+\beta$ are nonintegers and let at $t=0$ the function
$f(t)$ satisfy the condition
\begin{equation}
f_{t}^{(k)}(0)=0,\quad k=0,...,[\theta].\label{3.24.1}%
\end{equation}
Then it's fractional Caputo - Jrbashyan derivative $D_{\ast t}^{\theta}f(t)$,
which was defined in \eqref{1.3}, \eqref{1.5}, coincides with the Marshaud
derivative(see \cite{Samko}, section 5.4) that is it has also the
representation
\begin{equation}
D_{\ast t}^{\theta}f(t)=C(\theta,m){\int\limits_{0}^{\infty}}\frac
{\delta_{-\tau,t}^{m}f(t)}{\tau^{1+\theta}}d\tau\equiv C(\theta,m)T_{\theta
,m}(f).\label{3.25}%
\end{equation}
Here $\delta_{-\tau,t}f(t)=f(t)-f(t-\tau)$, $\delta_{-\tau,t}^{n}%
f(t)=$\ $\delta_{-\tau,t}\left(  \delta_{-\tau,t}^{m-1}f(t)\right)  $ is a
backward difference of an arbitrary but fixed order $m>\theta$ and the
function $f(t)$ is extended by identical zero in the domain $t<0$.
\end{lemma}

\begin{proof}
The proof reproduces the schema of reasonings from \cite{Samko}, section 25.4,
Lemma 25.3 for the fractional Riess differentiation as it is applied in
\cite{Samko} to obtain representation \eqref{1.9}. Namely, we compare the
Fourier transforms of the left and of the right hand sides of \eqref{3.25}.

Note that zero extension of the function $f(t)$ in the domain $t<0$ belongs to
the space $C^{\theta+\beta}$ in the case $\{\theta\}+\beta<1$ ($\{\theta\}$ is
the fractional part of $\theta$) and in the case $\{\theta\}+\beta>1$ this
extension belongs to the space $C^{\theta+\beta^{\prime}}$ with an arbitrary
$\beta^{\prime}<1-\{\theta\}$. Because of this the following estimate is
valid
\begin{equation}
|\delta_{-\tau,t}^{m}f(t)|\leq%
%TCIMACRO{\QDATOPD{\{}{.}{C(f)\tau^{\theta+\gamma},\quad t<1}{C(f),\quad
%t>1,}}%
%BeginExpansion
\genfrac{\{}{.}{0pt}{0}{C(f)\tau^{\theta+\gamma},\quad t<1}{C(f),\quad t>1,}%
%EndExpansion
,\quad\gamma=\min\{\beta,\beta^{\prime}\}, \label{3.25.1}%
\end{equation}
where we preserve the same notation $f(t)$ for the extension.

Suppose first that the function $f(t)$ is more smooth in the sense that it
possess the continuous derivative of order $n=[\theta]+1$ ($n-1<\theta<n$ ).
Besides, we first suppose that $f(t)$ is integrable and even finite (for
simplicity) for $t\rightarrow+\infty$. Then the Fourier transform of the
derivative $D_{\ast t}^{\theta}f(t)$ from \eqref{1.5} is calculated in, for
example, \cite{Samko}, section 7 and it is equal to (see Lemma \ref{L4.2}
below and Remark \ref{R4.1} after it)
\begin{equation}
F(D_{\ast t}^{\theta}f)(\xi)\equiv\widehat{D_{\ast t}^{\theta}f}(\xi
)=(i\xi)^{[\theta]}(i\xi)^{\{\theta\}}\widehat{f}(\xi). \label{3.26}%
\end{equation}
Here $\widehat{f}(\xi)$ the Fourier transform of $f(t)$ and $(i\xi
)^{\{\theta\}}$ is the following analytic extension of the function
$z^{\{\theta\}}$ from the positive real axis to the right half-plane
$\operatorname{Re}z>0$
\begin{equation}
(i\xi)^{\{\theta\}}=|\xi|^{\{\theta\}}e^{i\theta\frac{\pi}{2}sign\xi}
\label{3.27}%
\end{equation}

Calculate now the Fourier image of the integral $T_{\theta,m}(f)(t)$ in the
right hand side of \eqref{3.25}. We have (comp. \cite{Samko}, section 25.4,
Lemma 25.3)
\begin{equation}
\widehat{T_{\theta,m}(f)}(\xi)={\int\limits_{0}^{\infty}}\frac{\widehat
{\delta_{-\tau,t}^{m}f(t)}}{\tau^{1+\theta}}d\tau=\widehat{f}(\xi
){\int\limits_{0}^{\infty}}\frac{(1-e^{-i\tau\xi})^{m}}{\tau^{1+[\theta
]+\{\theta\}}}d\tau.\label{3.27.1}%
\end{equation}
Make in the last integral the change of variables $z=i\tau\xi$. After such
change the last integral for $\xi>0$ is transformed into the integral along
positive imaginary axis in the positive direction that is
\begin{equation}
\widehat{T_{\theta,m}(f)}(\xi)=(i\xi)^{[\theta]}(i\xi)^{\{\theta\}}\widehat
{f}(\xi){\int\limits_{0i}^{\infty i}}\frac{(1-e^{-z})^{m}}{z^{1+[\theta
]+\{\theta\}}}dz.\label{3.27.2}%
\end{equation}
Consider for $R>0$ the integral of the last integrand along the circular
contour
\[
C_{R}^{+}=[0i,Ri]+[R\cdot e^{i\varphi},\varphi\in(\pi/2,0)]-[0,R]
\]
This contour consists of the interval of the imaginary axis $[0i,Ri]$, of the
clock-wise oriented quarter of circle $[R\cdot e^{i\varphi},\varphi\in
(\pi/2,0)]$, and of the negatively oriented interval $[0,R]$ of the real axis.
Since the integrand is analytic and continuous in the right half-plane, the
integral along this contour $C_{R}^{+}$ is equal to zero. Besides,
\[
\left\vert {\int\limits_{[R\cdot e^{i\varphi},\varphi\in(\pi/2,0)]}}%
\frac{(1-e^{-z})^{m}}{z^{1+[\theta]+\{\theta\}}}dz\right\vert \leq\frac{2^{m}%
}{R^{1+[\theta]+\{\theta\}}}\frac{\pi}{2}R\rightarrow0,\quad R\rightarrow
\infty.
\]
Consequently, letting $R\rightarrow\infty$, we obtain un view of the
directions of integration
\[
{\int\limits_{0i}^{\infty i}}\frac{(1-e^{-z})^{m}}{z^{1+[\theta]+\{\theta\}}%
}dz={\int\limits_{0}^{\infty}}\frac{(1-e^{-x})^{m}}{x^{1+\theta}}%
dx\equiv1/C(m,\theta)>0.
\]
If now $\xi<0$, then the last integral in \eqref{3.27.1} after the change
$z=i\tau\xi$ goes to an analogous to \eqref{3.27.2} integral, but along the
negatively oriented imaginary axis. The analogous considerations for this case
give the same result that is
\[
{\int\limits_{0i}^{-\infty i}}\frac{(1-e^{-z})^{m}}{z^{1+[\theta]+\{\theta\}}%
}dz={\int\limits_{0}^{\infty}}\frac{(1-e^{-x})^{m}}{x^{1+\theta}}%
dx\equiv1/C(m,\theta)>0
\]
with the same constant $C(m,\theta)$. Thus for any sign of $\xi$ we get
\begin{equation}
\widehat{T_{\theta,m}(f)}(\xi)=C(m,\theta)(i\xi)^{[\theta]}(i\xi)^{\{\theta
\}}\widehat{f}(\xi).\label{3.27.3}%
\end{equation}
From \eqref{3.26} and \eqref{3.27.3} it follows that $\widehat{D_{\ast
t}^{\theta}f}(\xi)=C(m,\theta)\widehat{T_{\theta,m}(f)}(\xi)$, $\xi\in
(-\infty,\infty)$, which proves assertion \eqref{3.25} of the lemma for finite
functions $f(t)$ of the class $C^{[\theta]+1}([0,\infty))$.

General case is obtained now by a mollifying and by a cutting off the function
$f(t)$ with a subsequent limiting process in the final equality \eqref{3.25},
which does not depend on extra smoothness.
\end{proof}

On the ground of Lemma \ref{L3.1}, by replication of the proof of Proposition
\ref{P3.6}, we obtain the following proposition.

\begin{proposition}
\label{P3.7} For a function $u(x,t)\in C^{\overline{\sigma}(1+\alpha
),\theta+\theta\alpha}(\overline{R_{T}^{N}})$ with
\[
u_{t}^{(k)}(x,0)=0,\quad k=0,...,[\theta+\theta\alpha]
\]
the following estimate is valid
\begin{equation}
|D_{\ast t}^{\theta}u|_{\overline{R_{T}^{N}}}^{(\overline{\sigma}\alpha
,\theta\alpha)}\leq C(N,\overline{\sigma},\theta,\alpha)|u|_{\overline
{R_{T}^{N}}}^{\left(  \overline{\sigma}(1+\alpha),\theta+\theta\alpha\right)
}. \label{3.31}%
\end{equation}

Moreover, for the H\"{o}lder seminorm of the derivative $D_{\ast t}^{\theta}u$
with respect to a group $z_{k}$ of the space variables we have the following
interpolation inequality with an arbitrary $\varepsilon>0$
\begin{equation}
\left\langle D_{\ast t}^{\theta}u\right\rangle _{z_{k},\overline{R_{T}^{N}}%
}^{(\theta\alpha)}\leq C\varepsilon^{\alpha}\left\langle u\right\rangle
_{t,\overline{R_{T}^{N}}}^{(\theta+\theta\alpha)}+\frac{C}{\varepsilon
}\left\langle u\right\rangle _{z_{k},\overline{R_{T}^{N}}}^{(\sigma_{k}%
+\sigma_{k}\alpha)}. \label{3.32}%
\end{equation}

\end{proposition}

Presenting this Proposition without a proof, we nevertheless note the
following. First of all, to make use of Lemma \ref{L3.1} and representation
\eqref{3.25}, we need to extend the function $u(x,t)$ to a finite in $t$
function from the domain $\overline{R_{T}^{N}}$ to the domain $t>T$ that is to
whole subspace $\overline{R_{\infty}^{N}}=R^{N}\times\lbrack0,\infty)$. This
can be done with a controlled preservation of the norm in $C^{\overline
{\sigma}(1+\alpha),\theta+\theta\alpha}(\overline{R_{T}^{N}})$ in the way,
which is described in, for example, \cite{Sol1}, \cite{LadParab}. In the
second place, extended further by zero in the domain $t<0$ function $u(x,t)$
preserves it's norm in $C^{\overline{\sigma}(1+\alpha),\theta+\theta\alpha
}(R^{N}\times R^{1})$. It is precisely this place, where we need all the
derivatives of $u(x,t)$ in $t$ up to the order $[\theta+\theta\alpha]$
(possibly bigger than $[\theta]$ as it is required in Lemma \ref{L3.1}) to
vanish at $t=0$. After this the proof of the above Proposition coincides with
the proof of Proposition \ref{P3.5} with the making use of the representation
from Lemma \ref{L3.1} for $D_{\ast t}^{\theta}u$ (for obtaining the smoothness
in the space variables) and the proof of Proposition \ref{P3.1} (the
smoothness in $t$). The only difference is that the interval of integration
$(0,\infty)$ in representation \eqref{3.25} for $D_{\ast t}^{\theta}u$ is
naturally split not into a centered at zero ball and it's outer part (as in
the proof of Proposition \ref{P3.5}), but into the intervals $(0,\varepsilon
|h|^{\frac{\sigma_{k}}{\theta}})$ and $(\varepsilon|h|^{\frac{\sigma_{k}%
}{\theta}},\infty)$.

\section{Some additional technical assertions on the properties of fractional
differentiation in anisotropic H\"{o}lder spaces.}

\label{sDOP}

In the present section we prove several useful for applications assertions of
the interpolation type on the properties of fractional differentiation in
anisotropic H\"{o}lder spaces.

\begin{lemma}
\label{L3.3} Let a function $u(x,t)$ is defined in $\overline{R_{T}^{N}}$ and
possesses bounded in $\overline{R_{T}^{N}}$ fractional derivative $D_{\ast
t}^{\mu}u(x,t)$ of a noninteger order $\mu>0$. Let also this derivative
possesses the smoothness of a noninteger order $\beta>0$ with respect to a
group $z_{k}\in R^{N_{k}}$ of the space variables that is
\begin{equation}
\left\langle D_{\ast t}^{\mu}u\right\rangle _{z_{k},\overline{R_{T}^{N}}%
}^{(\beta)}=\sup_{(x,t)\in\overline{R_{T}^{N}},\overline{h}\in R^{N_{k}%
},\overline{h}\neq0}\frac{|\delta_{\overline{h},z_{k}}^{m}D_{\ast t}^{\mu
}u(x,t)|}{|\overline{h}|^{\beta}}<\infty,\quad m>\beta. \label{3.54.1}%
\end{equation}
Here $\delta_{\overline{h},z_{k}}^{m}u(x,t)$ is the finite difference of order
$m$ from the function $D_{\ast t}^{\mu}u(x,t)$ in the variables $z_{k}$ with
the step $\overline{h}\in R^{N_{k}}$.

Then
\begin{equation}
\left\vert \delta_{\tau,t}^{p}\delta_{\overline{h},z_{k}}^{m}u(x,t)\right\vert
\leq C\left\langle D_{\ast t}^{\mu}u\right\rangle _{z_{k},\overline{R_{T}^{N}%
}}^{(\beta)}\tau^{\mu}|\overline{h}|^{\beta},\quad p\geq\lbrack\mu]+1, m >
\beta, \label{3.55}%
\end{equation}
where $\delta_{\tau,t}^{p}v(x,t)$ is the finite difference in $t$ of step
$\tau>0$ from the function $v(x,t)$.

If a function $u(x,t)$ possesses bounded in $\overline{R_{T}^{N}}$ derivative
$u_{t}^{(n)}(x,t)$ in $t$ of an integer order $n>0$, which possesses the
smoothness in a space group $z_{k}\in R^{N_{k}}$ of a noninteger order
$\beta>0$ that is
\begin{equation}
\left\langle u_{t}^{(n)}\right\rangle _{z_{k},\overline{R_{T}^{N}}}^{(\beta
)}=\sup_{(x,t)\in\overline{R_{T}^{N}},\overline{h}\in R^{N_{k}},\overline
{h}\neq0}\frac{|\delta_{\overline{h},z_{k}}^{m}u_{t}^{(n)}(x,t)|}%
{|\overline{h}|^{\beta}}<\infty,\quad m>\beta, \label{3.55.1}%
\end{equation}
then also
\begin{equation}
\left\vert \delta_{\tau,t}^{p}\delta_{\overline{h},z_{k}}^{m}u(x,t)\right\vert
\leq C\left\langle u_{t}^{(n)}\right\rangle _{z_{k},\overline{R_{T}^{N}}%
}^{(\beta)}\tau^{n}|\overline{h}|^{\beta},\quad m>\beta,\,p\geq n.
\label{3.55.2}%
\end{equation}

\end{lemma}

\begin{proof}
First of all, the case of an integer order $n>0$ follows directly from the
mean value theorem and from condition \eqref{3.55.1} - see formulas
\eqref{3.64.1}, \eqref{3.65} below. Therefore we consider only the case of a
noninteger derivative.

Further, we can assume without loss of generality that the function $u(x,t)$
possesses the property
\begin{equation}
\frac{\partial^{k}u(x,0)}{\partial t^{k}}\equiv0,\quad k=0,...,[\mu
].\label{3.56}%
\end{equation}
In the opposite case we can replace $u(x,t)$ with the function
\[
\overline{u}(x,t)\equiv u(x,t)-{\sum\limits_{k=0}^{[\mu]}}\frac{t^{k}}%
{k!}\frac{\partial^{k}u(x,0)}{\partial t^{k}}%
\]
since $\delta_{\tau,t}^{[\mu]+1}\delta_{\overline{h},z_{k}}^{m}u(x,t)=\delta
_{\tau,t}^{[\mu]+1}\delta_{\overline{h},z_{k}}^{m}\overline{u}(x,t)$ and
$D_{\ast t}^{\mu}u=D_{\ast t}^{\mu}\overline{u}$. But for a function $u(x,t)$
with \eqref{3.56} the following representation is valid (see \cite{Umarov},
section 3.5)
\begin{equation}
D_{\ast t}^{\mu}u(x,t)=D_{\ast t}^{\{\mu\}}\left[  u_{t}^{([\mu])}%
(x,t)\right]  ,\label{3.57}%
\end{equation}%
\[
u(x,t)=J_{\{\mu\}}\left\{  D_{\ast t}^{\{\mu\}}\left[  u_{t}^{([\mu
])}(x,t)\right]  \right\}  .
\]
Here $J_{\mu}$ is the operator of fractional integration of order $\mu$ that
is
\begin{equation}
u(x,t)=\frac{1}{\Gamma(\mu)}{\int\limits_{0}^{t}}(t-\tau)^{-1+\mu}D_{\ast\tau
}^{\mu}u(x,\tau)d\tau=\label{3.58}%
\end{equation}%
\[
=\frac{1}{\Gamma(\{\mu\})}{\int\limits_{0}^{t}}(t-\tau)^{-1+\{\mu\}}%
D_{\ast\tau}^{\mu}u_{\tau}^{([\mu])}(x,\tau)d\tau.
\]
Note also that we can assume $p=[\mu]+1$.

Let first $\mu\in(0,1)$ that is $\mu=\{\mu\}$, $[\mu]=0$. Then with $m>\beta$
\begin{equation}
\delta_{\tau,t}\delta_{\overline{h},z_{k}}^{m}u(x,t)=\frac{1}{\Gamma(\mu
)}\delta_{\tau,t}\left\{  {\int\limits_{0}^{t}}(t-\omega)^{-1+\mu}\left[
\delta_{\overline{h},z_{k}}^{m}D_{\ast\omega}^{\mu}u(x,\omega)\right]
d\omega\right\}  .\label{3.59}%
\end{equation}
Denote
\begin{equation}
v(x,t)=\frac{1}{\Gamma(\mu)}\left[  \delta_{\overline{h},z_{k}}^{m}D_{\ast
\tau}^{\mu}u(x,t)\right]  ,\label{3.60}%
\end{equation}
and remark that according to \eqref{3.54.1},
\begin{equation}
\left\vert v(x,t)\right\vert \leq C(\mu)\left\langle D_{\ast t}^{\mu
}u\right\rangle _{z_{k},\overline{R_{T}^{N}}}^{(\beta)}|\overline{h}|^{\beta
}.\label{3.61}%
\end{equation}
Assuming that $t>2\tau$, represent the difference $\delta_{\tau,t}%
\delta_{\overline{h},z_{k}}^{m}u(x,t)$ in the form
\[
\delta_{\tau,t}\delta_{\overline{h},z_{k}}^{m}u(x,t)={\int\limits_{t-2\tau
}^{t+\tau}}(t+\tau-\omega)^{-1+\mu}v(x,\omega)d\omega-
\]%
\[
-{\int\limits_{t-2\tau}^{t}}(t-\omega)^{-1+\mu}v(x,\omega)d\omega+
\]%
\begin{equation}
+{\int\limits_{0}^{t-2\tau}}\left[  (t+\tau-\omega)^{-1+\mu}-(t-\omega
)^{-1+\mu}\right]  v(x,\omega)d\omega\equiv I_{1}+I_{2}+I_{3}.\label{3.62}%
\end{equation}
Taking into account \eqref{3.61} we have for $I_{1}$
\[
|I_{1}|\leq C(\mu)\left\langle D_{\ast t}^{\mu}u\right\rangle _{z_{k}%
,\overline{R_{T}^{N}}}^{(\beta)}|\overline{h}|^{\beta}{\int\limits_{t-2\tau
}^{t+\tau}}(t+\tau-\omega)^{-1+\mu}d\omega=
\]%
\[
=C(\mu)\left\langle D_{\ast t}^{\mu}u\right\rangle _{z_{k},\overline{R_{T}%
^{N}}}^{(\beta)}|\overline{h}|^{\beta}\tau^{\mu}.
\]
The integral $I_{2}$ is estimated analogously, which gives
\begin{equation}
|I_{1}|+|I_{2}|\leq C(\mu) \left\langle D_{\ast t}^{\mu}u\right\rangle
_{z_{k},\overline{R_{T}^{N}}}^{(\beta)}|\overline{h}|^{\beta}\tau^{\mu
}.\label{3.63}%
\end{equation}
To estimate $I_{3}$ we make use of the mean value theorem
\[
I_{3}=C(\mu)\tau{\int\limits_{0}^{t-2\tau}}(t+\varkappa(t-\omega)\tau
-\omega)^{-2+\mu}v(x,\omega)d\omega,\quad\varkappa(t-\omega)\in(0,1),
\]
and note that since $\omega<t-2\tau$ and consequently $\tau<(t-\omega)/2$,
then
\[
|t-\omega|\leq|t+\varkappa(t-\omega)\tau-\omega|\leq\frac{3}{2}|t-\omega|.
\]
Therefore
\[
|I_{2}|\leq C \left\langle D_{\ast t}^{\mu}u\right\rangle _{z_{k}%
,\overline{R_{T}^{N}}}^{(\beta)}|\overline{h}|^{\beta}\tau{\int\limits_{0}%
^{t-2\tau}}(t-\omega)^{-2+\mu}d\omega\leq
\]%
\begin{equation}
\leq C \left\langle D_{\ast t}^{\mu}u\right\rangle _{z_{k},\overline
{R_{T}^{N}}}^{(\beta)}|\overline{h}|^{\beta}\tau{\int\limits_{-\infty
}^{t-2\tau}}(t-\omega)^{-2+\mu}d\omega= C \left\langle D_{\ast t}^{\mu
}u\right\rangle _{z_{k},\overline{R_{T}^{N}}}^{(\beta)}|\overline{h}|^{\beta
}\tau^{\mu}.\label{3.64}%
\end{equation}
If now $t<2\tau$, then we can represent the difference in the form
$\delta_{\tau,t}\delta_{\overline{h},z_{k}}^{m}u(x,t)=\delta_{\overline
{h},z_{k}}^{m}u(x,t+\tau)-\delta_{\overline{h},z_{k}}^{m}u(x,t)$, estimate
modulo each term of the difference separately analogously to the estimates of
the integrals $I_{1}$ and $I_{2}$ in \eqref{3.63}, and, taking into account
\eqref{3.63}, \eqref{3.64}, arrive at estimate \eqref{3.55} in the case
$\mu\in(0,1)$.

Let now $\mu>1$. We make use of the integral mean value theorem to represent
the difference in $t$ of order $[\mu]$ as follows
\[
\delta_{\tau,t}^{[\mu]}u(x,t)=\tau^{\lbrack\mu]}{\int\limits_{0}^{1}}d\xi
_{1}...{\int\limits_{0}^{1}}d\xi_{\lbrack\mu]}u_{t}^{[\mu]}(x,t+\xi_{1}%
\tau+...\xi_{1}\tau)=
\]%
\begin{equation}
=\tau^{\lbrack\mu]}{\int\limits_{P_{[\mu]}}}u_{t}^{[\mu]}\left(  x,t+\tau
{\sum\limits_{i=1}^{[\mu]}}\xi_{i}\right)  d\xi,\label{3.64.1}%
\end{equation}
where $\xi=(\xi_{1},...,\xi_{\lbrack\mu]})\in P_{[\mu]}=\{\xi:0<\xi_{i}<1\}$.
Then the double difference $\delta_{\tau,t}^{[\mu]+1}\delta_{\overline
{h},z_{k}}^{m}u(x,t)$ is equal to
\begin{equation}
\delta_{\tau,t}^{[\mu]+1}\delta_{\overline{h},z_{k}}^{m}u(x,t)=\delta_{\tau
,t}\delta_{\overline{h},z_{k}}^{m}u(x,t)=\tau^{\lbrack\mu]}{\int
\limits_{P_{[\mu]}}}\delta_{\tau,t}\delta_{\overline{h},z_{k}}^{m}u_{t}%
^{[\mu]}\left(  x,t+\tau{\sum\limits_{i=1}^{[\mu]}}\xi_{i}\right)
d\xi.\label{3.65}%
\end{equation}
At the same time, according to the conditions of the lemma, the function
$u(x,t)$ possesses the fractional derivative $D_{\ast t}^{\mu}u(x,t)$, which
satisfies \eqref{3.54.1}. Therefore at a fixed $\tau$ and $\xi$ the function
$v(x,t)=u_{t}^{[\mu]}\left(  x,t+\tau{\sum\limits_{i=1}^{[\mu]}}\xi
_{i}\right)  $ admits, by virtue of \eqref{3.57}, the fractional derivative
$D_{\ast t}^{\{\mu\}}v(x,t)$ with \eqref{3.54.1}. Therefore by what was proved
above
\[
\left\vert \delta_{\tau,t}^{[\mu]+1}\delta_{\overline{h},z_{k}}^{m}%
u(x,t)\right\vert \leq\tau^{\lbrack\mu]}{\int\limits_{P_{[\mu]}}}\left\vert
\delta_{\tau,t}\delta_{\overline{h},z_{k}}^{m}u_{t}^{[\mu]}\left(
x,t+\tau{\sum\limits_{i=1}^{[\mu]}}\xi_{i}\right)  \right\vert d\xi\leq
\]%
\[
\leq C\tau^{\lbrack\mu]}\left\langle D_{\ast t}^{\mu}u\right\rangle
_{z_{k},\overline{R_{T}^{N}}}^{(\beta)}|\overline{h}|^{\beta}\tau^{\{\mu
\}}=C\left\langle D_{\ast t}^{\mu}u\right\rangle _{z_{k},\overline{R_{T}^{N}}%
}^{(\beta)}|\overline{h}|^{\beta}\tau^{\mu},
\]
which finishes the proof of the lemma.
\end{proof}

\begin{theorem}
\label{T3.3} Let a function $u(x,t)$ is defined in $\overline{R_{T}^{N}}$ and
possesses bounded in $\overline{R_{T}^{N}}$ derivative $D_{t}^{\theta}u(x,t)$
of an integer or a noninteger order $\theta>0$ and in the case of a noninteger
order we mean the Caputo - Jrbashyan derivative $D_{*t}^{\theta}u(x,t)$.
Suppose that the derivative $D_{t}^{\theta}u(x,t)$ possesses the smoothness of
a noninteger order $\beta>0$ with respect to a spacial group $z_{k}\in
R^{N_{k}}$ that is
\begin{equation}
\left\langle D_{t}^{\theta}u\right\rangle _{z_{k},\overline{R_{T}^{N}}%
}^{(\beta)}=\sup_{(x,t)\in\overline{R_{T}^{N}},\overline{h}\in R^{N_{k}%
},\overline{h}\neq0}\frac{|\delta_{\overline{h},z_{k}}^{m}D_{\ast t}^{\theta
}u(x,t)|}{|\overline{h}|^{\beta}}<\infty,\quad m>\beta. \label{3.66}%
\end{equation}
Besides, suppose that the function $u(x,t)$ itself possesses the smoothness of
a noninteger order $\gamma>\max\{1,\beta\}$ with respect to the same variables
$z_{k}\in R^{N_{k}}$ that is
\begin{equation}
\left\langle u\right\rangle _{z_{k},\overline{R_{T}^{N}}}^{(\gamma)}%
=\sup_{(x,t)\in\overline{R_{T}^{N}},\overline{h}\in R^{N_{k}},\overline{h}%
\neq0}\frac{|\delta_{\overline{h},z_{k}}^{m}u(x,t)|}{|\overline{h}|^{\gamma}%
}<\infty,\quad m>\gamma. \label{3.67}%
\end{equation}

Then a derivative of $u(x,t)$ of an integer order $\beta<n<\gamma$ with
respect to $z_{k}\in R^{N_{k}}$
\[
D_{z_{k}}^{\overline{\rho}}u(x,t)=\frac{\partial^{\rho_{1}}...\partial
^{\rho_{N_{k}}}u(x,t)}{\partial x_{i_{1}}^{\rho_{1}}...\partial x_{i_{N_{k}}%
}^{\rho_{N_{k}}}},\,\overline{\rho}=(\rho_{1},...,\rho_{N_{k}}),\,\rho
=\rho_{1}+...+\rho_{N_{k}}=n,
\]
possesses the smoothness with respect to $t$ of order $\theta(\gamma
-n)/(\gamma-\beta)$ that is
\[
\left\langle D_{z_{k}}^{\overline{\rho}}u\right\rangle _{t,\overline{R_{T}%
^{N}}}^{(\theta(\gamma-n)/(\gamma-\beta))}=\sup_{(x,t)\in\overline{R_{T}^{N}%
},\tau>0}\frac{|\delta_{\tau,t}^{m}D_{z_{k}}^{\overline{\rho}}u(x,t)|}%
{\tau^{\theta(\gamma-n)/(\gamma-\beta)}}<\infty,\quad m>\theta(\gamma
-n)/(\gamma-\beta).
\]
Moreover,
\begin{equation}
\left\langle D_{z_{k}}^{\overline{\rho}}u\right\rangle _{t,\overline{R_{T}%
^{N}}}^{(\theta(\gamma-n)/(\gamma-\beta))}\leq C(\theta,\beta,\gamma,n)\left(
\varepsilon\left\langle u\right\rangle _{z_{k},\overline{R_{T}^{N}}}%
^{(\gamma)}+\frac{1}{\varepsilon^{c}}\left\langle D_{t}^{\theta}u\right\rangle
_{z_{k},\overline{R_{T}^{N}}}^{(\beta)}\right)  , \label{3.68}%
\end{equation}
where $\varepsilon>0$ is arbitrary.

Note that in the case $n<\beta$ the derivative $D_{z_{k}}^{\overline{\rho}%
}u(x,t)$ just admits the derivative $D_{t}^{\theta}D_{z_{k}}^{\overline{\rho}%
}u(x,t)$ in $t$ as it follows from the first condition of the theorem.
\end{theorem}

\begin{proof}
Let first $\beta\in(0,1)$. We prove now the assertion of the theorem for $n=1$
that is for the derivative of the first order with respect to a single spacial
variable $x_{l}$ from the group $z_{k}$. For that we use the schema of
reasonings from \cite{20}, section 3, Lemma 2, when obtaining formula
\eqref{2.6} there. Let for the simplicity of notations $l=1$ that is we
consider the derivative $u_{x_{1}}(x,t)$. Consider the following
representation for the finite difference from $u_{x_{1}}(x,t)$ of order
$m>\gamma$ with a step $\xi>0$ with respect to $x_{1}$
\[
\delta_{\xi,x_{1}}^{m}u_{x_{1}}(x,t)={\sum\limits_{j=0}^{m}}(-1)^{m-j}%
C_{m}^{j}u_{x_{1}}(x_{1}+j\xi,x_{2},...x_{N},t).
\]
Expressing from this the term without a shift $u_{x_{1}}(x,t)$ over the others
terms of the identity and integrating over $\xi$ in the range from zero to
$\varepsilon\tau^{a}$, $\tau>0$, $\varepsilon>0$, $a=\theta/(\gamma-\beta)$,
we obtain
\[
u_{x_{1}}(x,t)=\frac{1}{\varepsilon\tau^{a}}{\sum\limits_{j=1}^{m}}%
(-1)^{j+1}\frac{C_{m}^{j}}{j}{\int\limits_{0}^{\varepsilon\tau^{a}}}u_{\xi
}(x_{1}+j\xi,x_{2},...x_{N},t)d\xi+
\]%
\[
+\frac{(-1)^{m}}{\varepsilon\tau^{a}}{\int\limits_{0}^{\varepsilon\tau^{a}}%
}\delta_{\xi,x_{1}}^{m}u_{x_{1}}(x,t)d\xi=
\]%
\[
={\sum\limits_{j=1}^{m}}\frac{1}{\varepsilon\tau^{a}}B_{m,j}\left[
u(x_{1}+j\varepsilon\tau^{a},x_{2},...x_{N},t)-u(x_{1},x_{2},...x_{N}%
,t)\right]  +
\]%
\begin{equation}
+\frac{(-1)^{m}}{\varepsilon\tau^{a}}{\int\limits_{0}^{\varepsilon\tau^{a}}%
}\delta_{\xi,x_{1}}^{m}u_{x_{1}}(x,t)d\xi\equiv{\sum\limits_{j=1}^{m}}%
I_{j}+I_{0}.\label{3.69}%
\end{equation}
To estimate the smoothness of $u_{x_{1}}(x,t)$ in $t$ we estimate a finite
difference along this variable of a sufficiently high order $\delta_{\tau
,t}^{p}u_{x_{1}}(x,t)$, $p>\theta$. And for this we estimate the finite
differences of the terms $I_{j}$ and $I_{0}$ in \eqref{3.69}. For the terms
$I_{j}$ we have
\[
\left\vert \delta_{\tau,t}^{p}I_{j}\right\vert =\frac{C}{\varepsilon\tau^{a}%
}\left\vert \delta_{\tau,t}^{p}\delta_{j\varepsilon\tau^{a},x_{1}%
}u(x,t)\right\vert \leq\frac{C}{\varepsilon\tau^{a}}\left\langle D_{t}%
^{\theta}u\right\rangle _{z_{k},\overline{R_{T}^{N}}}^{(\beta)}\tau^{\theta
}(\varepsilon\tau^{a})^{\beta},
\]
where we made use of Lemma \ref{L3.3}. Thus, since $a=\theta/(\gamma-\beta)$,
\begin{equation}
\left\vert \delta_{\tau,t}^{p}I_{j}\right\vert \leq\frac{C}{\varepsilon
^{1-\beta}}\left\langle D_{t}^{\theta}u\right\rangle _{z_{k},\overline
{R_{T}^{N}}}^{(\beta)}\tau^{\theta(\gamma-1)/(\gamma-\beta)}.\label{3.70}%
\end{equation}
Further, since for each fixed $t>0$ the function $u_{x_{1}}(x,t)$ has the
order of smoothness $\gamma-1$ in $x_{1}$, then
\[
|I_{0}|\leq\frac{C}{\varepsilon\tau^{a}}\left\langle u_{x_{1}}\right\rangle
_{x_{1},\overline{R_{T}^{N}}}^{(\gamma-1)}{\int\limits_{0}^{\varepsilon
\tau^{a}}}\xi^{\gamma-1}d\xi\leq\frac{C}{\varepsilon\tau^{a}}\left\langle
u\right\rangle _{z_{k},\overline{R_{T}^{N}}}^{(\gamma)}(\varepsilon\tau
^{a})^{\gamma}=
\]%
\[
=C\varepsilon^{\gamma-1}\left\langle u\right\rangle _{z_{k},\overline
{R_{T}^{N}}}^{(\gamma)}\tau^{\theta(\gamma-1)/(\gamma-\beta)}.
\]
Consequently,
\begin{equation}
\left\vert \delta_{\tau,t}^{p}I_{0}\right\vert \leq C(p)\varepsilon^{\gamma
-1}\left\langle u\right\rangle _{z_{k},\overline{R_{T}^{N}}}^{(\gamma)}%
\tau^{\theta(\gamma-1)/(\gamma-\beta)}.\label{3.71}%
\end{equation}
The assertion of the theorem for $n=1$ under the condition $\beta\in(0,1)$
follows now from \eqref{3.70} and \eqref{3.71} with taking into account
\eqref{3.69}. Thus, the derivatives of the first order in $x_{i}$ from $z_{k}$
belongs, under fixed spacial variables from other groups, to the space
$C_{z_{k},\,\quad t}^{\gamma-1,\theta(\gamma-1)/(\gamma-\beta)}(R_{T}^{N_{k}%
})$ and estimate \eqref{3.68} is valid for $n=1$. But since the derivatives in
$z_{k}$ of order $n>1$ are the derivatives of order $n-1$ of the first
derivatives, then estimate \eqref{3.68} for $n>1$ follows now directly from
\eqref{2.6}. Consequently, the theorem is proved for $\beta\in(0,1)$.

If $\beta>1$, then it is enough instead of the function $u(x,t)$ itself to
consider it's derivative $D_{z_{k}}^{\overline{\rho}}u(x,t)$ of order
$\rho=[\beta]$ and denote this derivative by $v(x,t)=D_{z_{k}}^{\rho}u(x,t)$.
This function fully meets the conditions of the theorem with the same $\theta
$, with $\{\beta\}\in(0,1)$ instead of $\beta$, and with $\gamma-[\beta]$
instead of $\gamma$. The application of what was proved above to this function
$v(x,t)$ with the subsequent recalculation of the exponents of smoothness
leads to the proof of the theorem in the general case.
\end{proof}

The analogous assertion is valid and in the case, when instead of the usual
derivatives in a spacial group $z_{k}$ one considers the fractional Laplace
operator with respect to $z_{k}$.

\begin{theorem}
\label{T3.4} Let a function $u(x,t)$ be defined in $\overline{R_{T}^{N}}$ and
let it's derivative $D_{t}^{\theta}u(x,t)$ of order $\theta>0$ be bounded in
$\overline{R_{T}^{N}}$, where $\theta>0$ can be either an integer or a
noninteger (in the case of a noninteger $\theta>0$ we mean the Caputo -
Jrbashyan derivative). Suppose that the derivative $D_{t}^{\theta}u(x,t)$
possesses the smoothness in $z_{k}\in R^{N_{k}}$ of a noninteger order
$\beta>0$ in the sense of \eqref{3.66}. Suppose also that the function
$u(x,t)$ itself possesses the smoothness in $z_{k}\in R^{N_{k}}$ of a
noninteger order $\gamma>\beta$ in the sense of \eqref{3.67}. Then the
fractional Laplace operator of this function $(-\Delta_{z_{k}} )^{\frac{\mu
}{2}}u(x,t)$ in $z_{k}\in R^{N_{k}}$ of order $\mu\in(\beta,\gamma)$ possesses
the smoothness in $t$ of order $\theta(\gamma-\mu)/(\gamma-\beta)$ that is
\[
\left\langle (-\Delta_{z_{k}})^{^{\frac{\mu}{2}}}u\right\rangle _{t,\overline
{R_{T}^{N}}}^{(\theta(\gamma-\mu)/(\gamma-\beta))}=\sup_{(x,t)\in
\overline{R_{T}^{N}},\tau>0}\frac{|\delta_{\tau,t}^{m}(-\Delta_{z_{k}%
})^{^{\frac{\mu}{2}}}u(x,t)|}{\tau^{\theta(\gamma-\mu)/(\gamma-\beta)}}%
<\infty, m>\frac{\theta(\gamma-\mu)}{(\gamma-\beta)}.
\]
Moreover,
\[
\left\langle (-\Delta_{z_{k}})^{^{\frac{\mu}{2}}}u\right\rangle _{t,\overline
{R_{T}^{N}}}^{(\theta(\gamma-\mu)/(\gamma-\beta))} \leq
\]
\begin{equation}
\leq C(\theta,\beta,\gamma,\mu)\left(  \varepsilon^{\gamma-\mu}\left\langle
u\right\rangle _{z_{k},\overline{R_{T}^{N}}}^{(\gamma)}+\frac{1}%
{\varepsilon^{\left(  \mu-\beta\right)  }}\left\langle D_{t}^{\theta
}u\right\rangle _{z_{k},\overline{R_{T}^{N}}}^{(\beta)}\right)  , \label{3.72}%
\end{equation}
where $\varepsilon>0$ is arbitrary.

Note that for $\mu<\beta$ the Laplace operator $(-\Delta_{z_{k}}
)^{^{\frac{\mu}{2}}}u(x,t)$ just has the derivative $D_{t}^{\theta}%
(-\Delta_{z_{k}})^{\mu}u(x,t)$, as it follows from the first assumption of the theorem.
\end{theorem}

\begin{proof}
We use representation \eqref{1.9}
\begin{equation}
(-\Delta_{z_{k}})^{^{\frac{\mu}{2}}}u(x,t)=C_{N_{k},\mu,m}{\int\limits_{\eta
_{k}\in R^{N_{k}}}}\frac{\delta_{\eta_{k},z_{k}}^{m}u(x,t)}{|\eta_{k}%
|^{N_{k}+\mu}}d\eta_{k},\label{3.73}%
\end{equation}
where we choose $m>\gamma$. Consider, as in the previous theorem, the
difference in $t$ of an order $p>\theta$ with a step $\tau$
\begin{equation}
D(\tau)\equiv\delta_{\tau,t}^{p}(-\Delta_{z_{k}})^{^{\frac{\mu}{2}}%
}u(x,t)=C_{N_{k},\mu,m}{\int\limits_{\eta_{k}\in R^{N_{k}}}}\frac{\delta
_{\tau,t}^{p}\delta_{\eta_{k},z_{k}}^{m}u(x,t)}{|\eta_{k}|^{N_{k}+\mu}}%
d\eta_{k}.\label{3.74}%
\end{equation}
Split the integral in \eqref{3.74} into the two ones
\begin{equation}
D(\tau)=C_{N_{k},\mu,m}{\int\limits_{\left\vert \eta_{k}\right\vert
\leq\varepsilon\tau^{a}}}\frac{\delta_{\tau,t}^{p}\delta_{\eta_{k},z_{k}}%
^{m}u(x,t)}{|\eta_{k}|^{N_{k}+\mu}}d\eta_{k}+\label{3.75}%
\end{equation}%
\[
+C_{N_{k},\mu,m}\int\limits_{\left\vert \eta_{k}\right\vert >\varepsilon
\tau^{a}}\frac{\delta_{\tau,t}^{p}\delta_{\eta_{k},z_{k}}^{m}u(x,t)}{|\eta
_{k}|^{N_{k}+\mu}}d\eta_{k}\equiv I_{1}+I_{2},
\]
where $\varepsilon>0$, $a=\theta/(\gamma-\beta)$. Bearing in mind
\eqref{3.67}, we can obtain for $I_{1}$ the estimate
\[
\left\vert I_{1}\right\vert \leq C_{N_{k},\mu,m}\left\langle u\right\rangle
_{z_{k},\overline{R_{T}^{N}}}^{(\gamma)}{\int\limits_{\left\vert \eta
_{k}\right\vert \leq\varepsilon\tau^{a}}}\frac{|\eta_{k}|^{\gamma}}{|\eta
_{k}|^{N_{k}+\mu}}d\eta_{k}=
\]%
\begin{equation}
=C_{N_{k},\mu,m}\left\langle u\right\rangle _{z_{k},\overline{R_{T}^{N}}%
}^{(\gamma)}\left(  \varepsilon\tau^{a}\right)  ^{\gamma-\mu}=C_{N_{k},\mu
,m}\varepsilon^{\gamma-\mu}\left\langle u\right\rangle _{z_{k},\overline
{R_{T}^{N}}}^{(\gamma)}\tau^{\theta(\gamma-\mu)/(\gamma-\beta)}.\label{3.76}%
\end{equation}
To estimate $I_{2}$ we make use of \eqref{3.66} together with Lemma
\ref{L3.3}, which gives
\[
\left\vert I_{2}\right\vert \leq C_{N_{k},\mu,m}\left\langle D_{t}^{\theta
}u\right\rangle _{z_{k},\overline{R_{T}^{N}}}^{(\beta)}\tau^{\theta}%
{\int\limits_{\left\vert \eta_{k}\right\vert >\varepsilon\tau^{a}}}\frac
{|\eta_{k}|^{\beta}}{|\eta_{k}|^{N_{k}+\mu}}d\eta_{k}=
\]%
\[
=C_{N_{k},\mu,m}\left\langle D_{t}^{\theta}u\right\rangle _{z_{k}%
,\overline{R_{T}^{N}}}^{(\beta)}\tau^{\theta}\left(  \varepsilon\tau
^{a}\right)  ^{-\left(  \mu-\beta\right)  }=
\]%
\[
=C_{N_{k},\mu,m}\left\langle D_{t}^{\theta}u\right\rangle _{z_{k}%
,\overline{R_{T}^{N}}}^{(\beta)}\varepsilon^{-\left(  \mu-\beta\right)  }%
\tau^{\theta(\gamma-\mu)/(\gamma-\beta)}.
\]
The assertion of the theorem follows now from \eqref{3.74} - \eqref{3.76} in
view of \eqref{2.4}.
\end{proof}

\section{Theorems on Fourier multipliers in H\"{o}lder spaces}

\label{s4}

In this section we present some theorems from \cite{LadMult} and \cite{23} on
the Fourier multipliers in H\"{o}lder spaces. We consider the multipliers,
that act either in spaces with finite H\"{o}lder seminorm with respect to all
independent variables or in spaces with finite H\"{o}lder seminorm with
respect to a part of independent variables. These theorems will be applied
further for the proofs of theorems \ref{T2.1} - \ref{T2.3}.

We follow \cite{23} to give some necessary definitions.

Let $K>0$ be an integer and let
\begin{equation}
\gamma\in(0,1),\quad\beta=(\beta_{1},...,\beta_{K}),\quad\beta_{1}%
=1,\quad\beta_{i}\in(0,1],i=2,...,K.\label{4.1}%
\end{equation}
Consider the H\"{o}lder space $C^{\gamma\beta}(R^{K})$ with the norm
\begin{equation}
\left\Vert u\right\Vert _{C^{\gamma\beta}(R^{K})}\equiv\left\vert u\right\vert
_{R^{K}}^{(\gamma\beta)}\equiv|u|_{R^{K}}^{(0)}+{\sum\limits_{i=1}^{K}%
}\left\langle u\right\rangle _{x_{i},R^{K}}^{(\gamma\beta_{i})}.\label{4.2}%
\end{equation}
Along with this space we consider the more narrow space $H^{\gamma\beta}%
(R^{K})$ with the norm
\begin{equation}
\left\Vert u\right\Vert _{H^{\gamma\beta}(R^{K})}\equiv\left\Vert u\right\Vert
_{L_{2}(R^{K})}+{\sum\limits_{i=1}^{K}}\left\langle u\right\rangle
_{x_{i},R^{K}}^{(\gamma\beta_{i})},\label{4.3}%
\end{equation}
and it was shown in \cite{LadMult} that
\begin{equation}
\left\vert u\right\vert _{R^{K}}^{(\gamma\beta)}\leq C (\gamma
,\beta)\left\Vert u\right\Vert _{H^{\gamma\beta}(R^{K})}.\label{4.4}%
\end{equation}
Let a measurable and bounded function $\widetilde{m}(\xi)$, $\xi\in R^{K}$, be
defined in $R^{K}$. Define the operator $M:H^{\gamma\beta}(R^{K})\rightarrow
L_{2}(R^{K})$ as follows
\begin{equation}
Mu\equiv F^{-1}\left[  \widetilde{m}(\xi)F(u)(\xi)\right]  \equiv
F^{-1}\left[  \widetilde{m}(\xi)\widetilde{u}(\xi)\right]  .\label{4.5}%
\end{equation}
Here $F(u)(\xi)\equiv\widetilde{u}(\xi)$ is the Fourier transform of $u(x)$
extended on the space $L_{2}(R^{K})$, $F^{-1}$ is the inverse Fourier
transform. Since $u(x)\in H^{\gamma\beta}(R^{K})\subset L_{2}(R^{K})$, and the
function $\widetilde{m}(\xi)$ is bounded, the operator $M$ is correctly
defined. We call the function $\widetilde{m}(\xi)$ a Fourier multiplier.

Let the whole set of the variables $(\xi_{1},...,\xi_{K})=\xi$ be split into
$r$ subsets of length $K_{i}$, $i=1,...,r$, $K=K_{1}+...+K_{r}$ so that
\begin{equation}
\xi=(y_{1},...,y_{r}),\quad y_{1}=(\xi_{1},...,\xi_{K_{1}}),...,y_{r}%
=(\xi_{K_{1}+...K_{r-1}+1},...,\xi_{K}).\label{4.6}%
\end{equation}
Let, further, $\omega_{i}$, $i=1,...,r$ be multi-indexes each of length
$K_{i}$
\begin{equation}
\omega_{1}=(\omega_{1,1},...,\omega_{1,K_{1}}),...,\omega_{r}=(\omega
_{r,1},...,\omega_{r,K_{r}}),\quad\omega_{i,j}\in\mathbf{N\cup\{0\}.}%
\label{4.7}%
\end{equation}
Denote by $D_{y_{i}}^{\omega_{i}}\widetilde{u}(\xi)$ the derivative of the
function $\widetilde{u}(\xi)$ in the group $y_{i}$ of order $|\omega
_{i}|=\omega_{i,1}+...+\omega_{i,K_{i}}$ that is $D_{y_{i}}^{\omega_{i}%
}\widetilde{u}(\xi)=D_{\xi_{j_{1}}}^{\omega_{i,1}}...D_{\xi_{j_{N_{i}}}%
}^{\omega_{i,N_{i}}}\widetilde{u}(\xi)$. Let also $p\in(1,2]$ and positive
numbers $s_{i}$, $i=1,...,r$, satisfy the conditions
\begin{equation}
s_{i}>\frac{N_{i}}{p},\quad i=1,...,r.\label{4.8}%
\end{equation}
Denote for $\nu>0$
\begin{equation}
B_{\nu}=\{\xi\in R^{K}:\nu\leq|\xi|\leq\nu^{-1}\}.\label{4.9}%
\end{equation}
Suppose that for some $\nu>0$ the function $\widetilde{m}(\xi)$ satisfies with
a certain $\mu>0$ and uniformly in $\lambda>0$ the condition
\begin{equation}
{\sum\limits_{|\omega_{i}|\leq s_{i}}}\left\Vert D_{y_{1}}^{\omega_{1}%
}D_{y_{2}}^{\omega_{2}}...D_{y_{r}}^{\omega_{r}}\widetilde{m}(\lambda
^{\frac{1}{\beta_{1}}}\xi_{1},...,\lambda^{\frac{1}{\beta_{K}}}\xi
_{K})\right\Vert _{L_{p}(B_{\nu})}\leq\mu,\label{4.10}%
\end{equation}
where $\beta_{i}$ are defined in \eqref{4.1}.

\begin{theorem}
\label{T4.1}(\cite{LadMult}: T.2.1, L.2.1, L.2.2, T.2.2, T.2.3)

If a function $\widetilde{m}(\xi)$ satisfies conditions \eqref{4.10}, then the
defined in \eqref{4.5} operator $M$, is a bounded linear operator from the
space $H^{\gamma\beta}(R^{K})$ to itself and
\begin{equation}
\left\Vert Mu\right\Vert _{H^{\gamma\beta}(R^{K})}\leq C(K,\gamma,\beta
,p,\nu,\{s_{i}\})\mu\left\Vert u\right\Vert _{H^{\gamma\beta}(R^{K}%
)},\label{4.11}%
\end{equation}%
\begin{equation}
{\sum\limits_{i=1}^{K}}\left\langle Mu\right\rangle _{x_{i},R^{K}}%
^{(\gamma\beta_{i})}\leq C(K,\gamma,\beta,p,\nu,\{s_{i}\})\mu{\sum
\limits_{i=1}^{K}}\left\langle u\right\rangle _{x_{i},R^{K}}^{(\gamma\beta
_{i})}.\label{4.12}%
\end{equation}

\end{theorem}

Condition \eqref{4.10} can be especially easily verified in the cases, when
the function $\widetilde{m}(\xi)$ has the homogeneity of degree zero that is
when $\widetilde{m}(\lambda^{\frac{1}{\beta_{1}}}\xi_{1},...,\lambda^{\frac
{1}{\beta_{K}}}\xi_{K})=\widetilde{m}(\xi)$. Note also that condition
\eqref{4.10} includes the derivatives of $\widetilde{m}(\xi)$ in $y_{i}$ only
up to the orders $s_{i}$. The case $r=1$, $K_{1}=K$, $p=2$ is considered in
Lemma 2.1 in \cite{LadMult} and Lemma 2.2 of the same paper contains the case
$r=K$, $K_{i}=1$, $s_{i}=1$. The general case is analogous - see lemmas 2.2 -
2.4 in \cite{23}.

Now we formulate a generalization of Theorem \ref{T4.1} to the case of Fourier
multipliers in the spaces of functions with the H\"{o}lder condition only with
respect to a part of the variables. For that we need to split the whole set of
variables $x\in R^{K}$ and the corresponding dual (in the sense of the Fourier
transform) set of variables $\xi\in R^{K}$, besides splitting \eqref{4.6} and
regardless this splitting, also as follows.

Let $x=(x^{(1)},x^{(2)})$, $x^{(1)}=(x_{1},...,x_{S})\in R^{S}$,
$x^{(2)}=(x_{S+1},...,x_{K})\in R^{K-S}$ and correspondingly $\xi=(\xi
^{(1)},\xi^{(2)})$, $\xi^{(1)}=(\xi_{1},...,\xi_{S})\in R^{S}$, $\xi
^{(2)}=(\xi_{S+1},...,\xi_{K})\in R^{K-S}$. Let, further, analogously to
\eqref{4.1},
\[
\gamma\in(0,1),\quad\beta=(\beta_{1},...,\beta_{S}),\quad\beta_{1}%
=1,\quad\beta_{i}\in(0,1],i=2,...,S,
\]%
\begin{equation}
\varkappa=(\varkappa_{S+1},...,\varkappa_{K}),\quad\varkappa_{i}>0,\quad
i=S+1,...,K.\label{4.13}%
\end{equation}
Note that $\varkappa_{i}$ not necessarily belongs to $(0,1]$. Analogously to
\eqref{4.3} define the space $H_{x^{(1)},x^{(2)}}^{\gamma\beta,\quad
\gamma\varkappa}(R^{K})=C_{x^{(1)},x^{(2)}}^{\gamma\beta,\quad\gamma\varkappa
}(R^{K})\cap L_{2}(R^{K})$ as the Banach space of functions with the finite
norm
\begin{equation}
\left\Vert u\right\Vert _{H_{x^{(1)},x^{(2)}}^{\gamma\beta,\quad
\gamma\varkappa}(R^{K})}\equiv\left\Vert u\right\Vert _{L_{2}(R^{K})}%
+{\sum\limits_{i=1}^{S}}\left\langle u\right\rangle _{x_{i},R^{K}}%
^{(\gamma\beta_{i})}+{\sum\limits_{i=S+1}^{K}}\left\langle u\right\rangle
_{x_{i},R^{K}}^{(\gamma\varkappa_{i})}\label{4.14}%
\end{equation}
and analogously to \eqref{4.4}
\begin{equation}
\left\vert u\right\vert _{x^{(1)},x^{(2)},R^{K}}^{(\gamma\beta,\gamma
\varkappa)}\leq C (\gamma,\beta)\left\Vert u\right\Vert _{H_{x^{(1)}%
,x^{(2)}}^{\gamma\beta,\quad\gamma\varkappa}(R^{K})}.\label{4.15}%
\end{equation}
Besides, define the Banach space $H_{x^{(1)}}^{\gamma\beta}(R^{K})\supset
C_{x^{(1)}}^{\gamma\beta}(R^{K})\cap L_{2}(R^{K})$ with the finite norm
\begin{equation}
\left\Vert u\right\Vert _{H_{x^{(1)}}^{\gamma\beta}(R^{K})}\equiv\left\Vert
u\right\Vert _{L_{2}(R^{K})}+{\sum\limits_{i=1}^{S}}\left\langle
u\right\rangle _{x_{i},R^{K}}^{(\gamma\beta_{i})},\label{4.17}%
\end{equation}
and we stress that the functions from $H_{x^{(1)}}^{\gamma\beta}(R^{K})$ have
bounded H\"{o}lder seminorms with respect to the variables from the group
$x^{(1)}$ only. In particular, the functions from this space are not
necessarily bounded - see an example before Theorem 2.7 in \cite{23}.

\begin{theorem}
\label{T4.2}(\cite{23}: T.2.7) Let a function
\begin{equation}
\widetilde{m}(\xi)\in C(R^{K}\backslash\{0\})\text{ be continuous and bounded
in }R^{K}\backslash\{0\}\label{4.18}%
\end{equation}
and let it satisfy the condition
\begin{equation}
\widetilde{m}(\xi)|_{\xi^{(1)}=0}=\widetilde{m}(0,\xi^{(2)})\equiv0,\quad
\xi^{(2)}\in R^{K-S}\backslash\{0\}.\label{4.19}%
\end{equation}
Let, further, $p\in(1,2]$ and let positive numbers $s_{i}$, $i=1,...,r$,
satisfy the condition (comp. \eqref{4.8})
\begin{equation}
s_{i}>\frac{N_{i}}{p}+\gamma,\quad i=1,...,r.\label{4.20}%
\end{equation}
Let, besides, the following condition be satisfied (comp. \eqref{4.10})
\begin{equation}
{\sum\limits_{|\omega_{i}|\leq s_{i}}}\left\Vert D_{y_{1}}^{\omega_{1}%
}D_{y_{2}}^{\omega_{2}}...D_{y_{r}}^{\omega_{r}}\widetilde{m}(\lambda
^{\frac{1}{\beta_{1}}}\xi_{1},...,\lambda^{\frac{1}{\beta_{S}}}\xi_{S}%
,\lambda^{\frac{1}{\varkappa_{S+1}}}\xi_{1},...,\lambda^{\frac{1}%
{\varkappa_{K}}}\xi_{S})\right\Vert _{L_{p}(B_{\nu})}\leq\mu,\label{4.21}%
\end{equation}
where $B_{\nu}$ is defined in \eqref{4.9} and $\mu$ is a positive number.

Then the operator $M$ from \eqref{4.5} is a bounded linear operator from the
space $H_{x^{(1)}}^{\gamma\beta}(R^{K})$ to the space $H_{x^{(1)},x^{(2)}%
}^{\gamma\beta,\quad\gamma\varkappa}(R^{K})$, and
\begin{equation}
\left\Vert Mu\right\Vert _{H_{x^{(1)},x^{(2)}}^{\gamma\beta,\quad
\gamma\varkappa}(R^{K})}\leq C(K,\gamma,\beta,\varkappa,p,\nu,\{s_{i}%
\})\mu\left\Vert u\right\Vert _{H_{x^{(1)}}^{\gamma\beta}(R^{K})},\tag{4.22}%
\end{equation}%
\begin{equation}
{\sum\limits_{i=1}^{S}}\left\langle Mu\right\rangle _{x_{i},R^{K}}%
^{(\gamma\beta_{i})}+{\sum\limits_{i=S+1}^{K}}\left\langle Mu\right\rangle
_{x_{i},R^{K}}^{(\gamma\varkappa_{i})}\leq C(K,\gamma,\beta,\varkappa
,p,\nu,\{s_{i}\})\mu{\sum\limits_{i=1}^{S}}\left\langle u\right\rangle
_{x_{i},R^{K}}^{(\gamma\beta_{i})}.\tag{4.23}%
\end{equation}

\end{theorem}

As the conclusion of the section we present two auxiliary statements we need
in what follows.

\begin{lemma}
\label{L4.1}(\cite{23}: L.2.8) Let a function $\widetilde{f}(i\xi_{0} ,\xi)$
($i$ is the imaginary unit, $\xi_{0}\in R^{1}$, $\xi\in R^{N}$) be defined in
$R^{N+1}$ and can be extended to a function $\widetilde{f}(i\xi_{0}+a,\xi)$ in
the domain $a\geq0$ in the way that the extension $\widetilde{f}(i\xi
_{0}+a,\xi)$ possesses the properties:

1) $\widetilde{f}(i\xi_{0}+a,\xi)$ is continuous in the domain $a\geq0$;

2) the function $\widetilde{f}(i\xi_{0}+a,\xi)$ is analytic in the domain
$a>0$ with respect to the variable $p=i\xi_{0}+a$;

3) this function satisfies in the domain $a\geq0$ the following inequality
with some constants $C>0$, $m_{1}$, and $m_{2}$
\begin{equation}
\left\vert \widetilde{f}(p,\xi)\right\vert \leq C(1+|p|)^{m_{1}}%
(1+|\xi|)^{m_{2}}.\label{4.24}%
\end{equation}
Then the inverse Fourier transform of this function ( including in the sense
of distributions )
\[
f(t,x)=(2\pi)^{-(N+1)}{\int\limits_{R^{N+1}}}e^{it\xi_{0}+ix\xi}\widetilde
{f}(i\xi_{0},\xi)d\xi d\xi_{0}%
\]
vanishes for $t<0$
\begin{equation}
f(t,x)\equiv0,\quad t<0.\label{4.25}%
\end{equation}

\end{lemma}

\begin{lemma}
\label{L4.2} Let $f(t)\in C^{n}([0,\infty))$, $n\geq1$, $f^{(k)}(t)\in
L_{1}([0,\infty))$, $k=0,1,...,n$ and $f^{(k)}(0)=0$, $k=0,1,...,n-1$ and let
$\theta\in(n-1,n)$. Then the fractional derivative $D_{\ast t}^{\theta}f(t)$,
extended along with the function $f(t)$ itself, in the domain $\{t<0\}$ by
zero, has the following Fourier transform
\begin{equation}
\widehat{D_{\ast t}^{\theta}f(t)}(\xi_{0})\equiv\frac{1}{\sqrt{2\pi}}%
{\int\limits_{-\infty}^{\infty}}D_{\ast t}^{\theta}f(t)e^{-i\xi_{0}%
t}dt,\label{4.26}%
\end{equation}
where the integral is understood as an improper one. Moreover,
\begin{equation}
\widehat{D_{\ast t}^{\theta}f(t)}(\xi_{0})=(i\xi_{0})^{[\theta]}(i\xi
_{0})^{\{\theta\}}\widehat{f}(\xi_{0}),\label{4.27}%
\end{equation}
where $\widehat{f}(\xi_{0})$ is the Fourier transform of $f(t)$ and
\begin{equation}
(i\xi_{0})^{\{\theta\}}\equiv|\xi_{0}|^{\{\theta\}}e^{i\{\theta\}\frac{\pi}%
{2}sign(\xi_{0})}\label{4.28}%
\end{equation}
is an analytic extension of the function $z^{\{\theta\}}$ from the positive
real half-axis to the right half-plane $\operatorname{Re}z>0$.
\end{lemma}

\begin{remark}
\label{R4.1} We do not present here a detailed proof of this statement because
it, by essence, is contained in, for example, \cite{Samko}, Section 7. We note
only that at applications of \cite{Samko}, Section 7, one should take into
account that the monograph \cite{Samko} considers the direct Fourier transform
with $e^{ix\xi_{0}}$ and the present paper makes use in \eqref{1.7} of the
kernel $e^{-ix\xi_{0}}$. Besides, it should be taken into account that, since
$f(t)\equiv0$ for $t<0$, then
\[
D_{\ast t}^{\theta}f(t)=\frac{1}{\Gamma(1-\theta)}{\int\limits_{0}^{t}}%
\frac{f^{(n)}(\tau)d\tau}{(t-\tau)^{\theta-n+1}}=\frac{1}{\Gamma(1-\theta
)}{\int\limits_{-\infty}^{t}}\frac{f^{(n)}(\tau)d\tau}{(t-\tau)^{\theta-n+1}}.
\]

\end{remark}

\section{Cauchy problem for equation \eqref{1.1} in the case of a noninteger
$\theta\in(0,1)$ ,$\theta\alpha\in(0,1)$.}

\label{s5}

In the present section we study problem \eqref{1.1}, \eqref{1.2} with
$\theta\in(0,1)$ in the anisotropic H\"{o}lder spaces $C^{\overline{\sigma
}(1+\alpha),\theta+\theta\alpha}(\overline{R_{T}^{N}})$ from \eqref{2.8}. In
this way, consider the following Cauchy problem for a defined in
$\overline{R_{T}^{N}}$ unknown function $u(x,t)$
\begin{equation}
Lu(x,t)\equiv D_{\ast t}^{\theta}u(x,t)+{\sum_{k=1}^{r}}(-\Delta_{z_{k}%
})^{\frac{\sigma_{k}}{2}}u(x,t)=f(x,t),\quad(x,t)\in R_{T}^{N},\label{5.1}%
\end{equation}%
\begin{equation}
u(x,0)=u_{0}(x),\quad x\in R^{N},\label{5.2}%
\end{equation}
where
\begin{equation}
\theta\in(0,1),\quad\sigma_{k}>0,\quad k=1,...,r,\label{5.3.1}%
\end{equation}
and the given functions belong to the spaces
\begin{equation}
f(x,t)\in C^{\overline{\sigma}\alpha,\theta\alpha}(\overline{R_{T}^{N}}),\quad
u_{0}(x)\in C^{\overline{\sigma}(1+\alpha)}(R^{N}).\label{5.3}%
\end{equation}
In this section we suppose that $\alpha$ is so small that the following
condition is satisfied
\begin{equation}
\theta,\theta\alpha\in(0,1),\quad\sigma_{k}\alpha\in(0,1),\quad
k=1,...,r.\label{5.4}%
\end{equation}
Besides, the following compatibility conditions are expected to be met ( see
condition \eqref{1dop.1} and Remark \ref{1dop.000})
\begin{equation}
f(x,0)={\sum_{k=1}^{r}}(-\Delta_{z_{k}})^{\frac{\sigma_{k}}{2}}u_{0}(x),\quad
x\in R^{N}.\label{5.5}%
\end{equation}

The following theorem is valid.

\begin{theorem}
\label{T5.1} If conditions \eqref{5.3.1} - \eqref{5.5} are satisfied, then
problem \eqref{5.1}, \eqref{5.2} has the unique solution $u(x,t)\in
C^{\overline{\sigma}(1+\alpha),\theta+\theta\alpha}(\overline{R_{T}^{N}})$
with the following estimates
\begin{equation}
|u|_{\overline{R_{T}^{N}}}^{(\overline{\sigma}(1+\alpha),\theta+\theta\alpha
)}\leq C(\overline{\sigma},\theta,\alpha,T)\left(  |f|_{\overline{R_{T}^{N}}%
}^{(\overline{\sigma}\alpha,\theta\alpha)}+|u_{0}|_{R^{N}}^{(\overline{\sigma
}(1+\alpha))}\right)  , \label{5.6}%
\end{equation}%
\begin{equation}
\left\langle u\right\rangle _{\overline{R_{T}^{N}}}^{(\overline{\sigma
}(1+\alpha),\theta+\theta\alpha)}\leq C(\overline{\sigma},\theta
,\alpha)\left(  |f|_{\overline{R_{T}^{N}}}^{(\overline{\sigma}\alpha
,\theta\alpha)}+|u_{0}|_{R^{N}}^{(\overline{\sigma}(1+\alpha))}\right)  ,
\label{5.6.1}%
\end{equation}%
\begin{equation}
\left\vert u\right\vert _{\overline{R_{\widetilde{T}}^{N}}}^{(0)}\leq
C(\overline{\sigma},\theta,\alpha)\left(  |f|_{\overline{R_{T}^{N}}%
}^{(\overline{\sigma}\alpha,\theta\alpha)}+|u_{0}|_{R^{N}}^{(\overline{\sigma
}(1+\alpha))}\right)  \widetilde{T}^{\theta+\theta\alpha}+|u_{0}|_{R^{N}%
}^{(0)},\quad\widetilde{T}\leq T. \label{5.6.2}%
\end{equation}

\end{theorem}

The subsequent content of the section consists of the proof of the theorem and
this proof will be presented in several steps.

\subsection{Reduction of the initial data.}

\label{s5.1}

Note first that we can confine ourselves to the case
\begin{equation}
u_{0}(x)\equiv0,\quad f(x,0)\equiv0,\quad x\in R^{N}.\label{5.7}%
\end{equation}
Indeed, making in problem \eqref{5.1}, \eqref{5.2} the change of the unknown
function
\[
u(x,t)\rightarrow v(x,t)=u(x,t)-u_{0}(x),
\]
we see that $v(x,t)$ satisfies zero initial condition \eqref{5.2} and the
equation
\[
D_{\ast t}^{\theta}v+{\sum_{k=1}^{r}}(-\Delta_{z_{k}})^{\frac{\sigma_{k}}{2}%
}v=\widetilde{f(}x,t)\equiv f(x,t)-{\sum_{k=1}^{r}}(-\Delta_{z_{k}}%
)^{\frac{\sigma_{k}}{2}}u_{0}(x).
\]
From Proposition \ref{P3.6} it follows that $\widetilde{f(}x,t)\in
C^{\overline{\sigma}\alpha,\theta\alpha}(\overline{R_{T}^{N}})$, as well as
the original function $f(x,t)$, and besides
\begin{equation}
|\widetilde{f}|_{\overline{R_{T}^{N}}}^{(\overline{\sigma}\alpha,\theta
\alpha)}\leq C(\overline{\sigma},\theta,\alpha)\left(  |f|_{\overline
{R_{T}^{N}}}^{(\overline{\sigma}\alpha,\theta\alpha)}+|u_{0}|_{R^{N}%
}^{(\overline{\sigma}(1+\alpha))}\right)  .\label{5.8}%
\end{equation}
Moreover, by virtue of compatibility condition \eqref{5.5}, the condition
$\widetilde{f(}x,0)\equiv0$ is satisfied. Thus, we will assume below that
conditions \eqref{5.7} are satisfied.

\subsection{Extension of the data.}

\label{s5.2}

The function $f(x,t)$ can be extended in the domain $t>T$ to a finite in $t$
function with the preserving of the class $C^{\overline{\sigma}\alpha
,\theta\alpha}$ and with the qualified preserving of the norm
\begin{equation}
|f|_{\overline{R_{\infty}^{N}}}^{(\overline{\sigma}\alpha,\theta\alpha)}\leq
C(\overline{\sigma},\theta,\alpha)|f|_{\overline{R_{T}^{N}}}^{(\overline
{\sigma}\alpha,\theta\alpha)}. \label{5.9}%
\end{equation}
The way of such extension is described in, for example, \cite{Sol1},
\cite{LadParab}, Ch.4 (in our case under condition \eqref{5.4} the function
$f(x,t)$ can be simply extended in the even way trough the point $t=T$ with a
subsequent cut-off with respect to $t$).

Besides, since the second condition in \eqref{5.7} is satisfied and by virtue
of the first condition in \eqref{5.4}, we can assume that $f(x,t)$ is extended
by identical zero in the domain $t<0$ with the preserving of the class
$C^{\overline{\sigma}\alpha,\theta\alpha}$ and of the norm $|f|_{R^{N+1}
}^{(\overline{\sigma}\alpha,\theta\alpha)}=|f|_{\overline{R_{\infty}^{N}}%
}^{(\overline{\sigma}\alpha,\theta\alpha)}$.

Further, it is convenient for us for technical reasons to assume that the
extension of the function $f(x,t)$ is made in such a way that this function is
the derivative in $t$ of a sufficiently high order $n\geq1$ of a finite in $t$
function $F(x,t)$ with the properties
\[
F(x,t),\frac{\partial^{n}F(x,t)}{\partial t^{n}}\in C^{\overline{\sigma}%
\alpha,\theta\alpha}(R^{N+1}),\quad F(x,t)\equiv0,t<0,
\]%
\begin{equation}
f(x,t)=\frac{\partial^{n}F(x,t)}{\partial t^{n}},(x,t)\in R^{N+1}%
.\label{S.9.1}%
\end{equation}
This can be done as follows. Let the support of the extended in $t$ on the
whole $R^{N+1}$ function $f(x,t)$ is included in the domain $R^{N}%
\times\lbrack0,T_{1}]$, $T_{1}>T$. Define $F(x,t)$ as $n$-multiple integral in
$t$ of $f(x,t)$ with the subsequent cut-off
\begin{equation}
F(x,t)\equiv\eta(t){\int\limits_{-1}^{t}}\frac{(t-\tau)^{n-1}}{(n-1)!}%
f(x,\tau)d\tau,t\geq-1,\quad F(x,t)\equiv0,t<-1,\label{S.9.2}%
\end{equation}
where
\[
\eta(t)\in C^{\infty}(R^{1}),\quad\eta(t)\equiv1,t\leq T_{1}+1,\quad
\eta(t)\equiv0,t>T_{1}+1.
\]
In fact, $\partial^{n}F(x,t)/\partial t^{n}$ has the somewhat bigger support
in $R^{N}\times\lbrack0,T_{1}+1]$ than that of the original extended function
$f(x,t)$ and thus $\partial^{n}F(x,t)/\partial t^{n}$ does not coincides with
the original $f(x,t)$ for $T_{1}<t<T_{1}+1$. However, since $\partial
^{n}F(x,t)/\partial t^{n}=f(x,t)$ for $0\leq t\leq T$, we will consider it as
the finite in $t$ extension we need of the original function $f(x,t)$.
Properties \eqref{S.9.1} follow directly from the way of construction of
$F(x,t)$ in \eqref{S.9.2}.

Thus, in what follows we assume that $f(x,t)$ is a finite in $t$ function,
which is defined on the whole $R^{N+1}=R^{N}\times(-\infty,\infty)$, and it is
the derivative of order $n\geq1$ of some finite in $t$ function $F(x,t)$ with
the properties in \eqref{S.9.1}.

\subsection{Formulation of the problem in the domain $R^{N+1}=R^{N}%
\times(-\infty,\infty)$.}

\label{s5.3}

Since the function $f(x,t)$ is defined in $R^{N+1}$, we can assume that
problem \eqref{5.1}, \eqref{5.2} is the restriction to the interval $[0,T]$ of
a similar problem in the domain $R^{N+1}$ and we assume, in particular, that
$T=\infty$. Besides, in view of zero initial condition and in view of the
conditions $\theta<1$ and $\theta\alpha<1$, we can consider the alleged
solution $u(x,t)$ to \eqref{5.1}, \eqref{5.2} to be extended by identical zero
in the domain $\{t<0\}$. Under this extension the function $u(x,t)$ preserves
it's class $C^{\overline{\sigma}(1+\alpha),\theta+\theta\alpha}$, since at
$t=0$ not only $u(x,0)\equiv0$, but in the case $\theta+\theta\alpha\in(1,2)$
also $u_{t}(x,0)\equiv0$, as it follows from Proposition \ref{P3.1.1}.
Moreover, $D_{\ast t}^{\theta}u(x,t)$, which is defined only for $t>0$, also
can be extended by zero in the domain $t<0$. Such extension preserves the
smoothness not only for $u(x,t)$ but for $D_{\ast t}^{\theta}u(x,t)$ as well,
since this derivative satisfies $\left[  D_{\ast t}^{\theta}u(x,t)\right]
|_{t=0}=0$.

Thus problem \eqref{5.1}, \eqref{5.2} can be reformulated as the problem of
finding a function $u(x,t)\in C^{\overline{\sigma}(1+\alpha),\theta
+\theta\alpha}(R^{N+1})$, that satisfies in $R^{N+1}$ equation \eqref{5.1}
(with $f(x,t)\equiv0$ for $t<0$)
\begin{equation}
Lu(x,t)\equiv D_{\ast t}^{\theta}u(x,t)+{\sum_{k=1}^{r}}(-\Delta_{z_{k}%
})^{\frac{\sigma_{k}}{2}}u(x,t)=f(x,t),\quad(x,t)\in R^{N+1}\label{5.10}%
\end{equation}
and that satisfies the condition
\begin{equation}
u(x,t)\equiv0,\quad t<0.\label{5.11}%
\end{equation}

\subsection{ Smoothing of the data .}

\label{s5.4}

We are going to find a solution to problem \eqref{5.1}, \eqref{5.2} (and it's
sharp estimate) as the limit for a sequens of solutions of the class
$C^{\infty}(R^{N+1})$ to the same problem with finite with respect to all
variables data. Therefore we describe now some smoothing process for the data.
Fix a function (a mollifier kernel) with
\[
\omega(x,t)\in C^{\infty}(R^{N+1}),\quad\omega(x,t)\geq0,\quad\omega
(x,t)\equiv0,|x|+|t|>1,{\int\limits_{R^{N+1}}}\omega(x,t)dxdt=1,\quad
\]%
\begin{equation}
\omega_{\varepsilon}(x,t)\equiv\frac{1}{\varepsilon^{N+1}}\omega(\frac
{x}{\varepsilon},\frac{t}{\varepsilon}),\varepsilon>0.\label{5.12}%
\end{equation}
Fix, further, a cut-off function $\zeta(x)$ with
\begin{equation}
\zeta(x)\in C^{\infty}(R^{N}),\quad\zeta(x)\equiv1,|x|\leq1,\quad
\zeta(x)\equiv0,|x|>2.\label{5.13}%
\end{equation}
Introduce now the smoothed and cut off with respect to $x$ function from
\eqref{S.9.2}
\begin{equation}
F_{m,\varepsilon}(x,t)\equiv\zeta(\frac{x}{m}){\int\limits_{R^{N+1}}}%
\omega_{\varepsilon}(x-\xi,t-\tau)F(\xi,\tau-\varepsilon)dxdt,\quad
m=1,2,....\label{5.14}%
\end{equation}
Here, in view of properties of $\omega_{\varepsilon}(x,t)$ and
$F(x,t-\varepsilon)$, the function $F_{m,\varepsilon}(x,t)$ is finite with the
support in $\{|x|\leq2m\}\times\lbrack0,T_{\ast}]$, where $T_{\ast}$ is fixed
and does not depend on $m$ and $\varepsilon$. By virtue of the known
properties of convolution with a smooth kernel,
\[
F_{m,\varepsilon}(x,t)\in C^{\infty}(R^{N+1}),
\]%
\begin{equation}
|F_{m,\varepsilon}|_{R^{N+1}}^{(\sigma\alpha,\theta\alpha)}\leq
C|F(x,t-\varepsilon)|_{R^{N+1}}^{(\sigma\alpha,\theta\alpha)}%
=C|F(x,t)|_{R^{N+1}}^{(\sigma\alpha,\theta\alpha)}\label{5.15}%
\end{equation}
and, besides, for each compact set
\begin{equation}
P_{R}\equiv\{|x|\leq R\}\times\{|t|\leq R\},\quad R>0,\label{5.15.1}%
\end{equation}
and for an arbitrary $\alpha^{\prime}<\alpha$
\begin{equation}
\left\vert F_{m,\varepsilon}(x,t)-F(x,t)\right\vert _{P_{R}}^{(\sigma
\alpha^{\prime},\theta\alpha^{\prime})}\rightarrow0,\quad\varepsilon
\rightarrow0,m\rightarrow\infty.\label{5.16}%
\end{equation}
Further, in view of the properties of the convolution
\[
\frac{\partial^{n}F_{m,\varepsilon}(x,t)}{\partial t^{n}}=\zeta(\frac{x}%
{m})\omega_{\varepsilon}\ast\frac{\partial^{n}F(x,t-\varepsilon)}{\partial
t^{n}}=
\]%
\begin{equation}
=\zeta(\frac{x}{m}){\int\limits_{R^{N+1}}}\omega_{\varepsilon}(x-\xi
,t-\tau)f(\xi,\tau-\varepsilon)dxdt\equiv f_{m,\varepsilon}(x,t).\label{5.17}%
\end{equation}
And, analogously to the properties of $F_{m,\varepsilon}(x,t)$, the function
$f_{m,\varepsilon}(x,t)$ is finite with it's support in $\{|x|\leq
2m\}\times\lbrack0,T_{\ast}]$,
\[
f_{m,\varepsilon}(x,t)\in C^{\infty}(R^{N+1}),
\]%
\begin{equation}
|f_{m,\varepsilon}|_{R^{N+1}}^{(\sigma\alpha,\theta\alpha)}\leq
C|f(x,t-\varepsilon)|_{R^{N+1}}^{(\sigma\alpha,\theta\alpha)}%
=C|f(x,t)|_{R^{N+1}}^{(\sigma\alpha,\theta\alpha)}.\label{5.18}%
\end{equation}
Moreover, for each compact set $P_{R}$ from \eqref{5.15.1} and for any
$\alpha^{\prime}<\alpha$
\begin{equation}
\left\vert f_{m,\varepsilon}(x,t)-f(x,t)\right\vert _{P_{R}}^{(\sigma
\alpha^{\prime},\theta\alpha^{\prime})}\rightarrow0,\quad\varepsilon
\rightarrow0,m\rightarrow\infty.\label{5.19}%
\end{equation}

\subsection{Constructing a solution in the case of smooth finite data.}

\label{s5.5}

Suppose for a while that we are given such a solution $u(x,t)$ to
\eqref{5.10}, \eqref{5.11} from the class $u(x,t)\in C^{\overline{\sigma
}(1+\alpha),\theta+\theta\alpha}(R^{N+1})$ that it sufficiently rapidly decays
for $t\rightarrow+\infty$ and that vanishes for $t<0$. Make in equation
\eqref{5.10} the Fourier transform and denote the dual variable to $t$ by
$\xi_{0}$. The Fourier transform of $D_{\ast t}^{\theta}u(x,t)$ is correctly
defined in view of Lemma \ref{L4.2}: we need for that only $u(x,t)\in
L_{1}([0,\infty))$ for an arbitrary $x\in R^{N}$ since in our case $\theta
\in(0,1)$. The Fourier transform of the sum of the fractional powers of
Laplace operators we understand in the sense of the space $S^{\prime}%
(R^{N+1})$ since for each $t$ we have $u(x,\cdot)\in C^{\overline{\sigma
}(1+\alpha)}(R^{N})\subset S_{\sigma}^{\prime}(R^{N})\subset S^{\prime}%
(R^{N})$ - see Lemma \ref{LS.2}. In terms of the Fourier images equation
\eqref{5.10} takes the form (see \eqref{1.7}, \eqref{1.8}, \eqref{4.27})
\begin{equation}
\widehat{Lu}(\xi,\xi_{0})=(i\xi_{0})^{\theta}\widehat{u}(\xi,\xi_{0}%
)+{\sum_{k=1}^{r}}|\zeta_{k}|^{\sigma_{k}}\widehat{u}(\xi,\xi_{0})=\widehat
{f}(\xi,\xi_{0}),\quad(\xi,\xi_{0})\in R^{N+1}.\label{5.20}%
\end{equation}
Here $\xi=(\xi_{1},...,\xi_{N})$ are the dual variables to $x$, $\zeta
_{k}=(\xi_{i_{k}+1},...,\xi_{i_{k}+N_{k}})$ is group from the whole set $\xi$,
that is the dual group to the group $z_{k}$ from $x$. From equality
\eqref{5.10} it follows that
\begin{equation}
\widehat{u}(\xi,\xi_{0})=\frac{\widehat{f}(\xi,\xi_{0})}{(i\xi_{0})^{\theta
}+{\sum_{k=1}^{r}}|\zeta_{k}|^{\sigma_{k}}}+\widehat{P}(\xi,\xi_{0}%
),\label{5.21}%
\end{equation}
where $\widehat{P}(\xi,\xi_{0})$ is some distribution from $S^{\prime}%
(R^{N+1})$ with the support at the point $(\xi=0,\xi_{0}=0)$, which is a
finite linear combination of the $\delta$-function and it's derivatives (and
consequently which is the Fourier transform of a polynomial $P(x,t)$).
Consider the first term in the right hand side of \eqref{5.21} that is the
function
\[
\widehat{u}_{1}(\xi,\xi_{0})\equiv\frac{\widehat{f}(\xi,\xi_{0})}{(i\xi
_{0})^{\theta}+{\sum_{k=1}^{r}}|\zeta_{k}|^{\sigma_{k}}}=
\]%
\begin{equation}
=\frac{(i\xi_{0})^{n}}{(i\xi_{0})^{\theta}+{\sum_{k=1}^{r}}|\zeta_{k}%
|^{\sigma_{k}}}\widehat{F}(\xi,\xi_{0}),\label{5.22}%
\end{equation}
where we took advantage of the fact that
\[
\widehat{f}(\xi,\xi_{0})=\widehat{\frac{\partial^{n}F}{\partial t^{n}}}%
(\xi,\xi_{0})=(i\xi_{0})^{n}\widehat{F}(\xi,\xi_{0}).
\]
Note first that the function $\widehat{F}(\xi,\xi_{0})$, which is the Fourier
image of a finite function from $C^{\infty}(R^{N+1})$, decays at infinity
faster than any power of $(1+|\xi|+|\xi_{0}|)$. At the same time the function
$(i\xi_{0})^{n}/\left[  (i\xi_{0})^{\theta}+{\sum_{k=1}^{r}}|\zeta
_{k}|^{\sigma_{k}}\right]  $ is bounded at zero and grows at infinity not
faster than $(1+|\xi_{0}|)^{n-\theta}$. Therefore for an arbitrary $M>0$
\begin{equation}
\left\vert \widehat{u}_{1}(\xi,\xi_{0})\right\vert \leq C_{M}(1+|\xi|+|\xi
_{0}|)^{-M},\quad M>0.\label{5.23}%
\end{equation}
Consequently, there exists the Fourier pre-image of $\widehat{u}_{1}(\xi
,\xi_{0})$ and thus the corresponding Fourier pre-image $u_{1}(x,t)$ belongs
to $C^{\infty}(R^{N+1})$ and all derivatives of $u_{1}(x,t)$ are bounded on
$R^{N+1}$.

Further, in view of the described above properties of $\widehat{F}(\xi,\xi
_{0})$, the function $\widehat{u}_{1}(\xi,\xi_{0})$ has the derivatives in
$\xi_{0}$ up to the order $n-1$ with the estimate
\begin{equation}
\left\vert \frac{\partial^{k}\widehat{u}_{1}(\xi,\xi_{0})}{\partial\xi_{0}%
^{k}}\right\vert \leq C_{M}(1+|\xi|+|\xi_{0}|)^{-M},\quad k=0,...,n-1,M>0.
\label{5.24}%
\end{equation}
Therefore the function $u_{1}(x,t)$ satisfies the estimate
\begin{equation}
\left\vert \frac{\partial^{k}u_{1}(x,t)}{\partial t^{k}}\right\vert \leq
C_{k}(1+|t|)^{-(n-1)+k},\quad k=0,...,n-1, \label{5.25}%
\end{equation}
and therefore for $n \geq3$ the function $u_{1}(x,t)$ belongs to the space
$L_{1}(R^{1})$ at each fixed $x$.

At last, denote the factor at $\widehat{F}(\xi,\xi_{0})$ in the second
equality in \eqref{5.22} by
\begin{equation}
\widehat{G_{n}}(\xi,\xi_{0})\equiv\frac{(i\xi_{0})^{n}}{(i\xi_{0})^{\theta
}+{\sum_{k=1}^{r}}|\zeta_{k}|^{\sigma_{k}}}.\label{5.26}%
\end{equation}
This function satisfies all the conditions of Lemma \ref{L4.1} and therefore
the support of it's Fourier pre-image (that is the support of the distribution
$G_{n}(x,t)$) is completely included in the set $\{t\geq0\}$. The same is true
with respect to the function $F(x,t)$ by the assumptions. At the same time, by
virtue of the properties of the Fourier transform, from \eqref{5.22} it
follows that $u_{1}(x,t)$ is a convolution of the distribution $G_{n}(x,t)$
and the function $F(x,t)\in C^{\infty}(R^{N+1})$,
\[
u_{1}(x,t)=G_{n}(x,t)\ast F(x,t).
\]
Thus, in view of the known properties of convolution
\begin{equation}
u_{1}(x,t)=0,\quad t<0.\label{5.27}%
\end{equation}

The listed above properties of $u_{1}(x,t)$ mean that the operator
$Lu_{1}(x,t)$ in the left hand side of \eqref{5.10} is correctly defined on
$u_{1}(x,t)$. Besides, the Fourier transforms of all terms in the expression
for $Lu_{1}(x,t)$ are also correctly defined (in particular, for each fixed
$x\in R^{N}$ the function $u_{1}(x,t)$ satisfies the conditions of Lemma
\ref{L4.2}. Finally, the Fourier transform of $Lu_{1}(x,t)$ is equal to
\[
\widehat{Lu_{1}}(\xi,\xi_{0})=(i\xi_{0})^{\theta}\widehat{u}_{1}(\xi,\xi
_{0})+{\sum_{k=1}^{r}}|\zeta_{k}|^{\sigma_{k}}\widehat{u}_{1}(\xi,\xi
_{0})=\widehat{f}(\xi,\xi_{0})
\]
by virtue of the definition of $u_{1}(x,t)$ in \eqref{5.22}. But this means
that $u_{1}(x,t)$ satisfies problem \eqref{5.10}, \eqref{5.11} that is also
problem \eqref{5.1}, \eqref{5.2} with zero initial condition. Thus in equality
\eqref{5.21}
\[
\widehat{u}(\xi,\xi_{0})=\widehat{u}_{1}(\xi,\xi_{0})+\widehat{P}(\xi,\xi_{0})
\]
the distribution $\widehat{P}(\xi,\xi_{0})$ is identically equal to zero,
because $\widehat{u}(\xi,\xi_{0})$ and $\widehat{u}_{1}(\xi,\xi_{0})$ are the
Fourier images of bounded and vanishing for $t<0$ functions, while
$\widehat{P}(\xi,\xi_{0})$ is the Fourier image of a polynomial in the
variables $x$ and $t$. This situation is possible only in the case, when the
polynomial is identically equal to zero.

So, it is shown that for a function $f(x,t)\in C_{0}^{\infty}(R^{N+1})$
vanishing for $t<0$ there exists a solution $u(x,t)$ to problem \eqref{5.10},
\eqref{5.11} from the class $C^{\infty}(R^{N+1})$ with bounded derivatives of
any order and with some rapid decay for $t\rightarrow\infty$ for each fixed
$x\in R^{N}$. And as such solution, one can take $u(x,t)=u_{1}(x,t)$, since in
this case
\begin{equation}
\widehat{u}(\xi,\xi_{0})=\frac{\widehat{f}(\xi,\xi_{0})}{(i\xi_{0})^{\theta
}+{\sum_{k=1}^{r}}|\zeta_{k}|^{\sigma_{k}}}.\label{5.28}%
\end{equation}

\subsection{Estimate for H\"{o}lder norm of solution in the case of smooth
finite data.}

\label{s5.6}

We obtain now an estimate of solution from \eqref{5.28} in the space
$C^{\overline{\sigma}(1+\alpha),\theta+\theta\alpha}(R^{N+1})$ for infinitely
smooth finite data. To prove such estimate we are going to make use of
representation \eqref{5.28} and Theorem \ref{T4.1}. Consider first the time
derivative $D_{\ast t}^{\theta}u(x,t)$. On the ground of \eqref{4.27} and
\eqref{5.28} the Fourier transform of this derivative is equal to
\[
\widehat{D_{\ast t}^{\theta}u}(\xi,\xi_{0})=\widehat{m}_{0}(\xi,\xi
_{0})\widehat{f}(\xi,\xi_{0}),
\]
where
\begin{equation}
\widehat{m}_{0}(\xi,\xi_{0})\equiv\frac{(i\xi_{0})^{\theta}}{(i\xi
_{0})^{\theta}+{\sum_{k=1}^{r}}|\zeta_{k}|^{\sigma_{k}}}.\label{5.29}%
\end{equation}
Thus denoting by $\emph{F}$ the Fourier transform with respect to whole set of
the variables $(x,t)$, we have
\begin{equation}
D_{\ast t}^{\theta}u(x,t)=\emph{F}^{-1}\left[  \widehat{m}_{0}(\xi,\xi
_{0})\emph{F}\left[  f(x,t)\right]  \right]  ,\label{5.30}%
\end{equation}
that is $D_{\ast t}^{\theta}u(x,t)$ is obtained from $f(x,t)$ by the Fourier
multiplier $\widehat{m}_{0}(\xi,\xi_{0})$. Verify the properties of
$\widehat{m}_{0}(\xi,\xi_{0})$ required by Theorem \ref{T4.1}. At first, the
$\widehat{m}_{0}(\xi,\xi_{0})$ is evidently bounded on $R^{N+1}$. Further,
according to Theorem \ref{T4.1}, split the set of the variables $(\xi,\xi
_{0})$ into the groups $(\xi,\xi_{0})=(\zeta_{1},...,\zeta_{r},\zeta_{0})$,
where $\zeta_{1}$ - $\zeta_{r}$ are defined in \eqref{1.7}, \eqref{1.8} as the
dual for the space groups $z_{k}$, $k=1,...,r$, and $\zeta_{0}=\xi_{0}$ is the
dual variable to $t$. The ordered set of the smoothness exponents of $f(x,t)$
in the space variables and time is $(\sigma_{1}\alpha,...,\sigma_{r}%
\alpha,\theta\alpha)$. So in condition \eqref{4.10} we must consider the
derivatives of the functions ($\lambda>0$)
\[
\widehat{m}_{0}(\lambda^{\frac{1}{\sigma_{1}\alpha}}\zeta_{1},...,\lambda
^{\frac{1}{\sigma_{r}\alpha}}\zeta_{r},\lambda^{\frac{1}{\theta\alpha}}\xi
_{0})=\frac{(i\lambda^{\frac{1}{\theta\alpha}}\xi_{0})^{\theta}}%
{(i\lambda^{\frac{1}{\theta\alpha}}\xi_{0})^{\theta}+{\sum_{k=1}^{r}}%
|\lambda^{\frac{1}{\sigma_{k}\alpha}}\zeta_{k}|^{\sigma_{k}}}=
\]%
\begin{equation}
=\frac{(i\xi_{0})^{\theta}}{(i\xi_{0})^{\theta}+{\sum_{k=1}^{r}}|\zeta
_{k}|^{\sigma_{k}}}=\widehat{m}_{0}(\xi,\xi_{0}).\label{5.30.0}%
\end{equation}
That is condition \eqref{4.10} must be verified for the function $\widehat
{m}_{0}(\xi,\xi_{0})$ itself. Let
\begin{equation}
p=\frac{1}{1-\delta}>1,\quad\varepsilon\in(0,1),\label{5.30.1}%
\end{equation}
where $\delta>0$ is sufficiently small and will be chosen below. Let further
$s_{0}=1$ so that according to \eqref{4.8}
\begin{equation}
1=s_{0}>\frac{N_{0}}{p}=\frac{1}{p}=1-\delta,\label{5.30.2}%
\end{equation}
where $N_{0}=1$ is the dimension of the group $\zeta_{0}=(\xi_{0})$.
Consequently, it is enough to consider the integrability with the power $p$
over the annulus $B_{\nu}\equiv\{(\xi,\xi_{0})\in R^{N+1}:\nu\leq|(\xi,\xi
_{0})|\leq\nu^{-1},\nu>0\}$ from \eqref{4.9} of the (possibly mixed)
derivatives of $\widehat{m}_{0}(\xi,\xi_{0})$, that contains the
differentiation in $\xi_{0}$ of order not higher than one. Analogously, choose
$s_{k}=N_{k}$, $k=1,...,r$, where $N_{k}$ is the dimension of the group
$\zeta_{k}$, so that
\begin{equation}
s_{k}=N_{k}>\frac{N_{k}}{p}=N_{k}(1-\delta).\label{5.30.3}%
\end{equation}
Therefore it is enough to consider the integrability with the power $p$ over
the annulus $B_{\nu}$ of the mixed derivatives of $\widehat{m}_{0}(\xi,\xi
_{0})$ in the variables of the group $\zeta_{k}$ of order not higher than
$N_{k}$. At the same time direct simple consideration shows that on $B_{\nu}$
\begin{equation}
\left\vert D_{\xi_{0}}^{\omega_{0}}D_{\zeta_{1}}^{\omega_{1}}...D_{\zeta_{r}%
}^{\omega_{r}}\widehat{m}_{0}(\xi,\xi_{0})\right\vert \leq C_{\nu
,\theta,\overline{\sigma}}|\xi_{0}|^{-1+\theta}{\prod\limits_{k=1}^{r}}%
|\zeta_{k}|^{-N_{k}+\sigma_{k}},\label{5.31}%
\end{equation}
under the condition that $\omega_{0}\leq1=s_{0}$, $|\omega_{k}|\leq
N_{k}=s_{k}$. Consequently, we have for the integration of this derivative
with the power $p$ over $B_{\nu}$
\[
{\int\limits_{B_{\nu}}}\left\vert D_{\xi_{0}}^{\omega_{0}}D_{\zeta_{1}%
}^{\omega_{1}}...D_{\zeta_{r}}^{\omega_{r}}\widehat{m}_{0}(\xi,\xi
_{0})\right\vert ^{p}d\xi d\xi_{0}\leq
\]%
\begin{equation}
\leq C_{\nu,\theta,\overline{\sigma}}{\int\limits_{|\xi_{0}|\leq\nu^{-1}}}%
|\xi_{0}|^{-\frac{1-\theta}{1-\delta}}d\xi_{0}{\prod\limits_{k=1}^{r}%
\int\limits_{|\zeta_{k}|\leq\nu^{-1}}}|\zeta_{k}|^{-\frac{N_{k}-\sigma_{k}%
}{1-\delta}}d\zeta_{k}<\infty,\label{5.32}%
\end{equation}
if $\delta>0$ chosen sufficiently small to satisfy the conditions
\begin{equation}
\frac{1-\theta}{1-\delta}<1,\quad\frac{N_{k}-\sigma_{k}}{1-\delta}<N_{k},\quad
k=1,...,r.\label{5.33}%
\end{equation}
Thus, under the choice of $\delta>0$ from conditions \eqref{5.33}, condition
\eqref{4.10} is satisfied for the multiplier $\widehat{m}_{0}(\xi,\xi_{0})$
and therefore, according to Theorem \ref{T4.1},
\begin{equation}
\left\langle D_{\ast t}^{\theta}u(x,t)\right\rangle _{R^{N+1}}^{(\overline
{\sigma}\alpha,\theta\alpha)}\leq C(\theta,\overline{\sigma})\left\langle
f(x,t)\right\rangle _{R^{N+1}}^{(\overline{\sigma}\alpha,\theta\alpha
)}.\label{5.34}%
\end{equation}

Analogously, for the Fourier transform of the fractional Laplace operator
$(-\Delta_{z_{i}})^{\frac{\sigma_{i}}{2}}u(x,t)$ on the solution $u(x,t)$ with
respect to the space group $z_{i}$ we have the representation
\begin{equation}
\widehat{(-\Delta_{z_{i}})^{\frac{\sigma_{i}}{2}}u}(\xi,\xi_{0})=\frac
{|\zeta_{i}|^{\sigma_{i}}}{(i\xi_{0})^{\theta}+{\sum_{k=1}^{r}}|\zeta
_{k}|^{\sigma_{k}}}\widehat{f}(\xi,\xi_{0})\equiv\label{5.35}%
\end{equation}%
\[
\equiv\widehat{m}_{i}(\xi,\xi_{0})\widehat{f}(\xi,\xi_{0}).
\]
The direct verification shows that the multiplier $\widehat{m}_{i}(\xi,\xi
_{0})$ possesses, similar to $\widehat{m}_{0}(\xi,\xi_{0})$, all the necessary
properties for the application of Theorem \ref{T4.1}, including estimates
\eqref{5.31} and \eqref{5.32} with the same $p>1$. Consequently
\begin{equation}
\left\langle (-\Delta_{z_{i}})^{\frac{\sigma_{i}}{2}}u(x,t)\right\rangle
_{R^{N+1}}^{(\overline{\sigma}\alpha,\theta\alpha)}\leq C(\theta
,\overline{\sigma})\left\langle f(x,t)\right\rangle _{R^{N+1}}^{(\overline
{\sigma}\alpha,\theta\alpha)}.\label{5.36}%
\end{equation}
Based now on \eqref{5.34}, \eqref{5.36} and making use of Proposition
\ref{P3.1.1} and of estimate \eqref{3.6.1}, we get
\begin{equation}
\left\langle u(x,t)\right\rangle _{\overline{R_{T}^{N}}}^{(\overline{\sigma
}+\overline{\sigma}\alpha,\theta+\theta\alpha)}\leq C(\theta,\overline{\sigma
})\left\vert f(x,t)\right\vert _{\overline{R_{T}^{N}}}^{(\overline{\sigma
}\alpha,\theta\alpha)},\label{5.37.1}%
\end{equation}
where we suppose that either the support of $f(x,t)$ is included in
$\overline{R_{T}^{N}}$ or the norm of $f(x,t)$ on it's compact support is
estimated by it's norm over $\overline{R_{T}^{N}}$ (as it is the case after
the extension of $f(x,t)$ out of $\overline{R_{T}^{N}}$). Besides, we are
considering now the restriction of the solution $u(x,t)$ from $R^{N+1}$ to
$\overline{R_{T}^{N}}$.

Further we note that estimate \eqref{5.37.1} for the highest seminorm of
$u(x,t)$ does not depend on the dimensions of the the support of $f(x,t)$ and
it is defined only by the total norm $\left\vert f(x,t)\right\vert _{R^{N+1}%
}^{(\overline{\sigma}\alpha,\theta\alpha)}$ over $R^{N+1}$. At the same time
the lowest norm $|u(x,t)|_{\overline{R_{T}^{N}}}^{(0)}$ depends on $T$ in
general . The simplest estimate for this norm follows from the fact that
$u(x,0)\equiv0$, from the last assertion of Proposition \ref{P3.1.1}, from the
contained in \eqref{5.37.1} estimate
\[
\left\langle u(x,t)\right\rangle _{t,\overline{R_{T}^{N}}}^{(\theta
+\theta\alpha)}\leq C(\theta,\overline{\sigma})\left\vert f(x,t)\right\vert
_{\overline{R_{T}^{N}}}^{(\overline{\sigma}\alpha,\theta\alpha)},
\]
and from from the definition of the H\"{o}lder seminorm with respect to $t$
itself. Such estimate takes place for an arbitrary $\widetilde{T}\leq T$
\begin{equation}
|u(x,t)|_{\overline{R_{\widetilde{T}}^{N}}}^{(0)}\leq\left\langle
u(x,t)\right\rangle _{t,\overline{R_{\widetilde{T}}^{N}}}^{(\theta
+\theta\alpha)}\widetilde{T}^{\theta+\theta\alpha}\leq C(\theta,\overline
{\sigma})\left\vert f(x,t)\right\vert _{\overline{R_{T}^{N}}}^{(\overline
{\sigma}\alpha,\theta\alpha)}\widetilde{T}^{\theta+\theta\alpha}%
,\quad\widetilde{T}\leq T. \label{5.37.2}%
\end{equation}

Combining \eqref{5.37.1} and \eqref{5.37.2} we arrive at the full estimate of
the norm
\begin{equation}
\left\vert u(x,t)\right\vert _{\overline{R_{T}^{N}}}^{(\overline{\sigma
}+\overline{\sigma}\alpha,\theta+\theta\alpha)}\leq C(\theta,\overline{\sigma
},T)\left\vert f(x,t)\right\vert _{\overline{R_{T}^{N}}}^{(\overline{\sigma
}\alpha,\theta\alpha)}=C(\theta,\overline{\sigma},T)\left\vert
f(x,t)\right\vert _{\overline{R_{T}^{N}}}^{(\overline{\sigma}\alpha
,\theta\alpha)}. \label{5.37}%
\end{equation}

\subsection{Existence of a solution for an arbitrary $f(x,t)\in C^{\overline
{\sigma}\alpha,\theta\alpha}(\overline{R_{T}^{N}})$.}

\label{s5.7}

Let now $f(x,t)$ be an arbitrary function from $C^{\overline{\sigma}%
\alpha,\theta\alpha}(\overline{R_{T}^{N}})$ with $f(x,0)\equiv0$. Consider the
sequence of it's smooth finite approximations $f_{m,\varepsilon}(x,t)$ that
was constructed in section \ref{s5.4} and which satisfies \eqref{5.18},
\eqref{5.19}. In sections \ref{s5.5} and \ref{s5.6} it was shown that for each
$f_{m,\varepsilon}(x,t)$ there exists a solution $u_{m,\varepsilon}(x,t)\in
C^{\overline{\sigma}(1+\alpha),\theta+\theta\alpha}(\overline{R_{T}^{N}})$ to
problem \eqref{5.1}, \eqref{5.2} (with $u_{0}(x)\equiv0$), which satisfies
estimate \eqref{5.37} that is
\[
\left\vert u_{m,\varepsilon}(x,t)\right\vert _{\overline{R_{T}^{N}}%
}^{(\overline{\sigma}+\overline{\sigma}\alpha,\theta+\theta\alpha)}\leq
C(\theta,\overline{\sigma},T)\left\vert f_{m,\varepsilon}(x,t)\right\vert
_{\overline{R_{T}^{N}}}^{(\overline{\sigma}\alpha,\theta\alpha)}\leq
\]%
\begin{equation}
\leq C(\theta,\overline{\sigma},T)\left\vert f(x,t)\right\vert _{\overline
{R_{T}^{N}}}^{(\overline{\sigma}\alpha,\theta\alpha)}.\label{5.38}%
\end{equation}
Or, more precisely,
\begin{equation}
\left\langle u_{m,\varepsilon}(x,t)\right\rangle _{\overline{R_{T}^{N}}%
}^{(\overline{\sigma}+\overline{\sigma}\alpha,\theta+\theta\alpha)}\leq
C(\theta,\overline{\sigma})\left\vert f(x,t)\right\vert _{\overline{R_{T}^{N}%
}}^{(\overline{\sigma}\alpha,\theta\alpha)},\label{5.38.1}%
\end{equation}%
\begin{equation}
|u_{m,\varepsilon}(x,t)|_{\overline{R_{\widetilde{T}}^{N}}}^{(0)}\leq
C(\theta,\overline{\sigma})\left\vert f(x,t)\right\vert _{\overline
{R_{\widetilde{T}}^{N}}}^{(\overline{\sigma}\alpha,\theta\alpha)}\widetilde
{T}^{\theta+\theta\alpha},\quad\widetilde{T}\leq T.\label{5.38.2}%
\end{equation}
From the know properties of H\"{o}lder spaces (see \cite{23}, the end part of
the proof to Theorem 2.1 on page 404) it follows that there exists a function
$u(x,t)\in C^{\overline{\sigma}(1+\alpha),\theta+\theta\alpha}(\overline
{R_{T}^{N}})$ with
\begin{equation}
\left\vert u(x,t)\right\vert _{\overline{R_{T}^{N}}}^{(\overline{\sigma
}+\overline{\sigma}\alpha,\theta+\theta\alpha)}\leq C(\theta,\overline{\sigma
},T)\left\vert f(x,t)\right\vert _{\overline{R_{T}^{N}}}^{(\overline{\sigma
}\alpha,\theta\alpha)},\label{5.39}%
\end{equation}%
\begin{equation}
\left\langle u(x,t)\right\rangle _{\overline{R_{T}^{N}}}^{(\overline{\sigma
}+\overline{\sigma}\alpha,\theta+\theta\alpha)}\leq C(\theta,\overline{\sigma
})\left\vert f(x,t)\right\vert _{\overline{R_{T}^{N}}}^{(\overline{\sigma
}\alpha,\theta\alpha)},\label{5.39.1}%
\end{equation}%
\begin{equation}
|u(x,t)|_{\overline{R_{\widetilde{T}}^{N}}}^{(0)}\leq C(\theta,\overline
{\sigma})\left\vert f(x,t)\right\vert _{\overline{R_{T}^{N}}}^{(\overline
{\sigma}\alpha,\theta\alpha)}\widetilde{T}^{\theta+\theta\alpha}%
,\quad\widetilde{T}\leq T,\label{5.39.2}%
\end{equation}
and for each compact set
\[
B_{R,T}\equiv\{(x,t)\in\overline{R_{T}^{N}}:|x|\leq R,\quad t\in
\lbrack0,T]\},\quad R>0,
\]
and each $\alpha^{\prime}\in(0,\alpha)$%
\begin{equation}
\left\vert u_{m,\varepsilon}(x,t)-u(x,t)\right\vert _{B_{R,T}}^{(\overline
{\sigma}+\overline{\sigma}\alpha^{\prime},\theta+\theta\alpha^{\prime}%
)}\rightarrow0,\quad\varepsilon\rightarrow0,m\rightarrow\infty,\label{5.40}%
\end{equation}
at least for a subsequence. Relations \eqref{5.38} - \eqref{5.40} permit to
perform a limiting process in equation \eqref{5.1}. Indeed, from this
relations it follows, in particular, that for a fixed $t\in\lbrack0,T]$ the
sequence (subsequence) $\{u_{m,\varepsilon}(x,t)\}\subset S_{\overline{\sigma
}}^{\prime}(R^{N})\subset S^{\prime}(R^{N})$ is included in the class
$S_{\overline{\sigma}}^{\prime}(R^{N})$, which was defined in \eqref{S.7.2}
and this sequence satisfies the conditions of Proposition \ref{PS.1}.
Consequently, this sequens converges to $u(x,t)$ in $S^{\prime}(R^{N})$ for a
fixed $t\in\lbrack0,T]$,
\[
u_{m,\varepsilon}(\cdot,t)\rightarrow_{S^{\prime}(R^{N})}u(\cdot,t),
\]%
\begin{equation}
{\sum_{k=1}^{r}}(-\Delta_{z_{k}})^{\frac{\sigma_{k}}{2}}u_{m,\varepsilon
}(\cdot,t)\rightarrow_{S^{\prime}(R^{N})}{\sum_{k=1}^{r}}(-\Delta_{z_{k}%
})^{\frac{\sigma_{k}}{2}}u(\cdot,t),\quad\varepsilon\rightarrow0,m\rightarrow
\infty,\label{5.41}%
\end{equation}
where $S_{\overline{\sigma}}^{\prime}(R^{N})$ is defined in \eqref{S.7.2}.
Further, Proposition \ref{P3.1}, applied to the sequence $u_{m,\varepsilon
}(x,t)$ uniformly in $x\in R^{N}$, together with \eqref{5.40} means, in
particular, the convergence
\begin{equation}
\max_{x\in B_{R,T}}\left\vert D_{\ast t}^{\theta}u_{m,\varepsilon}%
(x,\cdot)-D_{\ast t}^{\theta}u(x,\cdot)\right\vert _{[0,T]}^{(\theta
\alpha^{\prime})}\rightarrow0,\quad\varepsilon\rightarrow0,m\rightarrow
\infty,\quad x\in R^{N}.\label{5.42}%
\end{equation}
Consequently, the more, for a fixed $t\in\lbrack0,T]$,
\begin{equation}
D_{\ast t}^{\theta}u_{m,\varepsilon}(\cdot,t)\rightarrow_{S^{\prime}(R^{N}%
)}D_{\ast t}^{\theta}u(\cdot,t).\label{5.43}%
\end{equation}
Thus, taking also into account \eqref{5.19}, we can perform the limiting
process in equation \eqref{5.1} for a fixed $t\in\lbrack0,T]$ in space
$S^{\prime}(R^{N})$ for the functions $u_{m,\varepsilon}(x,t)$ and
$f_{m,\varepsilon}(x,t)$. This means that that for a fixed $t\in\lbrack0,T]$
the limit function $u(x,t)$ satisfies the equation in the sense of $S^{\prime
}(R^{N})$. But the function $u(x,t)$ belongs to the space $C^{\overline
{\sigma}(1+\alpha),\theta+\theta\alpha}(\overline{R_{T}^{N}})$ and all the
fractional differential operators in the left hand side of \eqref{5.1} are
defined in the usual sense (see section \ref{Hold}). Therefore $u(x,t)$ is a
solution from the space $C^{\overline{\sigma}(1+\alpha),\theta+\theta\alpha
}(\overline{R_{T}^{N}})$ to problem \eqref{5.1}, \eqref{5.2} with
$u_{0}(x)\equiv0$, which satisfies estimate \eqref{5.39}. Since the general
case with $u_{0}(x)\in C^{\overline{\sigma}(1+\alpha)}(R^{N})$ can be reduced
to zero initial data, as it was shown in section \ref{s5.1}, then we've proved
the following assertion.

\begin{proposition}
\label{P5.1} If conditions \eqref{5.3.1} - \eqref{5.5} are satisfied, then
problem \eqref{5.1}, \eqref{5.2} has a solution $u(x,t)\in C^{\overline
{\sigma}(1+\alpha),\theta+\theta\alpha}(\overline{R_{T}^{N}})$ with estimates
\eqref{5.6} - \eqref{5.6.2}.
\end{proposition}

Formulate now a corollary of this proposition about the existence of a global
solution to problem \eqref{5.1}, \eqref{5.2} for all $t>0$, that is for
$T=\infty$, on infinite interval $(0,\infty)$.

\begin{corollary}
\label{C5.1} Let for problem \eqref{5.1}, \eqref{5.2} with $T=\infty$
conditions \eqref{5.3.1} - \eqref{5.5} are satisfied and the given function
$f(x,t)$ is defined for all $t>0$ in the domain $\overline{R_{\infty}^{N}%
}=R^{N}\times\lbrack0,\infty)$ and it's norm in the space $C^{\overline
{\sigma}\alpha,\theta\alpha}(\overline{R_{\infty}^{N}})$ is finite that is
$\left\vert f(x,t)\right\vert _{\overline{R_{\infty}^{N}}}^{(\overline{\sigma
}\alpha,\theta\alpha)}<\infty$. Then problem \eqref{5.1}, \eqref{5.2} has a
solution $u(x,t)$, which locally in $t$ belongs to the space $C^{\overline
{\sigma}(1+\alpha),\theta+\theta\alpha}(\overline{R_{\infty}^{N} })$ and which
satisfies estimates \eqref{5.6.1}, \eqref{5.6.2} that is
\begin{equation}
\left\langle u\right\rangle _{\overline{R_{\infty}^{N}}}^{(\overline{\sigma
}(1+\alpha),\theta+\theta\alpha)}\leq C(\overline{\sigma},\theta
,\alpha)\left(  |f|_{\overline{R_{\infty}^{N}}}^{(\overline{\sigma}%
\alpha,\theta\alpha)}+|u_{0}|_{R^{N}}^{(\overline{\sigma}(1+\alpha))}\right)
, \label{5.44}%
\end{equation}%
\begin{equation}
\left\vert u\right\vert _{\overline{R_{\widetilde{T}}^{N}}}^{(0)}\leq
C(\overline{\sigma},\theta,\alpha)\left(  |f|_{\overline{R_{\infty}^{N}}%
}^{(\overline{\sigma}\alpha,\theta\alpha)}+|u_{0}|_{R^{N}}^{(\overline{\sigma
}(1+\alpha))}\right)  \widetilde{T}^{\theta+\theta\alpha}+|u_{0}|_{R^{N}%
}^{(0)},\quad\widetilde{T}\leq\infty. \label{5.44.1}%
\end{equation}
Consequently, for an arbitrary finite $T>0$ estimate \eqref{5.6} is also valid.
\end{corollary}

\begin{proof}
From Proposition \ref{P5.1} it follows that for an arbitrary $T>0$ there
exists a solution $u(x,t)\in C^{\overline{\sigma}(1+\alpha),\theta
+\theta\alpha}(\overline{R_{T}^{N}})$ to problem \eqref{5.1}, \eqref{5.2} in
the domain $\overline{R_{T}^{N}}$ with properties in \eqref{5.6} -
\eqref{5.6.2}. Consider a sequence $\{T_{n}\}$, $n=1,2,...$, $T_{n+1}>T_{n}%
>0$, $T_{n}\rightarrow\infty$, $n\rightarrow\infty$. Denote by $u_{n}(x,t)\in
C^{\overline{\sigma}(1+\alpha),\theta+\theta\alpha}(\overline{R_{T_{n}}^{N}})$
a corresponding solution to \eqref{5.1}, \eqref{5.2} in the domain
$\overline{R_{T_{n}}^{N}}$. We stress at this point that we don't have an
assertion on the uniqueness of the solution to \eqref{5.1}, \eqref{5.2} in the
domains $\overline{R_{T_{n} }^{N}}$, therefore we can not assert that for
$m>n$ the solution $u_{n}$ in the domain $\overline{R_{T_{n}}^{N}}$ coincides
with the restriction of the solution $u_{m}$ in the domain $\overline
{R_{T_{m}}^{N}}$ to the domain $\overline{R_{T_{n}}^{N}}$, $T_{n}<T_{m}$.

Consider the first number (index) $n=1$ and consider the solutions $u_{m}$ in
the more wider domains $\overline{R_{T_{m} }^{N}}$, $m>1$. On the ground of
estimates \eqref{5.6.1}, \eqref{5.6.2} we conclude that for the all numbers
$m>1$ we have estimates \eqref{5.6} - \eqref{5.6.2} in the domain
$\overline{R_{T_{1}}^{N}}$ for $u_{m}$
\begin{equation}
|u_{m}|_{\overline{R_{T_{1}}^{N}}}^{(\overline{\sigma}(1+\alpha),\theta
+\theta\alpha)}\leq C(\overline{\sigma},\theta,\alpha,T_{1})\left(
|f|_{\overline{R_{T}^{N}}}^{(\overline{\sigma}\alpha,\theta\alpha)}%
+|u_{0}|_{R^{N}}^{(\overline{\sigma}(1+\alpha))}\right)  ,\quad
m>1,\label{Cor.1}%
\end{equation}
\begin{equation}
\left\langle u_{m}\right\rangle _{\overline{R_{T_{1}}^{N}}}^{(\overline
{\sigma}(1+\alpha),\theta+\theta\alpha)}\leq C(\overline{\sigma},\theta
,\alpha)\left(  |f|_{\overline{R_{T_{1}}^{N}}}^{(\overline{\sigma}%
\alpha,\theta\alpha)}+|u_{0}|_{R^{N}}^{(\overline{\sigma}(1+\alpha))}\right)
,\quad m>1, \label{Cor.2}%
\end{equation}%
\begin{equation}
\left\vert u_{m}\right\vert _{\overline{R_{\widetilde{T}}^{N}}}^{(0)}\leq
C(\overline{\sigma},\theta,\alpha)\left(  |f|_{\overline{R_{T_{1}}^{N}}%
}^{(\overline{\sigma}\alpha,\theta\alpha)}+|u_{0}|_{R^{N}}^{(\overline{\sigma
}(1+\alpha))}\right)  \widetilde{T}^{\theta+\theta\alpha}+|u_{0}|_{R^{N}%
}^{(0)},\quad\widetilde{T}\leq T_{1},m>1. \label{Cor.3}%
\end{equation}
Consequently, as it was in the proof of Proposition \ref{P5.1}, there exists
such a function $u(x,t)\in C^{\overline{\sigma}(1+\alpha),\theta+\theta\alpha
}(\overline{R_{T_{1}}^{N}})$, together with a subsequence $\{u_{m_{k}}%
^{(1)},k=1,2,...\}\subset\{u_{m},m>1\}$, that the sequence $\{u_{m_{k}}%
^{(1)},k=1,2,...\}$ converges to $u(x,t)$ in the sense of \eqref{5.40} -
\eqref{5.43}. And at that $u(x,t)$ is a solution to \eqref{5.1}, \eqref{5.2}
in the domain $\overline{R_{T_{1}}^{N}}$ with estimates \eqref{Cor.1} -
\eqref{Cor.3}. Consider now the sequence $\{u_{m_{k}}^{(1)}\}$ in the domain
$\overline{R_{T_{2}}^{N}}$ for the numbers $k$ with $m_{k}>2$, so that all the
functions $u_{m_{k}}^{(1)}$ are defined in the domain $\overline{R_{T_{2}}%
^{N}} $. Exactly as it was at the first step, we choose from the sequence
$\{u_{m_{k}}^{(1)}\}$ such a subsequence $\{u_{m_{k}}^{(2)},k=1,2,...\}$, that
converges in the sense of \eqref{5.40} - \eqref{5.43} to a solution of
\eqref{5.1}, \eqref{5.2} already in the wider domain $\overline{R_{T_{2}}^{N}%
}$. We denote the obtained solution by the same symbol $u(x,t)$ since on the
included domain $\overline{R_{T_{1}}^{N}}\subset\overline{R_{T_{2}}^{N}}$ the
limit remains the same after going to a subsequence. Proceeding with this
process, we obtain a countable set of included subsequences $\{u_{m_{k}}%
^{(1)}\}\supset\{u_{m_{k}}^{(2)}\}\supset...\supset\{u_{m_{k}}^{(n)}%
\}\supset....$. Each of these subsequences converges in the sense of
\eqref{5.40} - \eqref{5.43} to the same solution $u(x,t)\in C^{\overline
{\sigma}(1+\alpha),\theta+\theta\alpha}$ to \eqref{5.1}, \eqref{5.2} with
properties \eqref{5.44}, \eqref{5.44.1} in the corresponding expanding domains
$\overline{R_{T_{n} }^{N}}$, $n=1,2,...,n,...$. Choosing now from this set of
the subsequences a diagonal subsequence, we see that the last converges,
starting with the corresponding index, on each compact set in $\overline
{R_{\infty}^{N}}\equiv R^{N}\times\lbrack0,\infty)$ to the solution $u(x,t)$
to \eqref{5.1}, \eqref{5.2} for $T=\infty$. And it is readily verified by the
construction that the obtained solution $u(x,t)$ possesses properties
\eqref{5.44}, \eqref{5.44.1} in the whole domain $\overline{R_{\infty}^{N}}$,
which finishes the proof.
\end{proof}

\subsection{Extension of a solution to the whole interval $t\in(0,\infty)$.}

\label{s5.8}

In the present section we consider the issue of an extension of a solution
from the finite time interval to the whole time half-axis. This consideration
plays an auxiliary role. We need such considerations at this stage because we
don't have for a while an assertion on the uniqueness of the solution from the
class $C^{\overline{\sigma} (1+\alpha),\theta+\theta\alpha}(\overline
{R_{T}^{N}})$ on a time interval $[0,T]$. Therefore, obtaining a solution on
the whole time half-axis on the ground of Corollary \ref{C5.1}, we can not
assert for a while that this solution is the extension of a given solution on
a finite interval $[0,T]$. The uniqueness of the solution will be proved in
the following section \ref{s5.9} based on an extension and with an application
of the Fourier transform.

Let the right hand side $f(x,t)$ in \eqref{5.1} is defined (or extended from
the finite interval $[0,T]$) on $t\in(0,\infty)$ in the way that it belongs on
the whole $t\in(0,\infty)$ to the space $C^{\overline{\sigma}\beta,\theta
\beta}(\overline{R_{\infty}^{N}})$ with a possibly reduced exponent $\beta
\leq\alpha$ with $\theta+\theta\beta<1$ (note that $f(x,t)\in C^{\overline
{\sigma}\alpha,\theta\alpha}(\overline{R_{T}^{N}})\subset C^{\overline{\sigma
}\beta,\theta\beta}(\overline{R_{T}^{N}})$). We will show that then a solution
$u(x,t)\in C^{\overline{\sigma}(1+\alpha),\theta+\theta\alpha}(\overline
{R_{T}^{N}})$ to \eqref{5.1}, \eqref{5.2} on a finite interval $[0,T]$ can be
extended up to a solution to \eqref{5.1}, \eqref{5.2} on the whole time
half-axis with estimates in the space $C^{\overline{\sigma}(1+\beta
),\theta+\theta\beta}(\overline{R_{\infty}^{N}} )$ analogously to
\eqref{5.44}, \eqref{5.44.1} from Corollary \ref{C5.1} (with the replacing
$\alpha$ with $\beta$).

We first prove an auxiliary lemma.

\begin{lemma}
\label{L5.1} Let $\theta,\beta>0$ be nonintegers, $\{\theta\}+\theta\beta
\in(0,1)$. Let further a function $u(x,t)\in C^{\overline{\sigma}%
(1+\beta),\theta+\theta\beta}(\overline{R_{T}^{N}})$, $T>0$ and let at $t=0$
it satisfy the conditions
\begin{equation}
u_{t}^{(k)}(x,0)\equiv0,\quad k=0,1,...,[\theta].\label{L.1}%
\end{equation}
Let, besides, this function be extended on the whole $\overline{R_{\infty}%
^{N}}$ to the domain $t>T$ by the Taylor polynomial in $t$ of power $[\theta]$
at the point $t=T$ that is
\begin{equation}
u(x,t)={\sum\limits_{k=0}^{[\theta]}}u_{t}^{(k)}(x,T)\frac{(t-T)^{k}}%
{k!},\quad t\geq T.\label{L.2}%
\end{equation}
Then there exists the bounded for $t>0$ fractional Caputo - Jrbashyan
derivative $D_{\ast t}^{\theta}u(x,t)$ of the extended function and it belongs
to the space $C^{\overline{\sigma}\beta,\theta\beta}(\overline{R_{\infty}^{N}%
})$. Moreover,
\begin{equation}
\left\vert D_{\ast t}^{\theta}u(x,t)\right\vert _{\overline{R_{\infty}^{N}}%
}^{(\overline{\sigma}\beta,\theta\beta)}\leq C(\theta,\beta,T)\left\vert
u(x,t)\right\vert _{\overline{R_{T}^{N}}}^{(\overline{\sigma}(1+\beta
),\theta+\theta\beta)}.\label{L.3}%
\end{equation}

\end{lemma}

\begin{proof}
Note first that after the extension \eqref{L.2} the highest possible integer
derivative $u_{t}^{([\theta])}(x,t)$ is continuous $t=T$ (along with all
derivatives in $t$ of a less order). Besides, in view of the condition
$\{\theta\}+\theta\beta<1$, which means the absence of the next integer
derivative $u_{t}^{([\theta]+1)}(x,t)$, $u_{t}^{([\theta])}(x,t)$ retains in
the whole domain $\overline{R_{\infty}^{N}}$ the smoothness class inherited
from the subdomain $\overline{R_{T}^{N}}$. As it follows from \eqref{2.6},
$u_{t}^{([\theta])}(x,t)\in C^{\frac{\{\theta\}+\theta\beta}{\theta
+\theta\beta}\overline{\sigma}(1+\beta),\{\theta\}+\theta\beta}(\overline
{R_{\infty}^{N}})$. At that, since $u_{t}^{([\theta])}(x,t)\equiv
u_{t}^{([\theta])}(x,T)$ for all $t\geq T$, then
\begin{equation}
\left\vert u_{t}^{([\theta])}(x,t)\right\vert _{\overline{R_{\infty}^{N}}%
}^{\left(  \frac{\{\theta\}+\theta\beta}{\theta+\theta\beta}\overline{\sigma
}(1+\beta),\{\theta\}+\theta\beta\right)  }\leq C(\theta,\beta)\left\vert
u(x,t)\right\vert _{\overline{R_{T}^{N}}}^{(\overline{\sigma}(1+\beta
),\theta+\theta\beta)}.\label{L.4}%
\end{equation}
Make use of the representation for $D_{\ast t}^{\theta}u(x,t)$ from the second
equality in \eqref{1.5}. That is, bearing in mind \eqref{L.1},
\begin{equation}
D_{\ast t}^{\theta}u(x,t)=\frac{1}{\Gamma(1-\theta)}\frac{d}{dt}%
{\int\limits_{0}^{t}}\frac{u_{\tau}^{([\theta])}(x,\tau)d\tau}{(t-\tau
)^{\{\theta\}}}.\label{L.5}%
\end{equation}
From this representation it follows that in fact $D_{\ast t}^{\theta}u(x,t)$
coincides (up to a constant factor) with the Caputo - Jrbashyan derivative in
$t$ of order $\{\theta\}$ of the function $u_{t}^{([\theta])}(x,t)$ (we remind
condition \eqref{L.1}). This derivative is defined correctly since
$u_{t}^{([\theta])}(x,t)$ has the smoothness in $t$ of order $\{\theta
\}+\theta\beta$. Applying Theorem \ref{T3.2}, we conclude that on the time
interval $[0,3T]$ (that is in the domain $\overline{R_{3T}^{N}}$) the
following estimate is valid
\[
\left\vert D_{\ast t}^{\theta}u\right\vert _{\overline{R_{3T}^{N}}%
}^{(\overline{\sigma}\beta,\theta\beta)}\leq C(\theta,\beta,T)\left\vert
u_{t}^{([\theta])}\right\vert _{\overline{R_{3T}^{N}}}^{\left(  \frac
{\{\theta\}+\theta\beta}{\theta+\theta\beta}\overline{\sigma}(1+\beta
),\{\theta\}+\theta\beta\right)  }\leq
\]%
\begin{equation}
\leq C(\theta,\beta,T)\left\vert u(x,t)\right\vert _{\overline{R_{T}^{N}}%
}^{(\overline{\sigma}(1+\beta),\theta+\theta\beta)}.\label{L.6}%
\end{equation}
Let now $t>2T$. Taking into account that $u_{t}^{([\theta])}(x,t)\equiv
u_{t}^{([\theta])}(x,T)$ for $t\geq T$, represent the derivative $D_{\ast
t}^{\theta}u$ from \eqref{L.5} in the form of the sum
\[
D_{\ast t}^{\theta}u(x,t)=C(\theta)\frac{d}{dt}\left[  {\int\limits_{0}^{T}%
}\frac{u_{\tau}^{([\theta])}(x,\tau)d\tau}{(t-\tau)^{\{\theta\}}}%
+{\int\limits_{T}^{t}}\frac{u_{\tau}^{([\theta])}(x,\tau)d\tau}{(t-\tau
)^{\{\theta\}}}\right]  =
\]%
\[
=C(\theta){\int\limits_{0}^{T}}\frac{u_{\tau}^{([\theta])}(x,\tau)d\tau
}{(t-\tau)^{1+\{\theta\}}}+C(\theta)u_{t}^{([\theta])}(x,T)\frac{d}{dt}%
{\int\limits_{T}^{t}}\frac{d\tau}{(t-\tau)^{\{\theta\}}}=
\]%
\[
=C(\theta){\int\limits_{0}^{T}}\frac{u_{\tau}^{([\theta])}(x,\tau)d\tau
}{(t-\tau)^{1+\{\theta\}}}+C(\theta)u_{t}^{([\theta])}(x,T)(t-T)^{-\{\theta
\}}\equiv I_{1}(x,t)+I_{2}(x,t).
\]
Here for $t>2T$ the integral $I_{1}(x,t)$ does not have even weak singularity
and it is, in fact, infinitely differentiable in $t$ with the preserving of
the smoothness in $x$ of $u_{t}^{([\theta])}(x,t)$. The same is valid with
respect to $I_{2}(x,t)$ as well. Therefore, due to the estimate for
$u_{t}^{([\theta])}(x,t)$ in \eqref{L.4}, we the more have for $t>2T$
\begin{equation}
\left\vert D_{\ast t}^{\theta}u\right\vert _{R^{N}\times\lbrack2T,\infty
)}^{(\overline{\sigma}\beta,\theta\beta)}\leq C(\theta,\beta,T)\left\vert
u(x,t)\right\vert _{\overline{R_{T}^{N}}}^{(\overline{\sigma}(1+\beta
),\theta+\theta\beta)}.\label{L.7}%
\end{equation}
Combining \eqref{L.6} and \eqref{L.7} we arrive at the assertion of the lemma.
\end{proof}

\begin{proposition}
\label{P5.2} Let us be given a solution $u(x,t)\in C^{\overline{\sigma
}(1+\alpha),\theta+\theta\alpha}(\overline{R_{T}^{N}})$ to \eqref{5.1},
\eqref{5.2}, which is defined in $\overline{R_{T} ^{N}}$. Let also the right
hand side $f(x,t)$ in \eqref{5.1} be defined (or extended from $[0,T]$) for
all $t>0$ in the domain $\overline{R_{\infty}^{N}}$ and let $f(x,t)\in
C^{\overline{\sigma}\beta,\theta\beta}(\overline{R_{\infty}^{N}})$  with such
$\beta\leq\alpha$ that $\theta+\theta\beta<1$. Then there exists a solution to
\eqref{5.1}, \eqref{5.2} in the domain $\overline{R_{\infty}^{N}}$ (which is
denoted by the same symbol $u(x,t)$). This solution belongs to the space
$C^{\overline{\sigma}(1+\beta),\theta+\theta\beta}(\overline{R_{\widetilde{T}%
}^{N}})$ for any $\widetilde{T}>0$, it coincides with $u(x,t)$ in
$\overline{R_{T}^{N}}$, and it obeys in $\overline{R_{\infty}^{N}}$ to the
analogous to \eqref{5.44}, \eqref{5.44.1} estimates
\begin{equation}
\left\langle u\right\rangle _{\overline{R_{\infty}^{N}}}^{(\overline{\sigma
}(1+\beta),\theta+\theta\beta)}\leq C(\overline{\sigma},\theta,\beta)\left(
|f|_{\overline{R_{\infty}^{N}}}^{(\overline{\sigma}\beta,\theta\beta
)}+\left\vert u\right\vert _{\overline{R_{T}^{N}}}^{(\overline{\sigma
}+\overline{\sigma}\beta,\theta+\theta\beta)}\right)  , \label{Ex.1}%
\end{equation}%
\begin{equation}
\left\vert u\right\vert _{\overline{R_{\widetilde{T}}^{N}}}^{(0)}\leq
C(\overline{\sigma},\theta,\alpha)\left(  |f|_{\overline{R_{\infty}^{N}}%
}^{(\overline{\sigma}\beta,\theta\beta)}+\left\vert u\right\vert
_{\overline{R_{T}^{N}}}^{(\overline{\sigma}+\overline{\sigma}\beta
,\theta+\theta\beta)}\right)  \widetilde{T}^{\theta+\theta\beta}+\left\vert
u\right\vert _{\overline{R_{T}^{N}}}^{(\overline{\sigma}+\overline{\sigma
}\beta,\theta+\theta\beta)},\quad\widetilde{T}\leq\infty. \label{Ex.2}%
\end{equation}

\end{proposition}

\begin{proof}
Note first that, as it was shown in section \ref{s5.1}, we can assume without
loss of generality that we are given zero initial data in \eqref{5.2} that is
$u(x,0)=u_{0}(x)\equiv0$. Consider in the domain $R_{T,\infty}^{N}\equiv
R^{N}\times\lbrack T,\infty)$ the following Cauchy problem for the unknown
function, which is denoted by the same symbol $u(x,t)$,
\begin{equation}
Lu(x,t)\equiv D_{T,\ast t}^{\theta}u(x,t)+{\sum_{k=1}^{r}}(-\Delta_{z_{k}%
})^{\frac{\sigma_{k}}{2}}u(x,t)=f(x,t)-D_{\ast t}^{\theta}\widetilde
{u}(x,t),\quad(x,t)\in R_{T,T_{1}}^{N},\label{5.77}%
\end{equation}%
\begin{equation}
u(x,T)=u_{T}(x),\quad x\in R^{N}.\label{5.78}%
\end{equation}
Here $D_{T,\ast t}^{\theta}u(x,t)$ is the Kaputo - Jrbashyan derivative of
$u(x,t)$ with the starting point $t=T$
\[
D_{T,\ast t}^{\theta}u(x,t)\equiv\frac{1}{\Gamma(1-\theta)}\frac{d}{dt}%
{\int\limits_{T}^{t}}\frac{\left[  u(x,\tau)-u(x,T)\right]  d\tau}%
{(t-\tau)^{\theta}}=
\]%
\begin{equation}
=\frac{1}{\Gamma(1-\theta)}\frac{d}{dt}{\int\limits_{T}^{t}}\frac{\left[
u(x,\tau)-u_{T}(x)\right]  d\tau}{(t-\tau)^{\theta}},\label{5.79}%
\end{equation}
the function $\widetilde{u}(x,t)$ is the extension of the given in
$\overline{R_{T}^{N}}$ solution $u(x,t)$ to the domain $t>T$ by the Taylor
polynomial of zero order in $t$ (that is by the corresponding constant in
$t$)
\begin{equation}
\widetilde{u}(x,t)\equiv%
%TCIMACRO{\QDATOPD{\{}{.}{u(x,t),\quad t\in\lbrack0,T],}{u(x,T),\quad t\geq
%T,}}%
%BeginExpansion
\genfrac{\{}{.}{0pt}{0}{u(x,t),\quad t\in\lbrack0,T],}{u(x,T),\quad t\geq T,}%
%EndExpansion
\label{5.80}%
\end{equation}
and $D_{\ast t}^{\theta}\widetilde{u}(x,t)$ is the Caputo - Jrbashyan
derivative of $\widetilde{u}(x,t)$ with the starting point $t=0$
\begin{equation}
D_{\ast t}^{\theta}\widetilde{u}(x,t)\equiv\frac{1}{\Gamma(1-\theta)}\frac
{d}{dt}{\int\limits_{0}^{t}}\frac{\widetilde{u}(x,\tau)d\tau}{(t-\tau
)^{\theta}},\label{5.81}%
\end{equation}
since by assumption $\widetilde{u}(x,0)=u(x,0)\equiv0$.

We are going to apply to problem \eqref{5.77}, \eqref{5.78} Corollary
\ref{C5.1} on the solvability. Note first that the difference of the starting
point $t=T$ from the case $t=0$ in Corollary \ref{C5.1} is not essential. This
case is trivially reduced to the starting point $t=0$ by the time change
$t=\overline{t}+T$. Therefore, it is enough to verify the conditions of
Corollary \ref{C5.1} for problem \eqref{5.77}, \eqref{5.78}.

Firstly, $f(x,t)\in C^{\overline{\sigma}\beta,\theta\beta}(\overline
{R_{\infty}^{N}})$ by assumption and the more $f(x,t)\in C^{\overline{\sigma
}\beta,\theta\beta}(\overline{R_{T,\infty}^{N}})$. Further, the function
$\widetilde{u}(x,t)$ from \eqref{5.80} belongs to the space $C^{\overline
{\sigma}(1+\alpha),\theta+\theta\alpha}(\overline{R_{T}^{N}})\subset
C^{\overline{\sigma}(1+\beta),\theta+\theta\beta}(\overline{R_{T}^{N}})$.
Therefore, from Lemma \ref{L5.1} it follows that
\begin{equation}
\left\vert D_{\ast t}^{\theta}\widetilde{u}(x,t)\right\vert _{\overline
{R_{\infty}^{N}}}^{(\overline{\sigma}\beta,\theta\beta)}\leq C \left\vert
u(x,t)\right\vert _{\overline{R_{T}^{N}}}^{(\overline{\sigma}+\overline
{\sigma}\beta,\theta+\theta\beta)}.\label{5.83}%
\end{equation}
At last, verify compatibility condition \eqref{5.5} at $t=T$. Indeed,
according with this condition we must have the equality
\[
{\sum_{k=1}^{r}}(-\Delta_{z_{k}})^{\frac{\sigma_{k}}{2}}u(x,T)=f(x,T)-D_{\ast
t}^{\theta}\widetilde{u}(x,t)|_{t=T}.
\]
But at $t=T$ we have by definition $D_{\ast t}^{\theta}\widetilde
{u}(x,t)|_{t=T}=D_{\ast t}^{\theta}u(x,t)|_{t=T}$ and the equality holds in
view of the fact that $u(x,t)$ is a solution to \eqref{5.1} in $\overline
{R_{T}^{N}}$. Therefore, on the ground of Corollary \ref{C5.1}, we infer that
there exists a solution $u(x,t)$ to problem \eqref{5.77}, \eqref{5.78} with
estimates \eqref{Ex.1}, \eqref{Ex.2}.

Show now that the originally defined in $\overline{R_{T}^{N}}$ function
$u(x,t)$, which is defined already in the whole $\overline{R_{\infty}%
^{N}\text{ }}$, satisfies equation \eqref{5.1} totally in $\overline
{R_{\infty} ^{N}\text{ }}$. Note that in view of initial condition
\eqref{5.78} and the condition $\theta+\theta\beta<1$ the function $u(x,t)\in
C^{\overline{\sigma}(1+\beta),\theta+\theta\beta}(\overline{R_{\widetilde{T}%
}^{N}})$ for an arbitrary $\widetilde{T}>0$. Further, for $t \leq T$ the
function $u(x,t)$ satisfies equation \eqref{5.1} and initial condition
\eqref{5.2} by assumption. Let now $t \geq T$. Moving the expression $D_{\ast
t}^{\theta}\widetilde{u}(x,t)$ from the right hand side of equation
\eqref{5.77} to it's left hand side and taking into account \eqref{5.79} and
definition \eqref{5.80}, we see that for $t \geq T$
\[
D_{T,\ast t}^{\theta}u(x,t)+D_{\ast t}^{\theta}\widetilde{u}(x,t)=D_{\ast
t}^{\theta}u(x,t),
\]
that is equation \eqref{5.1} is satisfied for all $t>T$. Since $u(x,t)\in
C^{\overline{\sigma}(1+\beta),\theta+\theta\beta}(\overline{R_{\widetilde{T}%
}^{N}})$ for an arbitrary $\widetilde{T}>0$ and the operators $D_{\ast
t}^{\theta}$, $(-\Delta_{z_{k}})^{\frac{\sigma_{k}}{2}}$ in \eqref{5.1} are
continuous from the pointed space to the space $C^{\overline{\sigma}%
\beta,\theta\beta}(\overline{R_{\widetilde{T}}^{N}})$, equation \eqref{5.1} is
satisfied by continuity at $t=T$ as well. And this finishes the proof.
\end{proof}

\subsection{Uniqueness of solution and finishing the proof of Theorem
\ref{T5.1}.}

\label{s5.9}

On this step we prove the uniqueness of the obtained in Proposition \ref{P5.1}
solution from the class $C^{\overline{\sigma} (1+\alpha),\theta+\theta\alpha
}(\overline{R_{T}^{N}})$.

\begin{lemma}
\label{L5.2} Let a function $u(x,t)\in C^{\overline{\sigma}(1+\alpha
),\theta+\theta\alpha}(\overline{R_{T}^{N}})$ satisfy the homogeneous problem
\eqref{5.1}, \eqref{5.2} with $f(x,t)\equiv0$ and $u_{0}(x)\equiv0$. Then
$u(x,t)\equiv0$ in $\overline{R_{T}^{N}}$.
\end{lemma}

\begin{proof}
Make use of Proposition \ref{P5.2} and extend the given solution $u(x,t)$ for
all $t>0$ to a solution (with the same symbol $u(x,t)$) to the homogeneous
problem \eqref{5.1}, \eqref{5.2} in the whole domain $\overline{R_{\infty}%
^{N}}$. It is possible since in our case $f(x,t)\equiv0$ on $t\in\lbrack0,T]$
and it can be extended by identical zero to $t\in\lbrack0,\infty)$. The
obtained solution $u(x,t)$ can be also extended by identical zero in the
domain $t<0$ with the preserving of the smoothness (in view of zero initial
data and in view of \eqref{3.4.0}). Besides, this extended solution has, in
view of \eqref{Ex.2}, a power growth for $t\rightarrow\infty$. Consequently,
we can consider this function as an element of the space of distributions
$\Phi^{\prime}(R^{N+1})$ from section \ref{sfi}. Finally, the extended by zero
in the domain $t<0$ function $u(x,t)$ satisfies equation \eqref{5.1} on the
whole space $R^{N+1}$.

Apply the Fourier transform according to formulas \eqref{fi.3}, \eqref{fi.4}
to equation \eqref{5.1} for the function $u(x,t)$ under consideration. We
obtain
\[
\left[  (i\xi_{0})^{\theta}+{\sum_{k=1}^{r}}|\zeta_{k}|^{\sigma_{k}}\right]
\widehat{u}(\xi,\xi_{0})=0,\quad(\xi,\xi_{0})\in R^{N+1}.
\]
From this it follows that $\widehat{u}(\xi,\xi_{0})$ is a distribution from
$\Psi^{\prime}(R^{N+1})$ with the support at the point $(\xi,\xi_{0})=(0,0)$
that is, according to the definition of the space $\Psi^{\prime}(R^{N+1})$,
this distribution represents the zero class of $\Psi^{\prime}(R^{N+1})$. But
this means that the function $u(x,t)$ itself represents the zero class of the
space $\Phi^{\prime}(R^{N+1})$ that is $u(x,t)$ is a polynomial. But since
$u(x,t)$ identically vanishes for $t<0$, then $u(x,t)\equiv0$ in $R^{N+1}$
and, in particular, in $\overline{R_{T}^{N}}$. This finishes the proof.
\end{proof}

On the ground of the proved lemma the uniqueness of solution to problem
\eqref{5.1}, \eqref{5.2} in the class $C^{\overline{\sigma} (1+\alpha
),\theta+\theta\alpha}(\overline{R_{T}^{N}})$ follows now from linearity of
problem \eqref{5.1}, \eqref{5.2}.

Thus, Proposition \ref{P5.1} together with Lemma \ref{L5.2} finish the proof
of Theorem \ref{T5.1}.

\section{A Cauchy problem for equation \eqref{1.1} in the case of the integer
$\theta=1$,$\theta\alpha=\alpha\in(0,1)$.}

\label{s6}

In this section we consider solvability of problem \eqref{1.1}, \eqref{1.2} in
the anisotropic H\"{o}lder spaces $C^{\overline{\sigma}(1+\alpha),1+\alpha
}(\overline{R_{T}^{N}})$ from \eqref{2.8} in the case of the integer
derivative in time of order $\theta=1$. Consider the following Cauchy problem
for the defined in $\overline{R_{T}^{N}}$ unknown function $u(x,t)$
\begin{equation}
Lu(x,t)\equiv u_{t}(x,t)+{\sum_{k=1}^{r}}(-\Delta_{z_{k}})^{\frac{\sigma_{k}%
}{2}}u(x,t)=f(x,t),\quad(x,t)\in R_{T}^{N},\label{6.1}%
\end{equation}%
\begin{equation}
u(x,0)=u_{0}(x),\quad x\in R^{N},\label{6.2}%
\end{equation}
where the given functions belongs to the spaces
\begin{equation}
f(x,t)\in C^{\overline{\sigma}\alpha,\alpha}(\overline{R_{T}^{N}}),\quad
u_{0}(x)\in C^{\overline{\sigma}(1+\alpha)}(R^{N}).\label{6.3}%
\end{equation}
In the present section we first suppose that $\alpha$ is sufficiently small
\begin{equation}
\quad\alpha\in(0,1),\quad0<\alpha\leq\alpha_{0}=\alpha_{0}(\overline{\sigma
},\{N_{k}\}),\label{6.4}%
\end{equation}
where the particular value of $\alpha_{0}\in(0,1)$ will be given below. On the
later steps in this section the second restriction in \eqref{6.4} will be
removed. We stress also that for an integer derivative in $t$ of order
$\theta$ a compatibility condition of the kind \eqref{5.5} is not necessary.

The main assertion of the present section is as follows.

\begin{theorem}
\label{T6.1} If $\alpha\in(0,1)$ and conditions \eqref{6.3} are satisfied,
then problem \eqref{6.1}, \eqref{6.2} has the unique solution $u(x,t)\in
C^{\overline{\sigma}(1+\alpha),1+\alpha}(\overline{R_{T}^{N}})$ with the
estimates
\begin{equation}
|u|_{\overline{R_{T}^{N}}}^{(\overline{\sigma}(1+\alpha),1+\alpha)}\leq
C(\overline{\sigma},\alpha,T)\left(  |f|_{\overline{R_{T}^{N}}}^{(\overline
{\sigma}\alpha,\alpha)}+|u_{0}|_{R^{N}}^{(\overline{\sigma}(1+\alpha
))}\right)  , \label{6.6}%
\end{equation}%
\begin{equation}
\left\langle u\right\rangle _{\overline{R_{T}^{N}}}^{(\overline{\sigma
}(1+\alpha),1+\alpha)}\leq C(\overline{\sigma},\alpha)\left(  |f|_{\overline
{R_{T}^{N}}}^{(\overline{\sigma}\alpha,\alpha)}+|u_{0}|_{R^{N}}^{(\overline
{\sigma}(1+\alpha))}\right)  , \label{6.6.1}%
\end{equation}%
\begin{equation}
\left\vert u\right\vert _{\overline{R_{\widetilde{T}}^{N}}}^{(0)}\leq
C(\overline{\sigma},\alpha)\left(  |f|_{\overline{R_{T}^{N}}}^{(\overline
{\sigma}\alpha,\alpha)}+|u_{0}|_{R^{N}}^{(\overline{\sigma}(1+\alpha
))}\right)  (1 + \widetilde{T}^{1+\alpha} ) + |u_{0}|_{R^{N}}^{(0)}%
,\quad\widetilde{T}\leq T. \label{6.6.2}%
\end{equation}

\end{theorem}

Analogously to the previous section, we give the proof of the above theorem
into several steps, which is the subsequent content of the present section.

\subsection{Reduction of the problem to zero initial conditions.}

\label{s6.1}

Similar to the previous section we can assume from the very beginning zero
initial data
\begin{equation}
u_{0}(x)\equiv0,\quad x\in R^{N},\label{6.7}%
\end{equation}
which is achieved by the change of the unknown in problem \eqref{6.1},
\eqref{6.2}
\[
u(x,t)\rightarrow v(x,t)=u(x,t)-u_{0}(x).
\]
The new unknown function $v(x,t)$ satisfies initial condition \eqref{6.2} and
the equation
\[
v_{t}+{\sum_{k=1}^{r}}(-\Delta_{z_{k}})^{\frac{\sigma_{k}}{2}}v=\widetilde
{f(}x,t)\equiv f(x,t)-{\sum_{k=1}^{r}}(-\Delta_{z_{k}})^{\frac{\sigma_{k}}{2}%
}u_{0}(x).
\]
From Proposition \ref{P3.6} it follows that $\widetilde{f(}x,t)\in
C^{\overline{\sigma}\alpha,\alpha}(\overline{R_{T}^{N}})$ like the original
function $f(x,t)$, and besides
\begin{equation}
|\widetilde{f}|_{\overline{R_{T}^{N}}}^{(\overline{\sigma}\alpha,\theta
\alpha)}\leq C(\overline{\sigma},\alpha)\left(  |f|_{\overline{R_{T}^{N}}%
}^{(\overline{\sigma}\alpha,\theta\alpha)}+|u_{0}|_{R^{N}}^{(\overline{\sigma
}(1+\alpha))}\right)  .\label{6.8}%
\end{equation}
We stress that in contrast to the case of fractional time derivative
$\theta\in(0,1)$ in previous section, now we do not have a compatibility
condition of the kind \eqref{5.5}. Therefore we can not assert the validity of
the condition $\widetilde{f(}x,0)\equiv0$. This fact is the main difference
from the previous case, since this does not permit to extend $\widetilde
{f(}x,t)$ by zero for $t<0$ with the preserving of the smoothness in $t$.
Therefore in the present section, to estimate the highest seminorm in the
corresponding H\"{o}lder space, we make use Theorem \ref{T4.2} on Fourier
multipliers instead of Theorem \ref{T4.1}. Theorem \ref{T4.2} permits us to
perform the necessary estimates without assuming the smoothness of
$\widetilde{f(}x,t)$ for all $t$. But this causes some technical restrictions
for the H\"{o}lder exponent $\alpha$ (see the last condition in \eqref{6.4})
and therefore some additional consideration then needed to remove these restrictions.

Thus, below we assume zero initial conditions that is \eqref{6.7}.

\begin{remark}
\label{R6.1} Here the following should be noted. There exists a method of the
extension of the initial data $u_{0}(x)\in C^{\overline{\sigma}(1+\alpha
)}(R^{N})$ to the domain $t>0$ up to a function $U_{0}(x,t)$ from the space
$C^{\overline{\sigma}(1+\alpha),1+\alpha}(\overline{R_{T}^{N}})$ that not only
retains the initial value $U_{0}(x,0)=u_{0} (x)$ but also the initial value of
the derivative in time of the solution itself $\partial U_{0}(x,0)/\partial
t=u_{1} (x)\equiv\partial u(x,0)/\partial t$. This last is prescribed by the
equation \eqref{6.1} with taking into account condition \eqref{6.2}. If one
has such a function $U_{0}(x,t)$, one can make the change of the unknown
$u(x,t)\rightarrow u(x,t)-U_{0}(x,t)$ and the original problem can be reduced
to the case, when one has not only zero initial data but also $\widetilde
{f(}x,0)\equiv0$. Such method is described in, for example, \cite{Sol1}.
However, according to this method the auxiliary function $U_{0}(x,t)$ is
constructed as a solution to the simplest model parabolic equation with
integer derivatives and with the initial data $u_{0}(x)$ and $u_{1}(x)$. The
properties of Cauchy problems for the mentioned simplest model equation (heat
equation, for example) are supposed to be well known in advance. The
application of the method from \cite{Sol1} permits to obtain such extension
$U_{0}(x,t)$ for our case as well, but only when all the orders of fractional
differentiation $\sigma_{k}$ in $x$ are rational. But for the case when at
least on of the orders $\sigma_{k}$ is irrational this method does not appear
to be applicable. The reason is that equation \eqref{6.1} is precisely that
simplest model equation for the case of fractional differentiation and our
purpose is precisely to describe it's properties in the space $C^{\overline
{\sigma}(1+\alpha),1+\alpha}(\overline{R_{T}^{N}})$. This causes the necessity
to apply below Theorem \ref{T4.2} on Fourier multipliers in H\"{o}lder spaces
with some partial H\"{o}lder regularity.
\end{remark}

\subsection{Extension and smoothing of data, formulation of the problem in
$R^{N+1}=R^{N}\times(-\infty,\infty)$.}

\label{s6.2}

In this section we, analogously to sections \ref{s5.2}, \ref{s5.3}, assume
zero initial data $u(x,0)\equiv0$ and assume that the righthand side $f(x,t)$
is extended in $t$ to the domain $t>T$ up to a finite in $t$ function with the
class preservation. Besides, we assume that $f(x,t)$ is extended to $t<0$ by
identical zero. As a result, we have a finite in $t$ function (which is
denoted by the same symbol $f(x,t)$) and this function is defined in the whole
$R^{N+1}$. The extended $f(x,t)$ has (in general) a jump at $t=0$ and it has
the properties
\begin{equation}
|f|_{\overline{R_{\infty}^{N}}}^{(\overline{\sigma}\alpha,\alpha)}\leq
C(\overline{\sigma},\alpha)|f|_{\overline{R_{T}^{N}}}^{(\overline{\sigma
}\alpha,\alpha)},\quad|f|_{R^{N+1}}^{(0)}\leq C|f|_{\overline{R_{T}^{N}}%
}^{(0)},\quad f(x,t)\equiv0,t<0. \label{6.9}%
\end{equation}
That is $f(x,t)$ has a controlled H\"{o}lder norm in the upper half-space only
but it has a controlled bound in the whole space.

Further, as it was in section \ref{s5.2}, we assume that $f(x,t)$ is an
integer derivative in $t$ of a sufficiently high order $n$ of a finite in $t$
function $F(x,t)$ with
\[
F(x,t),\frac{\partial^{n}F(x,t)}{\partial t^{n}}\in C^{\overline{\sigma}%
\alpha,\alpha}(\overline{R_{\infty}^{N}}),\quad\frac{\partial^{n-1}%
F(x,t)}{\partial t^{n-1}}\in C^{\overline{\sigma}\alpha,\alpha}(R^{N+1}),
\]%
\begin{equation}
F(x,t)\equiv0,t<0,\quad f(x,t)=\frac{\partial^{n}F(x,t)}{\partial t^{n}%
},(x,t)\in R^{N+1}. \label{6.10}%
\end{equation}
The way to construct such a function $F(x,t)$ coincides with that from section
\ref{s5.2}.

We apply also to the functions $F(x,t)$ and $f(x,t)$ the smoothing by the
convolution with the mollifier $\omega_{\varepsilon}(x,t)$ from \eqref{5.12},
and after that we cut them off up to finite in $x$ functions $F_{m,\varepsilon
}(x,t)$ and $f_{m,\varepsilon}(x,t)$ from $C^{\infty}(R^{N+1})$. This process
is almost replicates the process from \eqref{5.13}, \eqref{5.14} but with a
tiny difference. In the present situation we put
\begin{equation}
F_{m,\varepsilon}(x,t)\equiv\zeta(\frac{x}{m}){\int\limits_{R^{N+1}}}%
\omega_{\varepsilon}(x-\xi,t-\tau)F(\xi,\tau+\varepsilon)dxdt,\quad
m=1,2,...,\label{6.10.1}%
\end{equation}
and correspondingly
\begin{equation}
f_{m,\varepsilon}(x,t)\equiv\zeta(\frac{x}{m}){\int\limits_{R^{N+1}}}%
\omega_{\varepsilon}(x-\xi,t-\tau)f(\xi,\tau+\varepsilon)dxdt,\quad
m=1,2,....\label{6.10.2}%
\end{equation}
We stress that the difference is that the density $F(\xi,\tau+\varepsilon)$
and $f(\xi,\tau+\varepsilon)$ of the convolutions are shifted by $\varepsilon$
in the negative direction of the $t$-axis, in contrast to \eqref{5.14}. The
aim of this is to move the possible jump of $f(x,t)$ from the point $t=0$ to
the point $t=-\varepsilon$ and thus to provide the uniform regularity of the
mollified functions in the half-space $\{t\geq0\}$ for all $\varepsilon$. At
that, identically to \eqref{5.17},
\begin{equation}
\frac{\partial^{n}F_{m,\varepsilon}(x,t)}{\partial t^{n}}\equiv
f_{m,\varepsilon}(x,t),\quad(x,t)\in R^{N+1}.\label{6.11}%
\end{equation}

Remind that the function $f(x,t)$ has, in general, a jump at $t=0$, and
therefore the smoothed functions $f_{m,\varepsilon}(x,t)$ possess more weak
uniform properties than those in \eqref{5.18}, \eqref{5.19}. Namely,
\begin{equation}
|f_{m,\varepsilon}|_{\overline{R_{\infty}^{N}}}^{(\sigma\alpha,\alpha)}\leq
C|f(x,t+\varepsilon)|_{R^{N}\times\lbrack-\varepsilon,\infty)}^{(\sigma
\alpha,\alpha)}=C|f(x,t)|_{\overline{R_{\infty}^{N}}}^{(\sigma\alpha,\alpha
)}\leq C(\overline{\sigma},\alpha)|f|_{\overline{R_{T}^{N}}}^{(\overline
{\sigma}\alpha,\alpha)} \label{6.12}%
\end{equation}
and moreover, for each compact $P_{R}^{+}\equiv P_{R}\cap\{t\geq0\}$, where
$P_{R}$ is defined in \eqref{5.15.1}, and for each $\alpha^{\prime}<\alpha$
\begin{equation}
\left\vert f_{m,\varepsilon}(x,t)-f(x,t)\right\vert _{P_{R}^{+}}%
^{(\sigma\alpha^{\prime},\alpha^{\prime})}\rightarrow0,\quad\varepsilon
\rightarrow0,m\rightarrow\infty. \label{6.13}%
\end{equation}
Besides, in the whole space $R^{N+1}$ the functions $f_{m,\varepsilon}(x,t)$
have the following properties ($\varepsilon\in(0,1)$)
\[
f_{m,\varepsilon}(x,t)\in C^{\infty}(R^{N+1}),
\]
\begin{equation}
f_{m,\varepsilon}(x,t)\equiv0,t<-2\varepsilon,\quad|f_{m,\varepsilon
}(x,t)|_{R^{N+1}}^{(0)}\leq C|f(x,t)|_{\overline{R_{T}^{N}}}^{(0)},
\label{6.14}%
\end{equation}%
\[
\left\langle f_{m,\varepsilon}\right\rangle _{x,R^{N+1}}^{(\sigma\alpha)}\leq
C|f(x,t-\varepsilon)|_{R^{N}\times\lbrack-\varepsilon,\infty)}^{(\sigma
\alpha,\alpha)} =
\]
\begin{equation}
=C|f(x,t)|_{\overline{R_{\infty}^{N}}}^{(\sigma\alpha,\alpha)}\leq
C(\overline{\sigma},\alpha)|f|_{\overline{R_{T}^{N}}}^{(\overline{\sigma
}\alpha,\alpha)}, \label{6.15}%
\end{equation}%
\[
\left\langle f_{m,\varepsilon}\right\rangle _{t,R^{N+1}}^{(\alpha)}\leq
C\varepsilon^{-\alpha}|f(x,t-\varepsilon)|_{R^{N}\times\lbrack-\varepsilon
,\infty)}^{(\sigma\alpha,\alpha)} =
\]
\begin{equation}
=C\varepsilon^{-\alpha}|f|_{\overline{R_{\infty}^{N}}}^{(\sigma\alpha,\alpha
)}\leq C(\overline{\sigma},\alpha)\varepsilon^{-\alpha}|f|_{\overline
{R_{T}^{N}}}^{(\overline{\sigma}\alpha,\alpha)}. \label{6.16}%
\end{equation}
That is the finite functions $f_{m,\varepsilon}$ have uniformly bounded
maximum of modulus over the whole space $R^{N+1}$ and the uniformly bounded
H\"{o}lder seminorms in $x$ over the whole space $R^{N+1}$, as it is stated in
\eqref{6.14}, \eqref{6.15}. This facts easy follows from the definitions of
$f_{m,\varepsilon}$ with taking into account that the extended by zero to
$t<0$ original function $f(x,t)$ does not have jumps in directions of $x$ -
variables. The last property in \eqref{6.16} asserts that the H\"{o}lder
seminorms in $t$ of the smoothed functions $f_{m,\varepsilon}$ over the whole
$R^{N+1}$ have the order $\varepsilon^{-\alpha}$ for $\varepsilon\rightarrow
0$. This follows from from definitions \eqref{5.12}, \eqref{6.10.2} and it is
because the extended function $f(x,t)$ may have a jump in $t$ at $t=0$. For
the completeness we present a simple auxiliary lemma, which is a direct ground
for \eqref{6.16}.

\begin{lemma}
\label{L6.1} Let $f(x,t)\in L_{\infty}(R^{N+1})$ and let $\omega_{\varepsilon
}(x,t)$ be defined in \eqref{5.12}. Then
\[
\left\langle f_{\varepsilon}\right\rangle _{t,R^{N+1}}^{(\alpha)}\leq C
\varepsilon^{-\alpha}\left\Vert f\right\Vert _{L_{\infty}(R^{N+1})},
\]
where
\[
f_{\varepsilon}(x,t)\equiv{\int\limits_{R^{N+1}}}\omega_{\varepsilon}%
(x-\xi,t-\tau)f(\xi,\tau)dxdt.
\]

\end{lemma}

\begin{proof}
On the ground of the definition of $f_{\varepsilon}(x,t)$, consider the
difference
\[
D(\Delta t)\equiv\left\vert \frac{f_{\varepsilon}(x,t+\Delta t)-f_{\varepsilon
}(x,t)}{(\Delta t)^{\alpha}}\right\vert =
\]%
\[
=\left\vert {\int\limits_{R^{N+1}}}\frac{\omega_{\varepsilon}(x-\xi,t+\Delta
t-\tau)-\omega_{\varepsilon}(x-\xi,t-\tau)}{(\Delta t)^{\alpha}}f(\xi
,\tau)dxdt\right\vert ,\quad\Delta t>0.
\]
Since $\omega(x,t)\in C^{\infty}(R^{N+1})$, then $\left\langle \omega
(x,t)\right\rangle _{t,R^{N+1}}^{(\alpha)}=C<\infty$. And at the same time, as
it can be easily verified by scaling of the variables, $\left\langle
\omega(\frac{x}{\varepsilon},\frac{t}{\varepsilon})\right\rangle _{t,R^{N+1}%
}^{(\alpha)}=C\varepsilon^{-\alpha}$. Besides, the support of the function
$\omega_{\varepsilon}(x,t)$ is included in the set $\left\{  |x|+|t|\leq
\varepsilon\right\}  $. Consider the following two cases.

Let first $\Delta t<\varepsilon$. Then the support of the fraction under the
integral is included in $\left\{  (\xi,\tau):|\xi-x|+|t-\tau|\leq
2\varepsilon\right\}  $. Consequently,
\[
D(\Delta t)\leq\left\langle \omega_{\varepsilon}(x,t)\right\rangle
_{t,R^{N+1}}^{(\alpha)}{\int\limits_{\left\{  (\xi,\tau):|\xi-x|+|t-\tau
|\leq2\varepsilon\right\}  }}\left\vert f(\xi,\tau)\right\vert dxdt\leq
\]%
\[
\leq C \varepsilon^{-\alpha-N-1}\left\Vert f\right\Vert _{L_{\infty
}(R^{N+1})}{\int\limits_{\left\{  (\xi,\tau):|\xi-x|+|t-\tau|\leq
2\varepsilon\right\}  }}dxdt= C \varepsilon^{-\alpha}\left\Vert
f\right\Vert _{L_{\infty}(R^{N+1})}.
\]

If now $\Delta t\geq\varepsilon$, then
\[
D(\Delta t)\leq\varepsilon^{-\alpha}\left(  \left\vert f_{\varepsilon
}(x,t+\Delta t)\right\vert +\left\vert f_{\varepsilon}(x,t)\right\vert
\right)  \leq2\varepsilon^{-\alpha}\left\Vert f\right\Vert _{L_{\infty
}(R^{N+1})},
\]
which follows from properties of $\omega_{\varepsilon}(x,t)$ in \eqref{5.12}.

This completes the proof.
\end{proof}

Eventually, similar to section \ref{s5.3}, instead of problem \eqref{6.1},
\eqref{6.2} consider now the follows sequence of problems. It is required to
find such a function $u(x,t)\in C^{\overline{\sigma}(1+\alpha),1+\alpha
}(R^{N+1})$ that satisfies in the whole $R^{N+1}$ equation \eqref{6.1} with
$f_{m,\varepsilon}(x,t)$ instead of $f(x,t)$ ($f_{m,\varepsilon}(x,t)$ is
defined in \eqref{6.10.2}),
\begin{equation}
Lu(x,t)\equiv u_{t}(x,t)+{\sum_{k=1}^{r}}(-\Delta_{z_{k}})^{\frac{\sigma_{k}%
}{2}}u(x,t)=f_{m,\varepsilon}(x,t),\quad(x,t)\in R^{N+1},\label{6.17}%
\end{equation}
and the condition
\begin{equation}
u(x,t)\equiv0,\quad t<-2\varepsilon.\label{6.18}%
\end{equation}

\subsection{Solvability and estimates for problem \eqref{6.17}, \eqref{6.18}.}

\label{s6.3}

By repeating verbatim the reasonings of section \ref{s5.5} (with the
application of the Fourier transform in $x$ and $t$ to obtain a solution to
problem \eqref{5.10}, \eqref{5.11} ) we obtain a solution $u_{m,\varepsilon
}(x,t)\in C^{\infty}(R^{N+1})$ for equation \eqref{6.17} and for problem
\eqref{6.17}, \eqref{6.18} in general. In terms of it's Fourier-image
$\widehat{u}_{m,\varepsilon}(\xi,\xi_{0})$, this solution has, similar to
\eqref{5.28}, the representation
\begin{equation}
\widehat{u}_{m,\varepsilon}(\xi,\xi_{0})=\frac{\widehat{f_{m,\varepsilon}}%
(\xi,\xi_{0})}{i\xi_{0}+{\sum_{k=1}^{r}}|\zeta_{k}|^{\sigma_{k}}}.\label{6.19}%
\end{equation}
Besides, similar to obtaining property \eqref{5.27} in section \ref{s5.5}, we
have for $u_{m,\varepsilon}(x,t)$ property \eqref{6.18} that is (in view of
$u_{m,\varepsilon}(x,t)\in C^{\infty}(R^{N+1})$)
\begin{equation}
u_{m,\varepsilon}(x,t)\equiv0,\quad t\leq-2\varepsilon.\label{6.20}%
\end{equation}
And for the derivatives $\partial u_{m,\varepsilon}(x,t)/\partial t$ and
$(-\Delta_{z_{i}})^{\frac{\sigma_{i}}{2}}u_{m,\varepsilon}(x,t)$ we have the
representations in terms of the Fourier images
\begin{equation}
\frac{\widehat{\partial u_{m,\varepsilon}}}{\partial t}(\xi,\xi_{0}%
)=\frac{i\xi_{0}}{i\xi_{0}+{\sum_{k=1}^{r}}|\zeta_{k}|^{\sigma_{k}}}%
\widehat{f_{m,\varepsilon}}(\xi,\xi_{0})\equiv\widehat{m}_{0}(\xi,\xi
_{0})\widehat{f_{m,\varepsilon}}(\xi,\xi_{0}),\label{6.21}%
\end{equation}%
\begin{equation}
\widehat{(-\Delta_{z_{i}})^{\frac{\sigma_{i}}{2}}u_{m,\varepsilon}}(\xi
,\xi_{0})=\frac{|\zeta_{i}|^{\sigma_{i}}}{i\xi_{0}+{\sum_{k=1}^{r}}|\zeta
_{k}|^{\sigma_{k}}}\widehat{f_{m,\varepsilon}}(\xi,\xi_{0})\equiv\widehat
{m}_{i}(\xi,\xi_{0})\widehat{f_{m,\varepsilon}}(\xi,\xi_{0}).\label{6.22}%
\end{equation}
Replicating the reasoning of section \ref{s5.6} with the application of
Theorem \ref{T4.1} to obtain estimate \eqref{5.34}, we obtain from
representation \eqref{6.21} on the ground of this theorem the estimate for the
H\"{o}lder seminorm of the derivative of $u_{m,\varepsilon}(x,t)$ in $t$ in
the whole space $R^{N+1}$
\begin{equation}
\left\langle \frac{\partial u_{m,\varepsilon}(x,t)}{\partial t}\right\rangle
_{t,R^{N+1}}^{(\alpha)}\leq C(\overline{\sigma})\left\langle f_{m,\varepsilon
}(x,t)\right\rangle _{R^{N+1}}^{(\overline{\sigma}\alpha,\alpha)}\leq
C(\overline{\sigma},\alpha)\varepsilon^{-\alpha}|f|_{\overline{R_{T}^{N}}%
}^{(\overline{\sigma}\alpha,\alpha)},\label{6.23}%
\end{equation}
where the second inequality follows from \eqref{6.16}. This estimates contains
in it's right hand side the singular factor $\varepsilon^{-\alpha}$ and this
estimate is temporary and technical. The similar estimate is valid also for
the highest H\"{o}lder seminorms in $x$ - variables, but for now estimate
\eqref{6.23} is sufficient for us.

The main step to obtain the estimate for the highest H\"{o}lder seminorm of
the solution $u_{m,\varepsilon}(x,t)$ is an application of Theorem \ref{T4.2}
to representation \eqref{6.22}. The direct verification shows that for the
functions $f_{m,\varepsilon}(x,t)$ and for the multipliers $\widehat{m}%
_{i}(\xi,\xi_{0})$ in \eqref{6.22} all the conditions of Theorem \ref{T4.2}
are satisfied if $\alpha$ is sufficiently small. Firstly, functions
$f_{m,\varepsilon}(x,t)$ are finite in $R^{N+1}$ and consequently belong to
$L_{2}(R^{N+1})$. Besides, in view of \eqref{6.15}, these functions have
uniformly bounded in $m$ and $\varepsilon$ H\"{o}lder seminorms in the space
variables $x$. Further, the multiplier $\widehat{m}_{i}(\xi,\xi_{0})$ in
\eqref{6.22} possesses properties \eqref{4.18}, \eqref{4.19} if one takes
$\xi^{(1)}\equiv\xi$, $\xi^{(2)}\equiv\xi_{0}$. Finally, exactly as it was in
section \ref{s5.6}, by replicating relations \eqref{5.30.0} - \eqref{5.33},
one can verify conditions \eqref{4.20}, \eqref{4.21} of Theorem \ref{T4.2} if
$\alpha\in(0,1)$ is sufficiently small. Similar to \eqref{5.30.0} -
\eqref{5.33} we choose $s_{0}=1$, $s_{k}=N_{k}$. At that, in the present case
of the application of Theorem \ref{T4.2}, under a choice of $\delta$ from
condition \eqref{5.33}, conditions \eqref{5.30.2}, \eqref{5.30.3} acquire a
more rigid form
\[
1=s_{0}>\frac{N_{0}}{p}+\alpha=\frac{1}{p}+\alpha=1-\delta+\alpha,
\]%
\[
s_{k}=N_{k}>\frac{N_{k}}{p}+\alpha=N_{k}(1-\delta)+\alpha.
\]
But these conditions are evidently satisfied if $\alpha<\delta$, where
$\delta$ is chosen from condition \eqref{5.33}. Consequently, on the ground of
Theorem \ref{T4.2} and \eqref{6.15} we have the estimate for solutions
$u_{m,\varepsilon}(x,t)$ to \eqref{6.17}, \eqref{6.18}
\begin{equation}
{\sum\limits_{k=1}^{r}}\left\langle \left(  -\Delta_{z_{k}}\right)
^{\frac{\sigma_{k}}{2}}u_{m,\varepsilon}\right\rangle _{R^{N+1}}%
^{(\overline{\sigma}\alpha,\alpha)}\leq C(\alpha,\overline{\sigma}%
,\{N_{k}\})\left\vert f(x,t)\right\vert _{\overline{R_{T}^{N}}}^{(\overline
{\sigma}\alpha,\alpha)},\quad m=1,2,...,\varepsilon\in(0,1).\label{6.24}%
\end{equation}
And from this, on the ground of \eqref{3.6.1}, it follows, in particular, that
$u_{m,\varepsilon}(x,t)$ has the desired smoothness in $x$
\begin{equation}
\left\langle u_{m,\varepsilon}\right\rangle _{x,R^{N+1}}^{(\overline{\sigma
}+\overline{\sigma}\alpha)}\leq C(\alpha,\overline{\sigma},\{N_{k}%
\})\left\vert f(x,t)\right\vert _{\overline{R_{T}^{N}}}^{(\overline{\sigma
}\alpha,\alpha)},\quad m=1,2,...,\varepsilon\in(0,1).\label{6.25}%
\end{equation}
As for the highest H\"{o}lder seminorm in $t$ that is $\left\langle
u_{m,\varepsilon}\right\rangle _{t,\overline{R_{T}^{N}}}^{(1+\alpha
)}=\left\langle \partial u_{m,\varepsilon}(x,t)/\partial t\right\rangle
_{t,\overline{R_{T}^{N}}}^{(\alpha)}$, it's estimate over the set
$\overline{R_{T}^{N}}$ follows directly from equation \eqref{6.17} (which is
satisfied in $\overline{R_{T}^{N}}$ by the functions $u_{m,\varepsilon}(x,t)$)
and from \eqref{6.24}, \eqref{6.12}
\begin{equation}
\left\langle \frac{\partial u_{m,\varepsilon}(x,t)}{\partial t}\right\rangle
_{t,\overline{R_{T}^{N}}}^{(\alpha)}\leq{\sum\limits_{k=1}^{r}}\left\langle
\left(  -\Delta_{z_{k}}\right)  ^{\frac{\sigma_{k}}{2}}u_{m,\varepsilon
}\right\rangle _{t,\overline{R_{T}^{N}}}^{(\alpha)}+\left\langle
f_{m,\varepsilon}(x,t)\right\rangle _{t,\overline{R_{T}^{N}}}^{(\alpha)}\leq
C(\overline{\sigma},\alpha)|f|_{\overline{R_{T}^{N}}}^{(\overline{\sigma
}\alpha,\alpha)}.\label{6.29}%
\end{equation}
Thus estimates \eqref{6.25} and \eqref{6.29} give the full estimate of the
highest H\"{o}lder seminorm over $\overline{R_{T}^{N}}$ that is
\begin{equation}
\left\langle u_{m,\varepsilon}\right\rangle _{\overline{R_{T}^{N}}%
}^{(\overline{\sigma}+\overline{\sigma}\alpha,1+\alpha)}\leq C(\alpha
,\overline{\sigma},\{N_{k}\})\left\vert f(x,t)\right\vert _{\overline
{R_{T}^{N}}}^{(\overline{\sigma}\alpha,\theta\alpha)},\quad
m=1,2,...,\varepsilon\in(0,1).\label{6.30}%
\end{equation}

To finalize the estimate we need the behavior of the functions
$u_{m,\varepsilon}(x,t)$ and $\partial u_{m,\varepsilon}(x,t)/\partial t$ at
$t=0$ with respect to $m$ and $\varepsilon$, and also we need the estimates of
the module maxima for functions $\left\vert u_{m,\varepsilon}(x,t)\right\vert
_{\overline{R_{T}^{N}}}^{(0)}$. Firstly, from \eqref{6.20} (\eqref{6.18}) and
\eqref{6.16} it follows that for each $t\in\lbrack-2\varepsilon,0]$ we have
\[
\left\vert \frac{\partial u_{m,\varepsilon}(x,t)}{\partial t}\right\vert
=\left\vert \frac{\partial u_{m,\varepsilon}(x,t)}{\partial t}-\frac{\partial
u_{m,\varepsilon}(x,-2\varepsilon)}{\partial t}\right\vert \leq
\]%
\begin{equation}
\leq\left\langle u_{m,\varepsilon}\right\rangle _{t,R^{N+1}}^{(1+\alpha
)}|-2\varepsilon|^{\alpha}\leq C(\overline{\sigma},\alpha)|f|_{\overline
{R_{T}^{N}}}^{(\overline{\sigma}\alpha,\alpha)},\quad t\in\lbrack
-2\varepsilon,0].\label{6.27}%
\end{equation}
Consequently,
\begin{equation}
\left\vert u_{m,\varepsilon}(x,0)\right\vert \leq{\int\limits_{-2\varepsilon
}^{0}}\left\vert \frac{\partial u_{m,\varepsilon}(x,t)}{\partial t}\right\vert
dt\leq C(\overline{\sigma},\alpha)|f|_{\overline{R_{T}^{N}}}^{(\overline
{\sigma}\alpha,\alpha)}\varepsilon.\label{6.28}%
\end{equation}
Now from \eqref{6.29} and \eqref{6.27} we infer the estimate for $\left\vert
\partial u_{m,\varepsilon}(x,t)/\partial t\right\vert _{\overline{R_{T}^{N}}%
}^{(0)}$,
\[
\left\vert \frac{\partial u_{m,\varepsilon}(x,t)}{\partial t}\right\vert
\leq\left\vert \frac{\partial u_{m,\varepsilon}(x,0)}{\partial t}\right\vert
_{R^{N}}^{(0)}+\left\langle \frac{\partial u_{m,\varepsilon}(x,t)}{\partial
t}\right\rangle _{t,\overline{R_{T}^{N}}}^{(\alpha)}t^{\alpha}\leq
\]%
\[
\leq C(\overline{\sigma},\alpha)|f|_{\overline{R_{T}^{N}}}^{(\overline{\sigma
}\alpha,\alpha)}\left(  1+T^{\alpha}\right)  .
\]
And this permits to obtain, at last, the estimate for $\left\vert
u_{m,\varepsilon}(x,t)\right\vert _{\overline{R_{T}^{N}}}^{(0)}$ for
$t\in\lbrack0,T]$
\[
\left\vert u_{m,\varepsilon}(x,t)\right\vert \leq\left\vert u_{m,\varepsilon
}(x,0)\right\vert +{\int\limits_{0}^{t}}\left\vert \frac{\partial
u_{m,\varepsilon}(x,\tau)}{\partial\tau}\right\vert d\tau\leq
\]%
\[
\leq C(\overline{\sigma},\alpha)|f|_{\overline{R_{T}^{N}}}^{(\overline{\sigma
}\alpha,\alpha)}(\varepsilon+t\left(  1+T^{\alpha}\right)  ),
\]
that is
\begin{equation}
\left\vert u_{m,\varepsilon}(x,t)\right\vert _{\overline{R_{T}^{N}}}^{(0)}\leq
C(\overline{\sigma},\alpha)|f|_{\overline{R_{T}^{N}}}^{(\overline{\sigma
}\alpha,\alpha)}(\varepsilon+T(\left(  1+T^{\alpha}\right)  )).\label{6.31}%
\end{equation}

Thus, combining the reasoning and the estimates of this section, we have
proved the following assertion.

\begin{lemma}
\label{L6.2} There exists such $\alpha_{0}=\alpha_{0}(\overline{\sigma
},\{N_{k}\})$ that for $\alpha\in(0,\alpha_{0}]$ problems \eqref{6.17},
\eqref{6.18} have solutions $u_{m,\varepsilon}(x,t)\in C^{\overline{\sigma
}(1+\alpha),1+\alpha}(R^{N+1})$ with the uniformly bounded in $\overline
{R_{T}^{N}}$ norms
\begin{equation}
\left\vert u_{m,\varepsilon}(x,t)\right\vert _{\overline{R_{T}^{N}}%
}^{(\overline{\sigma}+\overline{\sigma}\alpha,1+\alpha)}\leq C(T,\overline
{\sigma})\left\vert f(x,t)\right\vert _{\overline{R_{T}^{N}}}^{(\overline
{\sigma}\alpha,\alpha)}, \label{6.32}%
\end{equation}
and
\begin{equation}
\left\langle u_{m,\varepsilon}(x,t)\right\rangle _{\overline{R_{T}^{N}}%
}^{(\overline{\sigma}+\overline{\sigma}\alpha,1+\alpha)}\leq C(\overline
{\sigma})\left\vert f(x,t)\right\vert _{\overline{R_{T}^{N}}}^{(\overline
{\sigma}\alpha,\alpha)} \label{6.33}%
\end{equation}
with some constant $C(\theta,\overline{\sigma})$, which does not depend on $T$.

Moreover, for an arbitrary $\widetilde{T}\leq T$
\begin{equation}
\left\vert u_{m,\varepsilon}(x,t)\right\vert _{\overline{R_{\widetilde{T}}%
^{N}}}^{(0)}\leq C(\overline{\sigma},\alpha)|f|_{\overline{R_{T}^{N}}%
}^{(\overline{\sigma}\alpha,\alpha)}(1+\widetilde{T}^{1+\alpha}), \label{6.34}%
\end{equation}
and also estimate \eqref{6.28} is valid at $t=0$.
\end{lemma}

\subsection{Solvability and estimates of the solutions to problem \eqref{6.1},
\eqref{6.2}. Extension of the solution and it's uniqueness.}

\label{s6.4}

Exactly as it was done in section \ref{s5.7} under the the proof of
Proposition \ref{P5.1}, we can go to the limit at $\varepsilon\rightarrow0$,
$m\rightarrow\infty$ in the sequence of problems \eqref{6.17}, \eqref{6.18}.
Similar to section \ref{s5.7}, this is grounded on the convergence on compact
sets $\overline{R_{T}^{N}}$ of the sequence of the mollified functions
$f_{m,\varepsilon}(x,t)$ to the original function $f(x,t)$ in the sense of
\eqref{6.13}. Besides, due to the estimates of Lemma \ref{L6.2}, a subsequence
of the solutions $u_{m,\varepsilon}(x,t)$ has the analogous convergence,
completely similar to \eqref{5.41} - \eqref{5.43}. The only difference from
section \ref{s5.7} is that the sequence of the solutions $u_{m,\varepsilon
}(x,t)$ does not vanish at $t=0$ automatically. However, estimate \eqref{6.28}
at $t=0$ tells that the limiting function $u(x,t)\in C^{\overline{\sigma
}(1+\alpha),1+\alpha}(\overline{R_{T}^{N}})$ not only satisfies equation
\eqref{6.1}, but vanishes at $t=0$ according to the zero initial condition.
Therefore, exactly replicating reasonings of section \ref{s5.7}, including the
proof of Corollary \ref{C5.1} and the proof of Proposition \ref{P5.2} of
section \ref{s5.8} on the extension of the solution, we arrive at the
following assertion.

\begin{proposition}
\label{P6.1} There exists such $\alpha_{0}=\alpha_{0}(\overline{\sigma
},\{N_{k}\})$ that for $\alpha\in(0,\alpha_{0}]$ and under conditions
\eqref{6.3}, \eqref{6.4} problem \eqref{6.1}, \eqref{6.2} has a solution
$u(x,t)\in C^{\overline{\sigma}(1+\alpha),1+\alpha}(\overline{R_{T}^{N}})$
with estimates \eqref{6.6} - \eqref{6.6.2}.

In the case when $f(x,t)$ is defined for all $t>0$ in the domain
$\overline{R_{\infty}^{N} }=R^{N}\times\lbrack0,\infty)$ and it's norm in the
space $C^{\overline{\sigma}\alpha,\alpha}(\overline{R_{\infty}^{N}})$,
$\alpha\in(0,\alpha_{0}]$, is finite that is $\left\vert f(x,t)\right\vert
_{\overline{R_{\infty}^{N}}}^{(\overline{\sigma}\alpha,\alpha)}<\infty$
problem \eqref{6.1}, \eqref{6.2} has such a solution $u(x,t)$ that locally in
time $t$ belongs to the space $C^{\overline{\sigma}(1+\alpha),1+\alpha
}(\overline{R_{\infty}^{N}})$ and that obeys estimates \eqref{6.6.1},
\eqref{6.6.2} that is
\begin{equation}
\left\langle u\right\rangle _{\overline{R_{\infty}^{N}}}^{(\overline{\sigma
}(1+\alpha),1+\alpha)}\leq C(\overline{\sigma},\alpha)\left(  |f|_{\overline
{R_{\infty}^{N}}}^{(\overline{\sigma}\alpha,\alpha)}+|u_{0}|_{R^{N}%
}^{(\overline{\sigma}(1+\alpha))}\right)  , \label{6.35}%
\end{equation}%
\begin{equation}
\left\vert u\right\vert _{\overline{R_{\widetilde{T}}^{N}}}^{(0)}\leq
C(\overline{\sigma},\alpha)\left(  |f|_{\overline{R_{\infty}^{N}}}%
^{(\overline{\sigma}\alpha,\alpha)}+|u_{0}|_{R^{N}}^{(\overline{\sigma
}(1+\alpha))}\right)  (1+\widetilde{T}^{1+\alpha})+|u_{0}|_{R^{N}}^{(0)}%
,\quad\widetilde{T}\leq\infty. \label{6.36}%
\end{equation}
Consequently, for each finite $T>0$ estimate \eqref{6.6} is also valid.

Besides, let a solution $u(x,t)\in C^{\overline{\sigma}(1+\alpha),1+\alpha
}(\overline{R_{T}^{N}})$, $\alpha\in(0,\alpha_{0}]$, to problem \eqref{6.1},
\eqref{6.2} in $\overline{R_{T} ^{N}}$ be fixed and let the right hand side
$f(x,t)$ in \eqref{6.1} be defined (or extended from $[0,T]$) for all $t>0$ on
the domain $\overline{R_{\infty}^{N}}$, and $f(x,t)\in C^{\overline{\sigma
}\alpha,\alpha}(\overline{R_{\infty}^{N}})$. Then there exists such solution
to problem \eqref{6.1}, \eqref{6.2} in $\overline{R_{\infty}^{N}}$ (denoted by
the same symbol $u(x,t)$) that belongs to $C^{\overline{\sigma}(1+\alpha
),1+\alpha}(\overline{R_{\widetilde{T}}^{N}})$ for each $\widetilde{T}>0$ and
that coincides with $u(x,t)$ in $\overline{R_{T}^{N}}$ with the analogous to
\eqref{6.35}, \eqref{6.36} estimates in $\overline{R_{\infty}^{N}}$
\begin{equation}
\left\langle u\right\rangle _{\overline{R_{\infty}^{N}}}^{(\overline{\sigma
}(1+\alpha),1+\alpha)}\leq C(\overline{\sigma},\alpha)\left(  |f|_{\overline
{R_{\infty}^{N}}}^{(\overline{\sigma}\alpha,\alpha)}+\left\vert u\right\vert
_{\overline{R_{T}^{N}}}^{(\overline{\sigma}+\overline{\sigma}\alpha,1+\alpha
)}\right)  , \label{6.37}%
\end{equation}
\begin{equation}
\left\vert u\right\vert _{\overline{R_{\widetilde{T}}^{N}}}^{(0)}\leq
C(\overline{\sigma},\alpha)\left(  |f|_{\overline{R_{\infty}^{N}}}%
^{(\overline{\sigma}\alpha,\alpha)}+\left\vert u\right\vert _{\overline
{R_{T}^{N}}}^{(\overline{\sigma}+\overline{\sigma}\alpha,1+\alpha)}\right)
(1+\widetilde{T}^{1+\alpha})+\left\vert u\right\vert _{\overline{R_{T}^{N}}%
}^{(\overline{\sigma}+\overline{\sigma}\alpha,1+\alpha)},\quad\widetilde
{T}\leq\infty. \label{6.38}%
\end{equation}

\end{proposition}

Note that the proof of the last assertion of this proposition about the
extension of the solution is not only a replication of the proof to
Proposition \ref{P5.2} but is the significant simplification of the last. In
particular, due to locality of the first derivative in $t$, we do not need an
additional term in the right hand side of the equation as it was in equation
\eqref{5.77}, and consequently, we do not need an additional lemma of the kind
of Lemma \ref{L5.1}. Besides, we do not need to reduce the smoothness exponent
from $\alpha$ to $\beta< \alpha$ to achieve the agreement of the solution at
$t=T$ as it was done in Proposition \ref{P5.2}. This reflect the known
semigroup property of Cauchy problems with the first time derivative.

On the ground of Proposition \ref{P6.1} on the extension of the solution and
by the simple replication of the proof of Lemma \ref{L5.2} we get the
assertion about the uniqueness of the solution from $C^{\overline{\sigma
}(1+\alpha),1+\alpha}(\overline{R_{T}^{N}})$ to problem \eqref{6.1},
\eqref{6.2}. We stress that in this place we do not need any restriction on
the smoothness exponent $\alpha$ to have it sufficiently small since the space
with a bigger $\alpha$ is continuously embedded in each space with a less
$\alpha$.

\begin{proposition}
\label{P6.2} Problem \eqref{6.1}, \eqref{6.2} has at most one solution from
the space $C^{\overline{\sigma}(1+\alpha),1+\alpha}(\overline{R_{T}^{N}})$.
\end{proposition}

\subsection{Removing the restriction on the H\"{o}lder exponent $0<\alpha
\leq\alpha_{0}=\alpha_{0}(\overline{\sigma},\{N_{k}\})$ and completion of the
proof of Theorem \ref{T6.1}.}

\label{s6.5}

Let now in problem \eqref{6.1}, \eqref{6.2} the exponent $\alpha$ satisfies
instead of conditions \eqref{6.4} the only condition
\begin{equation}
\alpha\in(0,1), \label{6.39}%
\end{equation}
without the restriction $\alpha\leq\alpha_{0} \in(0,1)$. Moreover, since the
case of a sufficiently small $\alpha\leq\alpha_{0}$ is already considered
above, we assume, to be specific, that $\alpha\in(\alpha_{0},1)$. Since the
functions $f(x,t)$ and $u_{0}(x)$ satisfy condition \eqref{6.3}, then they
satisfy the more weak condition
\[
f(x,t)\in C^{\overline{\sigma}\alpha_{0},\alpha_{0}}(\overline{R_{T}^{N}%
}),\quad u_{0}(x)\in C^{\overline{\sigma}(1+\alpha_{0})}(R^{N}).
\]
This follows from the continuous embeddings $C^{\overline{\sigma}\alpha
,\alpha}(\overline{R_{T}^{N}})\subset C^{\overline{\sigma}\alpha_{0}%
,\alpha_{0}}(\overline{R_{T}^{N}})$ and $C^{\overline{\sigma} (1+\alpha
)}(R^{N})\subset C^{\overline{\sigma}(1+\alpha_{0})}(R^{N})$ in view of the
assumption $\alpha\in(\alpha_{0},1)$. On the ground of Propositions \ref{P6.1}
and \ref{P6.2} there exists the unique solution $u(x,t)\in C^{\overline
{\sigma}(1+\alpha_{0}),1+\alpha_{0}}(\overline{R_{T}^{N}})$ to problem
\eqref{6.1}, \eqref{6.2} with estimates \eqref{6.6} - \eqref{6.6.2} with the
exponent $\alpha_{0}$ instead of $\alpha$. We show below that in fact the
solution belongs to the class $u(x,t)\in C^{\overline{\sigma}(1+\alpha
),1+\alpha}(\overline{R_{T}^{N}})$ and estimates \eqref{6.6} - \eqref{6.6.2}
are valid with the bigger exponent $\alpha$ itself.

Let an index $k\in\{1,...,r\}$ be fixed and let a non-zero $h\in R^{N_{k}}$ be
also fixed, where $R^{N_{k}}$ is the corresponding subspace of $R^{N}$,
containing space variables $z_{k}$. For $\omega\in(0,1)$ and $m>\sigma
_{k}+\sigma_{k}\alpha$ consider the function
\begin{equation}
u_{k}(x,t)\equiv\frac{\delta_{h,z_{k}}^{m}u(x,t)}{|h|^{\rho}},\quad\rho
\equiv(1-\omega)\sigma_{k}\alpha,\quad\omega\in(0,1),\quad\alpha_{k}%
\equiv\omega\alpha,\label{6.40}%
\end{equation}
where $\omega$ is chosen so small that $\alpha_{k}\equiv\omega\alpha\leq
\alpha_{0}$, and such that $\rho$, $\alpha_{k}$, $\sigma_{i}\alpha_{k}$, and
$\sigma_{i}+\sigma_{i}\alpha_{k}$ are noninteger and $\sigma_{i}\alpha_{k}%
\in(0,1)$, $i=1,2,...,r$. In view of linearity of problem \eqref{6.1},
\eqref{6.2}, it can be directly verified that the function $u_{k}(x,t)$
satisfies this problem with the following right hand side and the initial data
correspondingly
\begin{equation}
f_{k}(x,t)\equiv\frac{\delta_{h,z_{k}}^{m}f(x,t)}{|h|^{\rho}},\quad
u_{0,k}(x)\equiv\frac{\delta_{h,z_{k}}^{m}u_{0}(x)}{|h|^{\rho}}.\label{6.41}%
\end{equation}
According to Proposition \ref{DopMushel}, the functions $f_{k}(x,t)$ and
$u_{0,k}(x)$ belong to the spaces
\[
f_{k}(x,t)\in C^{\overline{\sigma}\alpha_{k},\alpha_{k}}(\overline{R_{T}^{N}%
}),\quad u_{0,k}(x)\in C^{\overline{\sigma}(1+\alpha_{k})}(R^{N}),
\]
and
\begin{equation}
\left\vert f_{k}\right\vert _{\overline{R_{T}^{N}}}^{(\overline{\sigma}%
\alpha_{k},\alpha_{k})}\leq C(\alpha,\overline{\sigma},\omega)\left\vert
f\right\vert _{\overline{R_{T}^{N}}}^{(\overline{\sigma}\alpha,\alpha)}%
,\quad\left\vert u_{0,k}\right\vert _{R^{N}}^{(\overline{\sigma}(1+\alpha
_{k}))}\leq C(\alpha,\overline{\sigma},\omega)\left\vert u_{0}\right\vert
_{R^{N}}^{(\overline{\sigma}(1+\alpha))}.\label{6.42}%
\end{equation}
From Propositions \ref{P6.1} and \ref{P6.2} and from \eqref{6.42} it follows
that $u_{k}(x,t)\in C^{\overline{\sigma}(1+\alpha_{k}),1+\alpha_{k}}%
(\overline{R_{T}^{N}})$, and
\[
\left\langle u_{k}\right\rangle _{\overline{R_{T}^{N}}}^{(\overline{\sigma
}(1+\alpha_{k}),1+\alpha_{k})}\leq C(\overline{\sigma},\alpha_{k})\left(
\left\vert f_{k}\right\vert _{\overline{R_{T}^{N}}}^{(\overline{\sigma}%
\alpha_{k},\alpha_{k})}+\left\vert u_{0,k}\right\vert _{R^{N}}^{(\overline
{\sigma}(1+\alpha_{k}))}\right)  \leq
\]%
\begin{equation}
\leq C(\alpha,\overline{\sigma},\omega)\left(  \left\vert f\right\vert
_{\overline{R_{T}^{N}}}^{(\overline{\sigma}\alpha,\alpha)}+\left\vert
u_{0}\right\vert _{R^{N}}^{(\overline{\sigma}(1+\alpha))}\right)
.\label{6.43}%
\end{equation}
In particular, this estimate includes the highest H\"{o}lder seminorm with
respect to the group of the space variables $z_{k}$
\[
\left\langle u_{k}\right\rangle _{z_{k},\overline{R_{T}^{N}}}^{(\sigma
_{k}+\sigma_{k}\alpha_{k})}\leq C(\alpha,\overline{\sigma},\omega)\left(
\left\vert f\right\vert _{\overline{R_{T}^{N}}}^{(\overline{\sigma}%
\alpha,\alpha)}+\left\vert u_{0}\right\vert _{R^{N}}^{(\overline{\sigma
}(1+\alpha))}\right)
\]
that is in view of property \eqref{2.4} (remind that $m>\sigma_{k}+\sigma
_{k}\alpha$),
\begin{equation}
\sup_{s\in R^{N_{k}},s\neq0}\frac{|\delta_{s,z_{k}}^{m}u_{k}(x,t)|}%
{|s|^{\sigma_{k}+\sigma_{k}\alpha_{k}}}=C\left\langle u_{k}\right\rangle
_{z_{k},\overline{R_{T}^{N}}}^{(\sigma_{k}+\sigma_{k}\alpha_{k})}\leq
C(\alpha,\overline{\sigma},\omega)\left(  \left\vert f\right\vert
_{\overline{R_{T}^{N}}}^{(\overline{\sigma}\alpha,\alpha)}+\left\vert
u_{0}\right\vert _{R^{N}}^{(\overline{\sigma}(1+\alpha))}\right)
.\label{6.44}%
\end{equation}
Thus, according to the definition of $u_{k}(x,t)$ in \eqref{6.40},
\begin{equation}
\sup_{s\in R^{N_{k}},s\neq0}\frac{|\delta_{s,z_{k}}^{m}\delta_{h,z_{k}}%
^{m}u(x,t)|}{|s|^{\sigma_{k}+\sigma_{k}\alpha_{k}}|h|^{\rho}}\leq
C(\alpha,\overline{\sigma},\omega)\left(  \left\vert f\right\vert
_{\overline{R_{T}^{N}}}^{(\overline{\sigma}\alpha,\alpha)}+\left\vert
u_{0}\right\vert _{R^{N}}^{(\overline{\sigma}(1+\alpha))}\right)
.\label{6.45}%
\end{equation}
Choosing here the particular value $s=h$ and making use of the definitions of
$\alpha_{k}$ and $\rho$ in \eqref{6.40}, we conclude that
\begin{equation}
\frac{|\delta_{h,z_{k}}^{2m}u(x,t)|}{|h|^{\sigma_{k}+\sigma_{k}\alpha}}\leq
C(\alpha,\overline{\sigma},\omega)\left(  \left\vert f\right\vert
_{\overline{R_{T}^{N}}}^{(\overline{\sigma}\alpha,\alpha)}+\left\vert
u_{0}\right\vert _{R^{N}}^{(\overline{\sigma}(1+\alpha))}\right)
.\label{6.46}%
\end{equation}
Since $h\in R^{N_{k}}\backslash\{0\}$ is arbitrary, we can infer from this, on
the ground of \eqref{2.4}, the estimate for the highest seminorm in $z_{k}$ of
the solution $u(x,t)$ itself with the original $\alpha$
\begin{equation}
\left\langle u(x,t)\right\rangle _{z_{k},\overline{R_{T}^{N}}}^{(\sigma
_{k}+\sigma_{k}\alpha)}\leq C(\alpha,\overline{\sigma},\omega)\left(
\left\vert f\right\vert _{\overline{R_{T}^{N}}}^{(\overline{\sigma}%
\alpha,\alpha)}+\left\vert u_{0}\right\vert _{R^{N}}^{(\overline{\sigma
}(1+\alpha))}\right)  .\label{6.47}%
\end{equation}
It is clear that these reasonings do not depend on the index $k$ of a group of
the space variables, and thus \eqref{6.47} is valid for each $k=1,...,r$ that
is
\begin{equation}
\left\langle u(x,t)\right\rangle _{x,\overline{R_{T}^{N}}}^{(\overline{\sigma
}(1+\alpha))}\leq C(\alpha,\overline{\sigma},\omega)\left(  \left\vert
f\right\vert _{\overline{R_{T}^{N}}}^{(\overline{\sigma}\alpha,\alpha
)}+\left\vert u_{0}\right\vert _{R^{N}}^{(\overline{\sigma}(1+\alpha
))}\right)  .\label{6.48}%
\end{equation}
On the ground of the properties of the fractional operators $(-\Delta_{z_{k}%
})^{\frac{\sigma_{k}}{2}}$ from equation \eqref{6.1}, that are formulated in
Proposition \ref{P3.5} in \eqref{3.18}, we conclude that the consecuence of
\eqref{6.48} is the estimate
\begin{equation}
\left\langle {\sum_{k=1}^{r}}(-\Delta_{z_{k}})^{\frac{\sigma_{k}}{2}%
}u(x,t)\right\rangle _{x,\overline{R_{T}^{N}}}^{(\overline{\sigma}\alpha)}\leq
C(\alpha,\overline{\sigma},\omega,\{N_{k}\})\left(  \left\vert f\right\vert
_{\overline{R_{T}^{N}}}^{(\overline{\sigma}\alpha,\alpha)}+\left\vert
u_{0}\right\vert _{R^{N}}^{(\overline{\sigma}(1+\alpha))}\right)
.\label{6.49}%
\end{equation}
But then from equation \eqref{6.1} and from \eqref{6.3} we obtain the estimate
for the H\"{o}lder seminorm of the time derivative $u_{t}(x,t)$ with respect
to $x$
\begin{equation}
\left\langle u_{t}(x,t)\right\rangle _{x,\overline{R_{T}^{N}}}^{(\overline
{\sigma}\alpha)}\leq C(\alpha,\overline{\sigma},\omega,\{N_{k}\})\left(
\left\vert f\right\vert _{\overline{R_{T}^{N}}}^{(\overline{\sigma}%
\alpha,\alpha)}+\left\vert u_{0}\right\vert _{R^{N}}^{(\overline{\sigma
}(1+\alpha))}\right)  .\label{6.50}%
\end{equation}
Consider again some fixed group $z_{k}$ of the space variables. Based on
\eqref{6.47} and \eqref{6.50}, we can apply Theorem \ref{T3.4} and it's
estimate \eqref{3.72} to the function $u(x,t)$ to obtain
\begin{equation}
\left\langle (-\Delta_{z_{k}})^{\frac{\sigma_{k}}{2}}u(x,t)\right\rangle
_{t,\overline{R_{T}^{N}}}^{(\alpha)}\leq C(\alpha,\overline{\sigma}%
,\omega,\{N_{k}\})\left(  \left\vert f\right\vert _{\overline{R_{T}^{N}}%
}^{(\overline{\sigma}\alpha,\alpha)}+\left\vert u_{0}\right\vert _{R^{N}%
}^{(\overline{\sigma}(1+\alpha))}\right)  .\label{6.51}%
\end{equation}
Since the group $z_{k}$ is arbitrary, we can infer from the last inequality
and again from equation \eqref{6.1} that
\begin{equation}
\left\langle u_{t}(x,t)\right\rangle _{t,\overline{R_{T}^{N}}}^{(\alpha)}\leq
C(\alpha,\overline{\sigma},\omega,\{N_{k}\})\left(  \left\vert f\right\vert
_{\overline{R_{T}^{N}}}^{(\overline{\sigma}\alpha,\alpha)}+\left\vert
u_{0}\right\vert _{R^{N}}^{(\overline{\sigma}(1+\alpha))}\right)
.\label{6.52}%
\end{equation}
Combining now estimates \eqref{6.48}, \eqref{6.52} (and estimate \eqref{6.6}
for the solution $u(x,t)$ in the weaker space $C^{\overline{\sigma}%
(1+\alpha_{0}),1+\alpha_{0}}(\overline{R_{T}^{N}})$, containing the estimate
of the lowest norm $\left\vert u(x,t)\right\vert _{\overline{R_{T}^{N}}}%
^{(0)}$), we get finally
\[
|u|_{\overline{R_{T}^{N}}}^{(\overline{\sigma}(1+\alpha),1+\alpha)}\leq
C(\overline{\sigma},\alpha,T)\left(  |f|_{\overline{R_{T}^{N}}}^{(\overline
{\sigma}\alpha,\alpha)}+|u_{0}|_{R^{N}}^{(\overline{\sigma}(1+\alpha
))}\right)  ,
\]
which completes the proof of Theorem \ref{T6.1}.

\section{Problem \eqref{6.1}, \eqref{6.2} for an arbitrary $\alpha>0$.}

\label{s7}

In the previous section we've considered problem \eqref{6.1}, \eqref{6.2}
under the restriction $\alpha\in(0,1)$. This imposes the restriction on the
smoothness of the solution in $t$ within one integer derivative in $t$ with
the finite H\"{o}lder seminorm in $t$ of order $\alpha\in(0,1)$. This also
imposes the restriction on the smoothness of the solution in $x$ to the
anisotropic orders $\sigma_{k} + \sigma_{k}\alpha$. In this section, we will
get rid of the restriction $\alpha\in(0,1)$, which will allow to consider
problem \eqref{6.1}, \eqref{6.2} in anisotropic H\"{o}lder spaces of arbitrary smoothness.

So, let first in problem \eqref{6.1}, \eqref{6.2} we have $\alpha\in(1,2)$ and
let \eqref{6.3} be satisfied. Then, on the ground of Theorem \ref{T6.1}, this
problem has the unique solution from the space $u(x,t)\in C^{\overline{\sigma
}(1+\beta),1+\beta}(\overline{R_{T}^{N}})$, $\beta\equiv\alpha-1\in(0,1)$, due
to the fact that the right hand side and the initial data are even surplus smooth.

Show first the additional smoothness of the solution in $t$. For this we will
simply differentiate equation \eqref{6.1} in $t$ to reduce the problem with a
bigger $\alpha$ to a problem with a smaller one. Formally differentiating
equation \eqref{6.1} in $t$ and denoting (formally for a while) the derivative
$u_{t}(x,t)$ by $v(x,t)\equiv u_{t}(x,t)$, we obtain for $v(x,t)$ the same
equation, but with the another right hand side
\begin{equation}
Lu(x,t)\equiv v_{t}(x,t)+{\sum_{k=1}^{r}}(-\Delta_{z_{k}})^{\frac{\sigma_{k}%
}{2}}v=g(x,t)\equiv f_{t}(x,t),\quad(x,t)\in R_{T}^{N}.\label{7.1}%
\end{equation}
The initial condition for the function $v(x,t)\equiv u_{t}(x,t)$ at $t=0$ is
defined from the original equation \eqref{6.1}, which shows that at $t=0$
\begin{equation}
v(x,0)=v_{0}(x)\equiv f(x,0)-{\sum_{k=1}^{r}}(-\Delta_{z_{k}})^{\frac
{\sigma_{k}}{2}}u_{0}(x),\quad x\in R^{N}.\label{7.2}%
\end{equation}
Making use of formulas \eqref{2.5}, \eqref{2.6} and Theorem \ref{T3dop2} to
calculate the spaces of smoothness for the functions $g(x,t)$ and $v_{0}(x)$,
from \eqref{7.1} and \eqref{7.2}, with taking the notation
\begin{equation}
\beta\equiv\alpha-1\in(0,1),\label{7.3}%
\end{equation}
we obtain that
\begin{equation}
g(x,t)\in C^{\overline{\sigma}\beta,\beta}(\overline{R_{T}^{N}}),\quad
v_{0}(x)\in C^{\overline{\sigma}(1+\beta)}(R^{N})\label{7.4}%
\end{equation}
and the following estimates are valid
\begin{equation}
|g|_{\overline{R_{T}^{N}}}^{(\overline{\sigma}\beta,\beta)}\leq C(\overline
{\sigma},\alpha,T)|f|_{\overline{R_{T}^{N}}}^{(\overline{\sigma}\alpha
,\alpha)},\quad|v_{0}|_{R^{N}}^{(\overline{\sigma}(1+\beta))}\leq
C(\overline{\sigma},\alpha)|u_{0}|_{R^{N}}^{(\overline{\sigma}(1+\alpha
))}.\label{7.5}%
\end{equation}
Based on Theorem \ref{T6.1}, we infer that problem \eqref{7.1}, \eqref{7.2}
has the unique solution $v(x,t)\in C^{\overline{\sigma}(1+\beta),1+\beta
}(\overline{R_{T}^{N}})$ and
\begin{equation}
|v|_{\overline{R_{T}^{N}}}^{(\overline{\sigma}(1+\beta),1+\beta)}\leq
C(\overline{\sigma},\alpha,T)\left(  |g|_{\overline{R_{T}^{N}}}^{(\overline
{\sigma}\beta,\beta)}+|v_{0}|_{R^{N}}^{(\overline{\sigma}(1+\beta))}\right)
.\label{7.6}%
\end{equation}
The direct verification with taking into account the definitions of $g(x,t)$
in \eqref{7.1} and of $v_{0}(x)$ in \eqref{7.2} shows that the function
\[
\widetilde{u}(x,t)\equiv u_{0}(x)+{\int\limits_{0}^{t}}v(x,\tau)d\tau
\]
satisfies the original problem \eqref{6.1}, \eqref{6.2}. Since in addition
$\widetilde{u}(x,t)\in C^{\overline{\sigma}(1+\beta),1+\beta}(\overline
{R_{T}^{N}})$ on the ground of \eqref{7.6}, then from the uniqueness in this
class it follows that $\widetilde{u}(x,t)$ coincides with the original
solution $u(x,t)$ that is, in particular, $v(x,t)\equiv u_{t}(x,t)$ not
formally but in fact. But then $u_{t}(x,t)\in C^{\overline{\sigma}%
(1+\beta),1+\beta}(\overline{R_{T}^{N}})$ and for this function estimate
\eqref{7.6} is valid. This means, in particular, that
\[
\left\langle u_{t}(x,t)\right\rangle _{t,\overline{R_{T}^{N}}}^{(1+\beta)}\leq
C(\overline{\sigma},\alpha,T)\left(  |g|_{\overline{R_{T}^{N}}}^{(\overline
{\sigma}\beta,\beta)}+|v_{0}|_{R^{N}}^{(\overline{\sigma}(1+\beta))}\right)  ,
\]
and from this, on the ground of the definition of $\beta$, of the definition
of H\"{o}lder seminorms, and of \eqref{7.5}, we obtain the higher smoothness
of the solution $u(x,t)$ in $t$ that is
\begin{equation}
\left\langle u(x,t)\right\rangle _{t,\overline{R_{T}^{N}}}^{(1+\alpha)}\leq
C(\overline{\sigma},\alpha,T)\left(  |f|_{\overline{R_{T}^{N}}}^{(\overline
{\sigma}\alpha,\alpha)}+|u_{0}|_{R^{N}}^{(\overline{\sigma}(1+\alpha
))}\right)  ,\quad\alpha\in(0,2).\label{7.7}%
\end{equation}

Turning now to the additional smoothness in $x$, we confine ourselves to a
brief explanation because the proof is identical to that from section
\ref{s6.5} at the removing of the restriction on the smallness of $\alpha$. At
that we consider the functions ($h\in R^{N_{k}}$)
\[
u_{k}(x,t)\equiv\frac{\delta_{h,z_{k}}^{m}u(x,t)}{|h|^{\rho}},\quad\rho
\equiv(1-\omega)\sigma_{k}\alpha,\quad\omega\in(0,1),\quad\alpha_{k}%
\equiv\omega\alpha,
\]
where similar to the previous reasonings $\omega$ is chosen such that
$\alpha_{k}\equiv\omega\alpha\in(0,1)$ and also such that the numbers $\rho$,
$\alpha_{k}$, $\sigma_{i}\alpha_{k}$, and $\sigma_{i}+\sigma_{i} \alpha_{k}$
are nonintegers. Further reasonings also coincide with reasonings from section
\ref{s6.5}, which gives
\begin{equation}
\left\langle u(x,t)\right\rangle _{x,\overline{R_{T}^{N}}}^{\left(
\overline{\sigma}(1+\alpha)\right)  }\leq C(\overline{\sigma},\alpha,T)\left(
|f|_{\overline{R_{T}^{N}}}^{(\overline{\sigma}\alpha,\alpha)}+|u_{0}|_{R^{N}%
}^{(\overline{\sigma}(1+\alpha))}\right)  ,\quad\alpha\in(0,2). \label{7.8}%
\end{equation}

Estimates \eqref{7.7} and \eqref{7.8} mean that Theorem \ref{T6.1} is valid
not only for $\alpha\in(0,1)$ but $\alpha\in(1,2)$ as well that is for a
noninteger $\alpha\in(0,2)$. Iterating this process by induction, we, finally,
arrive at the following assertion.

\begin{theorem}
\label{T7.1} If $\alpha>0$ is a noninteger and conditions \eqref{6.3} are
satisfied, then problem \eqref{6.1}, \eqref{6.2} has the unique solution
$u(x,t)\in C^{\overline{\sigma}(1+\alpha),1+\alpha}(\overline{R_{T}^{N}})$
with the estimates
\begin{equation}
|u|_{\overline{R_{T}^{N}}}^{(\overline{\sigma}(1+\alpha),1+\alpha)}\leq
C(\overline{\sigma},\alpha,T)\left(  |f|_{\overline{R_{T}^{N}}}^{(\overline
{\sigma}\alpha,\alpha)}+|u_{0}|_{R^{N}}^{(\overline{\sigma}(1+\alpha
))}\right)  , \label{7.9}%
\end{equation}%
\begin{equation}
\left\langle u\right\rangle _{\overline{R_{T}^{N}}}^{(\overline{\sigma
}(1+\alpha),1+\alpha)}\leq C(\overline{\sigma},\alpha)\left(  |f|_{\overline
{R_{T}^{N}}}^{(\overline{\sigma}\alpha,\alpha)}+|u_{0}|_{R^{N}}^{(\overline
{\sigma}(1+\alpha))}\right)  , \label{7.10}%
\end{equation}%
\begin{equation}
\left\vert u\right\vert _{\overline{R_{\widetilde{T}}^{N}}}^{(0)}\leq
C(\overline{\sigma},\alpha)\left(  |f|_{\overline{R_{T}^{N}}}^{(\overline
{\sigma}\alpha,\alpha)}+|u_{0}|_{R^{N}}^{(\overline{\sigma}(1+\alpha
))}\right)  \widetilde{T}^{1+\alpha}+|u_{0}|_{R^{N}}^{(0)},\quad\widetilde
{T}\leq T. \label{7.11}%
\end{equation}

\end{theorem}

Note that estimate \eqref{7.11} is obtained completely similar to \eqref{6.6.2}.

\section{Smoothness of solution to \eqref{5.1}, \eqref{5.2} for more smooth
initial data.}

\label{s8}

It was shown in the previous section that in problem \eqref{6.1}, \eqref{6.2},
which contains the integer derivative in $t$ of the first order, the
smoothness of the solution can be arbitrary high, depending on the smoothness
of the data. In the present section We will extend this property to solutions
of problem \eqref{5.1}, \eqref{5.2}. However, as we will show below, the
complete extension of this property to the case of a fractional derivative in
$t$ of order $\theta$ requires from the right hand side of the equation to
obey a series of conditions. At that there are no restrictions for the
increasing of the smoothness in the space variables $x$ and such smoothness
increases together with the increasing of the corresponding smoothness of the
data. Pretty different situation takes place with the smoothness in $t$. We
explain this with the following simple example, which completely reflects the
essence of the situation.

\begin{example}
\label{E8.1} Let in problem \eqref{5.1}, \eqref{5.2} the right hand side and
the initial data be equal correspondingly
\begin{equation}
f(x,t)\equiv t,\quad u_{0}(x)\equiv0. \label{8.1}%
\end{equation}
It can be directly verified that the unique (in view of Theorem \ref{T5.1})
solution to \eqref{5.1}, \eqref{5.2} with such data is the function with the
dependence on $t$ solely
\begin{equation}
u(x,t)=C(\theta)t^{1+\theta}. \label{8.2}%
\end{equation}
Evidently that the function $f(x,t)$ belongs to any space $C^{\overline
{\sigma}\alpha,\theta\alpha}(\overline{R_{T}^{N}})$ with an arbitrary large
$\alpha$ that is, in fact, to $C^{\infty}$. However, for $\theta<1$ the
solution $u(x,t)$ has the smoothness in $t$ in the closed domain $\{t \geq0\}$
only up to the order $1+\theta$, though this solution is infinitely smooth in
$x$. In terms of the previous section about the raising of smoothness by
differentiation this corresponds to the following. After the differentiation
of equation \eqref{5.1} in $t$ and at the consideration of the obtained from
\eqref{5.1} equation for $u_{t}(x,t)$ we obtain the right hand side
$f_{t}(x,t)\equiv1$ and the initial data $u_{t}(x,0)\equiv0$ (in view of
Proposition \ref{P3.1.00}). And thus the necessary condition \eqref{5.5} is
not satisfied. To satisfy this condition we must have in our case
$f_{t}(x,0)\equiv0$. Consequently, the violation of condition \eqref{5.5} is a
blockage for further raising of the smoothness in $t$ in the closed domain
$\{t \geq0\}$. Thus for the further raising of the smoothness in $t$ in closed
domain we must impose additional restrictions on the derivatives $f_{t}%
^{(k)}(x,0)$. At the same time for problem \eqref{6.1}, \eqref{6.2} with an
integer derivative in $t$ additional restrictions of kind \eqref{5.5} are not necessary.

This example shows also to some extent that for additional smoothness in $x$
no additional restrictions are required.
\end{example}

In particular, Theorem \ref{T5.1} can be generalized as follows.

\begin{proposition}
\label{P8.1} Theorem \ref{T5.1} stays valid if condition \eqref{5.4} is
replaced by the more weak condition
\begin{equation}
\theta,\theta\alpha\in(0,1) \label{8.3}%
\end{equation}
that is the restrictions $\sigma_{k}\alpha\in(0,1)$,$\quad k=1,...,r$, can be omitted.
\end{proposition}

We do not present here a detailed proof of this proposition since the proof
verbatim coincides with the reasonings of section \ref{s6.1} at the removing
the analogous restrictions on $\alpha$ in the proof of Theorem \ref{T6.1}. The
presence in the equation either integer or fractional derivative in $t$ does
not matter in these reasonings, therefore we refer the reader to section
\ref{s6.1}.

Formulate now the main assertion of the present section.

\begin{theorem}
\label{T8.1} Let in problem \eqref{5.1}, \eqref{5.2} conditions \eqref{5.3.1},
\eqref{5.3}, and \eqref{5.5} be satisfied. Let, further, $\alpha$ be such that
numbers $\theta\alpha$, $\theta+\theta\alpha$, $\sigma_{k}\alpha$, $\sigma
_{k}+\sigma_{k}\alpha$, $k=1,...,r$ are nonintegers and for some positive
integer $n$
\begin{equation}
\theta\alpha\in(n-1,n). \label{8.3.1}%
\end{equation}
Let, finally, for $n\geq2$ besides the agreement condition \eqref{5.5} the
following condition be satisfied
\begin{equation}
\frac{\partial^{m}f(x,0)}{\partial t^{m}}\equiv0,\quad m=1,...,n-1.
\label{8.3.2}%
\end{equation}

Then problem \eqref{5.1}, \eqref{5.2} has the unique solution $u(x,t)\in
C^{\overline{\sigma}(1+\alpha),\theta+\theta\alpha}(\overline{R_{T}^{N}})$
with the estimates
\begin{equation}
|u|_{\overline{R_{T}^{N}}}^{(\overline{\sigma}(1+\alpha),\theta+\theta\alpha
)}\leq C(\overline{\sigma},\theta,\alpha,T)\left(  |f|_{\overline{R_{T}^{N}}%
}^{(\overline{\sigma}\alpha,\theta\alpha)}+|u_{0}|_{R^{N}}^{(\overline{\sigma
}(1+\alpha))}\right)  , \label{8.3.3}%
\end{equation}%
\begin{equation}
\left\langle u\right\rangle _{\overline{R_{T}^{N}}}^{(\overline{\sigma
}(1+\alpha),\theta+\theta\alpha)}\leq C(\overline{\sigma},\theta
,\alpha)\left(  |f|_{\overline{R_{T}^{N}}}^{(\overline{\sigma}\alpha
,\theta\alpha)}+|u_{0}|_{R^{N}}^{(\overline{\sigma}(1+\alpha))}\right)  ,
\label{8.3.4}%
\end{equation}%
\begin{equation}
\left\vert u\right\vert _{\overline{R_{\widetilde{T}}^{N}}}^{(0)}\leq
C(\overline{\sigma},\theta,\alpha)\left(  |f|_{\overline{R_{T}^{N}}%
}^{(\overline{\sigma}\alpha,\theta\alpha)}+|u_{0}|_{R^{N}}^{(\overline{\sigma
}(1+\alpha))}\right)  \widetilde{T}^{\theta+\theta\alpha}+|u_{0}|_{R^{N}%
}^{(0)},\quad\widetilde{T}\leq T. \label{8.3.5}%
\end{equation}

\end{theorem}

\begin{proof}
We will carry out the proof by induction in the value of $\theta\alpha$ by
formal differentiation of equation \eqref{5.1} in $t$ and reducing a problem
with more smooth data to a similar problem with less smooth data, according to
the schema of the previous section.

Thus, since for $\theta\alpha\in(0,1)$ the situation is described in Theorem
\ref{T5.1} and Proposition \ref{P8.1}, we suppose now that $\theta\alpha
\in(1,2)$. Without loss of generality we can assume zero initial data
$u_{0}(x)$ in problem \eqref{5.1}, \eqref{5.2}, as it was explained in section
\ref{s5.1}. Formally differentiating equation \eqref{5.1} in $t$, we obtain
the following equation for the function $v(x,t)\equiv u_{t}(x,t)$
\begin{equation}
Lv(x,t)\equiv D_{\ast t}^{\theta}v(x,t)+{\sum_{k=1}^{r}}(-\Delta_{z_{k}%
})^{\frac{\sigma_{k}}{2}}v(x,t)=f_{t}(x,t),\quad(x,t)\in R_{T}^{N}\label{8.4}%
\end{equation}
with the initial condition
\begin{equation}
v(x,0)\equiv0,\quad x\in R^{N}.\label{8.5}%
\end{equation}
Initial condition \eqref{8.5} is due to the fact that under the condition
$u_{0}(x)\equiv0$ and for $\theta+\theta\alpha>1$ (as it is in our case) we
have $u_{t}(x,0)\equiv0$ in view of \eqref{3.4.0}. Note also that
differentiating in $t$ the nonlocal derivative $D_{\ast t}^{\theta}u(x,t)$ in
equation \eqref{5.1}, due to the fact that $u_{t}(x,0)\equiv0$ and according
to the definition of the Caputo - Jrbashyan derivative, we have
\begin{equation}
\left(  D_{\ast t}^{\theta}u(x,t)\right)  _{t}=D_{\ast t}^{\theta}%
u_{t}(x,t).\label{8.5.1}%
\end{equation}

On the ground of \eqref{2.6} the right hand side in \eqref{8.4} belongs to the
space
\begin{equation}
f_{t}(x,t)\in C^{\overline{\sigma}\beta,\theta\beta}(\overline{R_{T}^{N}%
}),\quad\beta\equiv\frac{\theta\alpha-1}{\theta}, \label{8.6}%
\end{equation}
and
\begin{equation}
\left\vert f_{t}(x,t)\right\vert _{\overline{R_{T}^{N}}}^{(\overline{\sigma
}\beta,\theta\beta)}\leq C (\overline{\sigma},\theta,\alpha)\left\vert
f(x,t)\right\vert _{\overline{R_{T}^{N}}}^{(\overline{\sigma}\alpha
,\theta\alpha)}. \label{8.7}%
\end{equation}
As it was shown in Example \ref{E8.1}, for problem \eqref{8.4}, \eqref{8.5} to
be solvable in the class $C^{\overline{\sigma} (1+\beta),\theta+\theta\beta}$
condition \eqref{5.5} must be met, which in our case has the form
\begin{equation}
f_{t}(x,0)\equiv0,\quad x\in R^{N}. \label{8.8}%
\end{equation}
Since according to the assumptions of the the theorem condition \eqref{8.8} is
met and $\theta\beta=\theta\alpha-1\in(0,1)$, then based on Theorem \ref{T5.1}
and on Proposition \ref{P8.1} we infer that there exists a solution $v(x,t)$
to problem \eqref{8.4}, \eqref{8.5} from the class $v(x,t)\in C^{\overline
{\sigma}(1+\beta),\theta+\theta\beta}(\overline{R_{T}^{N}})$, and
\begin{equation}
\left\vert v(x,t)\right\vert _{\overline{R_{T}^{N}}}^{(\overline{\sigma
}(1+\beta),\theta+\theta\beta)}\leq C \left\vert f_{t}(x,t)\right\vert
_{\overline{R_{T}^{N}}}^{(\overline{\sigma}\beta,\theta\beta)}\leq C
(\overline{\sigma},\alpha,\theta,T ,N)\left\vert f(x,t)\right\vert
_{\overline{R_{T}^{N}}}^{(\overline{\sigma}\alpha,\theta\alpha)}. \label{8.9}%
\end{equation}

Consider the function
\begin{equation}
u(x,t)\equiv{\int\limits_{0}^{t}}v(x,\tau)d\tau.\label{8.10}%
\end{equation}
From \eqref{8.5.1} it follows that
\[
D_{\ast t}^{\theta}u(x,t)={\int\limits_{0}^{t}}D_{\ast\tau}^{\theta}%
v(x,\tau)d\tau.
\]
Now from the last relation and from \eqref{8.8} it follows that the function
$u(x,t)$ satisfies problem \eqref{5.1}, \eqref{5.2} with zero initial
condition. Besides, on the ground of \eqref{8.10} and \eqref{8.9} we infer
that $u(x,t)\in C^{\overline{\sigma}(1+\beta),\theta+\theta\beta}%
(\overline{R_{T}^{N}})$ and, consequently, it is the unique solution to
problem \eqref{5.1}, \eqref{5.2} from the pointed above or more smooth class.
At the same time, again from \eqref{8.10} and \eqref{8.9} it follows that
$u(x,t)$ has the smoothness in $t$ up to the order $\theta+\theta
\beta+1=\theta+\theta\alpha$ and
\begin{equation}
\left\langle u(x,t)\right\rangle _{t,\overline{R_{T}^{N}}}^{(\theta
+\theta\alpha)}\leq C (\overline{\sigma},\alpha,\theta,T
,N)\left\vert f(x,t)\right\vert _{\overline{R_{T}^{N}}}^{(\overline{\sigma
}\alpha,\theta\alpha)}.\label{8.11}%
\end{equation}
As for the smoothness in $x$ up to the orders $\sigma_{k}+\sigma_{k}\alpha$
(in each group of space variables $z_{k}$), it is proved verbatim to the
corresponding reasonings from section \ref{s6.1}, which proves Theorem
\ref{T8.1} for the case $\theta\alpha\in(0,2)$.

Finally, the general case $\theta\alpha\in(n-1,n)$ with an arbitrary positive
integer $n$ is obtained by induction in $n$, which completes the proof of the theorem.
\end{proof}

\section{Construction of functions from $C^{\overline{\sigma}(1+\alpha
),\theta+\theta\alpha}(\overline{R_{T}^{N}})$ from their initial data at
$t=0$.}

\label{s9}

In the present section we describe a way of constructing a function $w(x,t)$
from the class $C^{\overline{\sigma} (1+\alpha),\theta+\theta\alpha}%
(\overline{R_{T}^{N}})$ with $\theta+\theta\alpha>1$ according to it's initial
data at $t=0$. These initial data are the trace at $t=0$ of a function
$w(x,0)$ itself and the traces of it's derivatives in $t$ up to the order
$[\theta]$ regardless of whether $\theta$ is integer or fractional. The
constructing of such function is an important technical device in
investigations of parabolic initial value problems including \eqref{1.1},
\eqref{1.2}. This permits to reduce investigations of a problem to it's
investigations in the case of zero initial data and to consider the problem in
the classes of such functions that vanish at $t=0$ together with all their
possible derivatives in $t$. This approach is rather standard nowdays at
different considerations of parabolic problems and was applied, in particular,
in the classical paper \cite{Sol1}. In section \ref{s5} it was applied for the
extension of a solution to the domain $t<0$ with the class preservation, which
permitted to consider the problem in the whole space $R^{N+1}$. In section
\ref{s6} we did not have such opportunity yet and the extension to the domain
$t<0$ was performed without class preservation in $t$, which caused some more
complex structure of the proof. If now we are going to consider the problem of
higher order with integer or fractional $\theta>1$, then the extension of a
solution to the domain $t<0$ without class preservation with the aim of the
considerations in the whole $R^{N+1}$ would lead to emergence in the equation
of supported at $t=0$ delta-functions. Clearly, such way potentially could
make the investigations even more complicated.

Naturally, the mentioned function $w(x,t)$ is not unique and the way of it's
constructing, we present below, closely reproduces the construction from
\cite{LadParab}, Ch. IV, Theorem 4.3. The only difference consists of the
making use of a parabolic operator with some fractional Laplace operator
instead of the standard heat operator. However, the author have not managed to
produce an algorithm to construct such $w(x,t)$ for an arbitrary value of
$\alpha$. The problem with following \cite{LadParab} is that we must have some
results on the solvability and the estimates for the simplest Cauchy problems
with the initial data from the less smooth spaces than the order of the
corresponding differential operator. For the operator of the heat equation
(and much more general - see \cite{Sol1}) such results are known due to
estimates of the corresponding parabolic potentials, unlike the situation with
fractional operators. Therefore in the present section we first confine
ourselves to the case of somewhat heightened exponent $\alpha$. Namely, we
suppose that
\begin{equation}
\{\theta\}+\theta\alpha>1. \label{9.0}%
\end{equation}

Before we turn to the strict statements of this section, we describe, on the
ground of property \eqref{2.6} of anisotropic H\"{o}lder spaces, the
particular functional classes for the derivatives in $t$ up to the order
$[\theta]$ of a function $w(x,t)$ from the space $C^{\overline{\sigma}
(1+\alpha),\theta+\theta\alpha}(\overline{R_{T}^{N}})$. And also we formulate
some definitions.

Let $w(x,t)\in C^{\overline{\sigma} (1+\alpha),\theta+\theta\alpha}%
(\overline{R_{T}^{N}})$ with $\theta>1$ so that $[\theta]\geq1$, and
consequently the function $w(x,t)$ has the first derivative in $t$. According
to \eqref{2.6} the derivative $w_{t}(x,t)$ belongs to the space
\begin{equation}
w_{t}(x,t)\in C^{\overline{\sigma}(1+\alpha)\frac{\theta+\theta\alpha
-1}{\theta+\theta\alpha},\theta+\theta\alpha-1}(\overline{R_{T}^{N}%
})=C^{\overline{\sigma}(1+\alpha)-\frac{1}{\theta}\overline{\sigma}%
,\theta+\theta\alpha-1}(\overline{R_{T}^{N}}). \label{8.12}%
\end{equation}
That is after the differentiation in $t$ the smoothness in $t$ decreases by
one, and the smoothness in a group of the space variables $z_{k} $ decreases
by $\sigma_{k}/\theta$. Somewhat transforming the smoothness exponents of the
space in \eqref{8.12}, we formulate \eqref{8.12} as follows
\begin{equation}
w_{t}(x,t)\in C^{\frac{\overline{\sigma}}{\theta}(1+\alpha(\theta
,1)),1+\alpha(\theta,1)}(\overline{R_{T}^{N}}), \label{8.13}%
\end{equation}
where
\begin{equation}
\frac{\overline{\sigma}}{\theta}\equiv\{\frac{\sigma_{1}}{\theta}%
,...,\frac{\sigma_{r}}{\theta}\},\quad\alpha(\theta,1)\equiv\theta
-1+\theta\alpha-1>0, \label{8.14}%
\end{equation}
and we note that positivity of the exponent $\alpha(\theta,1)$ follows from
assumption \eqref{9.0}. Generally, if $[\theta]=n\geq1$, then proceeding with
the differentiating in $t$ as it is described in \eqref{8.12} - \eqref{8.14},
we get by induction
\begin{equation}
\frac{\partial^{i}w(x,t)}{\partial t^{i}}\in C^{\frac{\overline{\sigma}%
}{\theta}(1+\alpha(\theta,i)),1+\alpha(\theta,i)}(\overline{R_{T}^{N}%
})=C^{\overline{\sigma}(1+\alpha)-\frac{i}{\theta}\overline{\sigma}%
,\theta+\theta\alpha-i}(\overline{R_{T}^{N}}),\quad i=1,2,...n, \label{8.15}%
\end{equation}
where
\begin{equation}
\alpha(\theta,i)\equiv\theta-i+\theta\alpha-1>0 \label{8.16}%
\end{equation}
and
\begin{equation}
\left\vert \frac{\partial^{i}w(x,t)}{\partial t^{i}}\right\vert _{\overline
{R_{T}^{N}}}^{(\overline{\sigma}(1+\alpha)-\frac{i}{\theta}\overline{\sigma
},\theta+\theta\alpha-i)}\leq C\left\vert w(x,t)\right\vert _{\overline
{R_{T}^{N}}}^{(\overline{\sigma}(1+\alpha),\theta+\theta\alpha)}. \label{8.17}%
\end{equation}
Besides, applying relation \eqref{2.6}, we infer that the differentiating in
$t$ is a bounded linear operator from $C^{\overline{\sigma}(1+\alpha)-\frac
{i}{\theta}\overline{\sigma} ,\theta+\theta\alpha-i}(\overline{R_{T}^{N}})$
with $i\leq n-1$ to $C^{\overline{\sigma}(1+\alpha)-\frac{i+1}{\theta
}\overline{\sigma},\theta+\theta\alpha-i-1}(\overline{R_{T}^{N}})$ that is
\begin{equation}
\left\vert \frac{\partial^{i+1}w(x,t)}{\partial t^{i+1}}\right\vert
_{\overline{R_{T}^{N}}}^{(\overline{\sigma}(1+\alpha)-\frac{i+1}{\theta
}\overline{\sigma},\theta+\theta\alpha-i-1)}\leq C\left\vert \frac
{\partial^{i}w(x,t)}{\partial t^{i}}\right\vert _{\overline{R_{T}^{N}}%
}^{(\overline{\sigma}(1+\alpha)-\frac{i}{\theta}\overline{\sigma}%
,\theta+\theta\alpha-i)}. \label{8.18}%
\end{equation}

We note in addition the following. If the exponents $\theta\alpha$ and
$\sigma_{k}\alpha$, $k=1,...,r$, (the "additional" smoothness exponents) are
nonintegers, then for each $i=1,2,...n$ the "additional" smoothness exponents
$(\theta-i)\alpha(\theta,i)=\theta\alpha$ and $\sigma_{k}(\theta
,i)\alpha(\theta,i)=\sigma_{k}\alpha$ stay in fact the same nonintegers.

We need also the fact that in view of Theorem \ref{T3dop2} with $\overline
{\rho}=\overline{\sigma}/\theta$ the operators $(-\Delta_{z_{k}}%
)^{\frac{\sigma_{k}}{2\theta}}$, $k=1,...,r$, are, similar to the
differentiating in $t$, linear bounded operators from $C^{\overline{\sigma
}(1+\alpha),\theta+\theta\alpha} (\overline{R_{T}^{N}})$ to $C^{\overline
{\sigma}(1+\alpha)-\frac{1}{\theta}\overline{\sigma},\theta+\theta\alpha
-1}(\overline{R_{T}^{N}})$. That is for $w(x,t)\in C^{\overline{\sigma
}(1+\alpha),\theta+\theta\alpha}(\overline{R_{T}^{N}})$
\begin{equation}
\left\vert (-\Delta_{z_{k}})^{\frac{\sigma_{k}}{2\theta}}w(x,t)\right\vert
_{\overline{R_{T}^{N}}}^{(\overline{\sigma}(1+\alpha)-\frac{1}{\theta
}\overline{\sigma},\theta+\theta\alpha-1)}\leq C\left\vert w(x,t)\right\vert
_{\overline{R_{T}^{N}}}^{(\overline{\sigma}(1+\alpha),\theta+\theta\alpha)}.
\label{8.19}%
\end{equation}
Besides, based on Theorem \ref{T3dop2} we infer that similar to \eqref{8.17},
\eqref{8.18} we have for iterations of the operators $(-\Delta_{z_{k}}%
)^{\frac{\sigma_{k}}{2\theta}}$%
\begin{equation}
\left\vert \left[  (-\Delta_{z_{k}})^{\frac{\sigma_{k}}{2\theta}}\right]
^{i}w(x,t)\right\vert _{\overline{R_{T}^{N}}}^{(\overline{\sigma}%
(1+\alpha)-\frac{i}{\theta}\overline{\sigma},\theta+\theta\alpha-i)}\leq
C\left\vert w(x,t)\right\vert _{\overline{R_{T}^{N}}}^{(\overline{\sigma
}(1+\alpha),\theta+\theta\alpha)},i\leq n . \label{8.20}%
\end{equation}
And moreover for $i\leq n-1$
\begin{equation}
\left\vert (-\Delta_{z_{k}})^{\frac{\sigma_{k}}{2\theta}}\left\{  \left[
(-\Delta_{z_{k}})^{\frac{\sigma_{k}}{2\theta}}\right]  ^{i}w(x,t)\right\}
\right\vert _{\overline{R_{T}^{N}}}^{(\overline{\sigma}(1+\alpha)-\frac
{i+1}{\theta}\overline{\sigma},\theta+\theta\alpha-i-1)}\leq\label{8.21}%
\end{equation}%
\[
\leq C\left\vert \left[  (-\Delta_{z_{k}})^{\frac{\sigma_{k}}{2\theta}%
}\right]  ^{i}w(x,t)\right\vert _{\overline{R_{T}^{N}}}^{(\overline{\sigma
}(1+\alpha)-\frac{i}{\theta}\overline{\sigma},\theta+\theta\alpha-i)}
\]
that is $(-\Delta_{z_{k}})^{\frac{\sigma_{k}}{2\theta}}$ is a bounded linear
operator from $C^{\overline{\sigma}(1+\alpha)-\frac{i}{\theta}\overline
{\sigma},\theta+\theta\alpha-i}(\overline{R_{T}^{N}})$ with $i\leq n-1$ to
$C^{\overline{\sigma}(1+\alpha)-\frac{i+1}{\theta}\overline{\sigma}%
,\theta+\theta\alpha-i-1}(\overline{R_{T}^{N}})$.

Formulate now the main assertion of the present section.

\begin{theorem}
\label{T9.1} Let such positive integer or noninteger numbers $\theta$,
$\sigma_{k}$, $k=1,...,r$, and $\alpha$ be given that the numbers
$\theta\alpha$, $\theta+\theta\alpha$, $\sigma_{k}\alpha$, and $\sigma
_{k}+\sigma_{k}\alpha$ are nonintegers. Let, further, condition \eqref{9.0} is
met. Let, finally, such $n+1=[\theta]+1$ functions $\varphi_{i}(x)$,
$i=0,...,n$, be given that they are defined in $R^{N}$ and belong to the
spaces ($\overline{\sigma}\equiv\{\sigma_{1},...,\sigma_{r}\}$)
\begin{equation}
\varphi_{i}(x)\in C^{\overline{\sigma}(1+\alpha)-\frac{i}{\theta}%
\overline{\sigma}}(R^{N}),\quad i=0,...,n. \label{9.1}%
\end{equation}

Then there exists such functions $w(x,t)\in C^{\overline{\sigma}%
(1+\alpha),\theta+\theta\alpha}(\overline{R_{T}^{N}})$ that at $t=0$
\begin{equation}
\frac{\partial^{i}w(x,0)}{\partial t^{i}}=\varphi_{i}(x),\quad
i=0,...,n\label{9.2}%
\end{equation}
and
\begin{equation}
\left\vert w(x,t)\right\vert _{\overline{R_{T}^{N}}}^{(\overline{\sigma
}(1+\alpha),\theta+\theta\alpha)}\leq C{\sum\limits_{i=0}^{n}}\left\vert
\varphi_{i}(x)\right\vert _{R^{N}}^{(\overline{\sigma}(1+\alpha)-\frac
{i}{\theta}\overline{\sigma})}.\label{9.3}%
\end{equation}

\end{theorem}

\begin{proof}
The schema of proof for this theorem coincides with that from Theorem 4.3, Ch.
IV in \cite{LadParab}. However we present the proof here for completeness
because instead of the standard heat operator with well known properties,
which was used in \cite{LadParab}, we make use of the operator that was
investigated above in sections \ref{s6}, \ref{s7}.

Define the differential operator
\begin{equation}
L\equiv\frac{\partial}{\partial t}+M\equiv\frac{\partial}{\partial t}%
+{\sum_{k=1}^{r}}(-\Delta_{z_{k}})^{\frac{\sigma_{k}}{2\theta}}.\label{9.4}%
\end{equation}
From \eqref{8.15}, \eqref{8.16}, \eqref{8.20}, and \eqref{8.21} it follows
that $L$ is a linear bounded operator from the space $C^{\overline{\sigma
}(1+\alpha)-\frac{i}{\theta}\overline{\sigma},\theta+\theta\alpha-i}%
(\overline{R_{T}^{N}})$ with $0\leq i\leq n-1$ to the space $C^{\overline
{\sigma}(1+\alpha)-\frac{i+1}{\theta}\overline{\sigma},\theta+\theta
\alpha-i-1}(\overline{R_{T}^{N}})$,%
\begin{equation}
L:C^{\overline{\sigma}(1+\alpha)-\frac{i}{\theta}\overline{\sigma}%
,\theta+\theta\alpha-i}(\overline{R_{T}^{N}})\rightarrow C^{\overline{\sigma
}(1+\alpha)-\frac{i+1}{\theta}\overline{\sigma},\theta+\theta\alpha
-i-1}(\overline{R_{T}^{N}}),\quad0\leq i\leq n-1.\label{9.4.1}%
\end{equation}
Define, further, the functions
\begin{equation}
\psi_{j}(x)\equiv{\sum_{s=0}^{j}}C_{j}^{s}M^{s}\varphi_{j-s}(x),\quad
j=0,...,n.\label{9.5}%
\end{equation}

Let a function $w(x,t)\in C^{\overline{\sigma}(1+\alpha),\theta+\theta\alpha
}(R_{T}^{N})$ satisfies the condition
\begin{equation}
\left.  \left(  \frac{\partial}{\partial t}+M\right)  ^{j}w(x,t)\right\vert
_{t=0}=\psi_{j}(x),\quad j=0,...,n,\label{9.6}%
\end{equation}
where all relations are correctly defined in view of \eqref{9.4.1}. Show that
then this function satisfies also conditions\eqref{9.2}. Taking into account
the definition of the function $\psi_{j}(x)$ in \eqref{9.5}, condition
\eqref{9.6} can be formulated in the form
\begin{equation}
{\sum_{s=0}^{j}}C_{j}^{s}M^{s}\left[  \left.  \frac{\partial^{j-s}%
w(x,t)}{\partial t^{j-s}}\right\vert _{t=0}-\varphi_{j-s}(x)\right]  =0,\quad
j=0,...,n.\label{9.7}%
\end{equation}
From this condition we get by induction starting from $j=0$
\[
w(x,t)|_{t=0}=\varphi_{0}(x).
\]
Further, making use of this relation, we infer from \eqref{9.7} for $j=1$
that
\[
\left.  \frac{\partial w(x,t)}{\partial t}\right\vert _{t=0}=\varphi_{1}(x).
\]
Proceeding this process by induction we verify all the relations in \eqref{9.2}.

Define now the function $w(x,t)$ we need recursively from the Cauchy problem
\begin{equation}
\frac{\partial w}{\partial t}+Mw=w^{(1)}(x,t),\quad w(x,0)=\psi_{0}%
w(x)=\varphi_{0}(x), \label{9.8}%
\end{equation}
where the function $w^{(1)}(x,t)$ is defined in advance from the problem
\begin{equation}
\frac{\partial w^{(1)}}{\partial t}+Mw^{(1)}=w^{(2)}(x,t),\quad w^{(1)}%
(x,0)=\psi_{1}(x), \label{9.9}%
\end{equation}
and so on. And the initial function $w^{(n)}$ is defined from the problem
\begin{equation}
\frac{\partial w^{(n)}}{\partial t}+Mw^{(n)}=0,\quad w^{(n)}(x,0)=\psi_{n}(x).
\label{9.10}%
\end{equation}
On the base of Theorem \ref{T7.1} all functions $w^{(j)}(x,t)$ are correctly
defined, since all spaces $C^{\frac{\overline{\sigma} }{\theta}(1+\alpha
(\theta,i)),1+\alpha(\theta,i)}(\overline{R_{T}^{N}})=C^{\overline{\sigma
}(1+\alpha)-\frac{i}{\theta}\overline{\sigma},\theta+\theta\alpha-i}%
(\overline{R_{T}^{N}})$ with $\alpha(\theta,i)>0$ in the above Cauchy problems
are appropriate for the application of Theorem \ref{T7.1}.

Since, in view of the definition,
\[
w^{(j)}(x,t)=\left(  \frac{\partial}{\partial t}+M\right)  ^{j}w(x,t),
\]
then the function $w(x,t)$ satisfies conditions \eqref{9.6} and, consequently
\eqref{9.2} by the construction of $w^{(j)}(x,t)$.

Estimate \eqref{9.3} is obtained now by the successive application of Theorem
\ref{T7.1} to the chain of problems \eqref{9.8} - \eqref{9.10}.
\end{proof}

We stress one more that requirement \eqref{9.0} in this theorem is due to the
fact that the least smooth initial data in problem \eqref{9.10} belongs to the
the space $C^{\overline{\sigma}(1+\alpha)-\frac{n}{\theta}\overline{\sigma}%
,}(R^{N})=C^{\frac{\overline{\sigma}}{\theta}(1+\alpha(\theta,n))}(R^{N})$
with
\[
\alpha(\theta,n)=\theta-n+\theta\alpha-1=\theta-[\theta]+\theta\alpha
-1=\{\theta\}+\theta\alpha-1.
\]
Therefore $\alpha(\theta,n)$ is negative under violation of condition
\eqref{9.0} ($\alpha(\theta,n)$ can not be equal to zero since $\theta
-n+\theta\alpha$ is a noninteger by the assumption). At the same time the
total smoothness exponent for the space $\psi_{n}(x)\in C^{\overline{\sigma
}(1+\alpha)-\frac{n}{\theta}\overline{\sigma},}(R^{N})$ is positive and for
the case of integer derivatives in \cite{Sol1} by methods of parabolic
potentials sharp estimates for the solutions of the corresponding problems
\eqref{9.10} for parabolic systems were obtained. As for the fractional
problem under consideration, such estimates are unknown by now.

However, if $\theta\in(0,2)$ that is $n=[\theta]\leq1$, the the desired
function can be constructed without condition \eqref{9.0}. For $n=0$ the
construction is trivial and for $n=1$ such function can be constructed as the
solution of the problem
\begin{equation}
\frac{\partial w(x,t)}{\partial t}+Mw=\varphi_{1}(x)+M\varphi_{0}(x),(x,t)\in
R_{T}^{N};\quad w(x,0)=\varphi_{0}(x),x\in R^{N}. \label{9.11}%
\end{equation}
Here the initial data $\varphi_{0}(x)$ and the right hand side $\varphi
_{1}(x)+M\varphi_{0}(x)$ belong to the appropriate for the application of
Theorem \ref{T6.1} space and consequently problem \eqref{9.11} has the unique
solution from the desired class with the corresponding estimate of it's norm
over the norms of the functions $\varphi_{0}(x)$ and $\varphi_{1}(x)$. Thus
the following assertion is valid.

\begin{theorem}
\label{T9.2} For $\theta\in(0,2)$ Theorem \ref{T9.1} stays valid without
restriction \eqref{9.0}.
\end{theorem}

Note finally that if we consider to construct a function $w(x,t)$ from the
space $C^{\overline{\sigma}(1+\alpha),\theta+\theta\alpha}(\overline{R_{T}%
^{N}})$ not by the full set of the traces of it's integer derivatives in $t$
up to the order $n=[\theta]$ but by the set of the traces of it's integer
derivatives in $t$ up to the order $n-1=[\theta]-1$, then we can do without
\eqref{9.0}. Indeed, in this case the last (the initial) problem in the chain
of problems \eqref{9.8} - \eqref{9.10} is not the problem with the number $n$,
but the one with the number $n-1$
\begin{equation}
\frac{\partial w^{(n-1)}}{\partial t}+Mw^{(n-1)}=0,\quad w^{(n-1)}%
(x,0)=\psi_{n-1}(x), \label{9.12}%
\end{equation}
where
\[
\psi_{n-1}(x)\in C^{\overline{\sigma}(1+\alpha)-\frac{n-1}{\theta}%
\overline{\sigma},}(R^{N})=C^{\frac{\overline{\sigma}}{\theta}(1+\alpha
(\theta,n-1))}(R^{N}),
\]
and
\[
\alpha(\theta,n-1)=\theta-(n-1)+\theta\alpha-1=\theta-[\theta]+\theta
\alpha=\{\theta\}+\theta\alpha>0.
\]
Thus the simple replication of the proof of Theorem \ref{T9.1} leads to the
following (more weak in a sense) assertion without restriction \eqref{9.0}.

\begin{theorem}
\label{T9.3} Let such positive integer or noninteger numbers $\theta$,
$\sigma_{k}$, $k=1,...,r$, and $\alpha$ be given that the numbers
$\theta\alpha$, $\theta+\theta\alpha$, $\sigma_{k}\alpha$, and $\sigma
_{k}+\sigma_{k}\alpha$ are nonintegers. Let also such $n=[\theta]$ functions
$\varphi_{i}(x)$, $i=0,...,n-1$, be given that they are defined in $R^{N}$ and
belong to the spaces ($\overline{\sigma}\equiv\{\sigma_{1},...,\sigma_{r}\}$)
\begin{equation}
\varphi_{i}(x)\in C^{\overline{\sigma}(1+\alpha)-\frac{i}{\theta}%
\overline{\sigma}}(R^{N}),\quad i=0,...,n-1.\label{9.13}%
\end{equation}
Then there exists such a function $w(x,t)\in C^{\overline{\sigma}%
(1+\alpha),\theta+\theta\alpha}(\overline{R_{T}^{N}})$ that at $t=0$
\begin{equation}
\frac{\partial^{i}w(x,0)}{\partial t^{i}}=\varphi_{i}(x),\quad
i=0,...,n-1\label{9.14}%
\end{equation}
and
\begin{equation}
\left\vert w(x,t)\right\vert _{\overline{R_{T}^{N}}}^{(\overline{\sigma
}(1+\alpha),\theta+\theta\alpha)}\leq C{\sum\limits_{i=0}^{n-1}}\left\vert
\varphi_{i}(x)\right\vert _{R^{N}}^{(\overline{\sigma}(1+\alpha)-\frac
{i}{\theta}\overline{\sigma})}.\label{9.15}%
\end{equation}

\end{theorem}

\section{The proofs of theorems \ref{T2.1}, \ref{T2.2}, and \ref{T2.3}.}

\label{s10}

In the present section we outline the proof of theorems \ref{T2.1},
\ref{T2.2}, \ref{T2.3}. But we do not present here the detailed proofs since
they, in fact, would be simple verbatim copies of the reasoning and
constructions of the proofs for the main assertions from the previous
sections. In particular, Theorem \ref{T2.1} is a direct generalization to the
case $\theta=n\geq1$ of Theorem \ref{T7.1}, and Theorem \ref{T2.2}
generalizes, evidently, Theorem \ref{T8.1} to the case of an arbitrary
noninteger $\theta>0$. These theorems were not proved in their full generality
in the previous sections \ref{s5} and \ref{s6} since the preliminary
considerations of the cases $\theta\in(0,1)$ and $\theta=1$ were necessary for
the only reason to obtain the results of section \ref{s9} on the extension of
the initial data to the domain $\{t>0\}$. Without this extension it would not
be possible to reduce Cauchy problems of higher order in $t$ to the case of
zero initial data with the subsequent extension of solutions by zero to the
domain $\{t<0\}$. After such reducing all others steps in the corresponding
proofs stays unaltered.

As for Theorem \ref{T2.3}, it is a direct corollary of Proposition \ref{P3.6},
Proposition \ref{P3.8}, and Theorem \ref{T5.1}.

Therefore we comment only on the proofs of Theorems \ref{T2.1} and \ref{T2.2}.

\subsection{Proof of Theorem \ref{T2.1}.}

\label{s10.1}

For the case $n=1$ Theorem \ref{T2.1} coincides with Theorem \ref{T7.1}. If
$n>1$ the proof of Theorem \ref{T2.1} completely follows the steps and
reasonings of the proofs of Theorems \ref{T6.1}, \ref{T7.1}.

Firstly, the boundedness of operator $\emph{L}$ in the spaces of Theorem
\ref{T2.1} follows directly from property \eqref{2.6} and Proposition
\ref{P3.6}.

Further, to demonstrate the inverse bounded operator for $\emph{L}$, we,
similar to Theorem \ref{T6.1}, assume first that the smoothness exponent
$n\alpha\in(0,1)$ and it is sufficiently small.

On the first step we reduce the problem to zero initial data \eqref{1.2}. For
that we make use of Theorem \ref{T9.3}.

Taking advantage of the fact that the given initial data are equal to zero, we
on the second step extend the desired solution and the right hand side to the
domain $\{t<0\}\cup\{t>T\}$ to a finite in $t$ function and formulate the
original problem in the domain $R^{N}\times(-\infty,\infty)$. At that equation
\eqref{1.1} stays unchanged since the derivative in $t$ of order $n-1$ of the
solution is continuous at $t=0$ and consequently no terms of the kind of
supported at $\{t=0\}$ distributions emerge in the equation while calculating
the highest $t$-derivative of order $n$ - completely similar to the case
$n=1$. We also cut the right hand side off to a finite in $x$ function and
apply the smoothing to obtain the right hand side from $C^{\infty}(R^{N}%
\times(-\infty,\infty))$, exactly as it was done in sections \ref{s5.4} and
\ref{s6.2}.

On the next step we obtain solvability and estimates of the solution to the
problem with a finite right hand side of the class $C^{\infty}(R^{N}%
\times(-\infty,\infty))$. As it was in section \ref{s6.3}, we apply the
Fourier transform with respect to all variables and obtain the representation
for the highest derivatives of the solution
\begin{equation}
\frac{\widehat{\partial^{n}u_{m,\varepsilon}}}{\partial t^{n}}(\xi,\xi
_{0})=\frac{(i\xi_{0})^{n}}{(i\xi_{0})^{n}+{\sum_{k=1}^{r}}|\zeta_{k}%
|^{\sigma_{k}}}\widehat{f_{m,\varepsilon}}(\xi,\xi_{0})\equiv\widehat{m}%
_{0}(\xi,\xi_{0})\widehat{f_{m,\varepsilon}}(\xi,\xi_{0}),\label{10.1}%
\end{equation}%
\begin{equation}
\widehat{(-\Delta_{z_{i}})^{\frac{\sigma_{i}}{2}}u_{m,\varepsilon}}(\xi
,\xi_{0})=\frac{|\zeta_{i}|^{\sigma_{i}}}{(i\xi_{0})^{n}+{\sum_{k=1}^{r}%
}|\zeta_{k}|^{\sigma_{k}}}\widehat{f_{m,\varepsilon}}(\xi,\xi_{0}%
)\equiv\widehat{m}_{i}(\xi,\xi_{0})\widehat{f_{m,\varepsilon}}(\xi,\xi
_{0}).\label{10.2}%
\end{equation}
Note that the right hand side $f(x,t)$ in the extended to the domain $\{t<0\}$
problem is not continuous at $t=0$ in general since the derivative in $t$ of
order $n$ in the equation is not generally continuous at $t=0$.
Correspondingly, the smoothed right hand side $f_{m,\varepsilon}(x,t)$ in
\eqref{10.1} and in \eqref{10.2}, similar to section \ref{s6.3}, does not have
uniformly bounded in mollifying parameter $\varepsilon$ H\"{o}lder seminorm in
$t$. Consequently, similar to section \ref{s6.3}, we apply to the multipliers
$\widehat{m}_{0}(\xi,\xi_{0})$ and $\widehat{m}_{i}(\xi,\xi_{0})$ not only
Theorem \ref{T4.1} to obtain estimate \eqref{6.23}, but also Theorem
\ref{T4.2}, which permits to obtain uniform in $\varepsilon$ estimate of the
highest H\"{o}lder seminorms of the solution in the space variables $x$. At
that, since by assumption $n$ is not equal to a number of the form $4j+2$,
$j=0,1,...$, then $(i\xi_{0})^{n}$ is not equal to $-|\xi_{0}|^{n}$ for all
nonzero $\xi_{0}$. Consequently, the modulus of denominator in the expressions
for $\widehat{m}_{0}(\xi,\xi_{0})$ and $\widehat{m}_{i}(\xi,\xi_{0})$ is
separated from zero on the sets $B_{\nu}$ from \eqref{4.9}. Therefore, similar
to section \ref{s6.3}, it is not difficult to verify conditions of theorems
\ref{T4.1} and \ref{T4.2}. Application of these theorems leads, under definite
smallness of the exponent $\alpha$, to the estimate of the highest seminorms
of the solution $u_{m,\varepsilon}(x,t)$ for the problem with the smoothed
finite right hand side $f_{m,\varepsilon}(x,t)$. In particular, for the
highest derivative in $t$, based also on Lemma \ref{L6.1}, we obtain the
estimate in the whole space $R^{N+1}$
\begin{equation}
\left\langle \frac{\partial^{n}u_{m,\varepsilon}(x,t)}{\partial t^{n}%
}\right\rangle _{t,R^{N+1}}^{(n\alpha)}\leq C(\overline{\sigma})\left\langle
f_{m,\varepsilon}(x,t)\right\rangle _{R^{N+1}}^{(\overline{\sigma}%
\alpha,n\alpha)}\leq C(\overline{\sigma},\alpha)\varepsilon^{-n\alpha
}|f|_{\overline{R_{T}^{N}}}^{(\overline{\sigma}\alpha,n\alpha)},\label{15.551}%
\end{equation}
which is an analog of \eqref{6.23}. And similar to \eqref{6.30}, based also on
the equation, we obtain the uniform in $\varepsilon$ estimate in the domain
$\overline{R_{T}^{N}}$
\begin{equation}
\left\langle u_{m,\varepsilon}\right\rangle _{\overline{R_{T}^{N}}%
}^{(\overline{\sigma}+\overline{\sigma}\alpha,n+n\alpha)}\leq C(\alpha
,\overline{\sigma},\{N_{k}\})\left\vert f(x,t)\right\vert _{\overline
{R_{T}^{N}}}^{(\overline{\sigma}\alpha,n\alpha)},\quad m=1,2,...,\varepsilon
\in(0,1).\label{15.552}%
\end{equation}
To obtain an analog of \eqref{6.28} and \eqref{6.31} we first make use of
\eqref{15.551} to obtain
\[
\left\vert \frac{\partial^{n}u_{m,\varepsilon}(x,t)}{\partial t^{n}%
}\right\vert =\left\vert \frac{\partial^{n}u_{m,\varepsilon}(x,t)}{\partial
t^{n}}-\frac{\partial^{n}u_{m,\varepsilon}(x,-2\varepsilon)}{\partial t^{n}%
}\right\vert \leq
\]%
\begin{equation}
\leq\left\langle u_{m,\varepsilon}\right\rangle _{t,R^{N+1}}^{(n+n\alpha
)}|-2\varepsilon|^{n\alpha}\leq C(n,\overline{\sigma},\alpha)|f|_{\overline
{R_{T}^{N}}}^{(\overline{\sigma}\alpha,n\alpha)},\quad t\in\lbrack
-2\varepsilon,0].\label{15.553}%
\end{equation}
And then, analogously to \eqref{6.28},
\begin{equation}
\left\vert u_{m,\varepsilon}(x,0)\right\vert \leq C(n){\int
\limits_{-2\varepsilon}^{0}}\left\vert \frac{\partial^{n}u_{m,\varepsilon
}(x,t)}{\partial t^{n}}\right\vert t^{n-1}dt\leq C(n,\overline{\sigma}%
,\alpha)|f|_{\overline{R_{T}^{N}}}^{(\overline{\sigma}\alpha,\alpha
)}\varepsilon^{n}.\label{15.554}%
\end{equation}
Eventually, analogously to the obtaining \eqref{6.31}, we successively obtain
on $\overline{R_{T}^{N}}$
\[
\left\vert \frac{\partial^{n}u_{m,\varepsilon}(x,t)}{\partial t^{n}%
}\right\vert \leq\left\vert \frac{\partial^{n}u_{m,\varepsilon}(x,0)}{\partial
t^{n}}\right\vert _{R^{N}}^{(0)}+\left\langle \frac{\partial^{n}%
u_{m,\varepsilon}(x,t)}{\partial t^{n}}\right\rangle _{t,\overline{R_{T}^{N}}%
}^{(n\alpha)}t^{n\alpha}\leq
\]%
\[
\leq C(n,\overline{\sigma},\alpha)|f|_{\overline{R_{T}^{N}}}^{(\overline
{\sigma}\alpha,n\alpha)}\left(  1+T^{n\alpha}\right)
\]
and consequently
\[
\left\vert u_{m,\varepsilon}(x,t)\right\vert \leq\left\vert u_{m,\varepsilon
}(x,0)\right\vert +{C(n)\int\limits_{0}^{t}}\left\vert \frac{\partial
^{n}u_{m,\varepsilon}(x,\tau)}{\partial\tau^{n}}\right\vert (t-\tau
)^{n-1}d\tau\leq
\]%
\[
\leq C(n,\overline{\sigma},\alpha)|f|_{\overline{R_{T}^{N}}}^{(\overline
{\sigma}\alpha,n\alpha)}(\varepsilon^{n}+t^{n}\left(  1+T^{n\alpha}\right)  ),
\]
that is
\begin{equation}
\left\vert u_{m,\varepsilon}(x,t)\right\vert _{\overline{R_{T}^{N}}}^{(0)}\leq
C(n,\overline{\sigma},\alpha)|f|_{\overline{R_{T}^{N}}}^{(\overline{\sigma
}\alpha,n\alpha)}(1+T^{n+n\alpha}).\label{15.555}%
\end{equation}

Further, completely similar to section \ref{s6.3} and analogously to reasoning
of section \ref{s5.7} at the proving of Proposition \ref{P5.1}, based on
estimates \eqref{15.552} and \eqref{15.555}, we make the transition to the
limit on the parameters of smoothing and cutting-off under $\varepsilon
\rightarrow0$, $m\rightarrow\infty$ on the set $\overline{R_{T}^{N}}$ in the
sequence of the problems with smooth finite data. This results in a solution
$u(x,t)$ of the original problem with the estimates
\begin{equation}
\left\langle u\right\rangle _{\overline{R_{T}^{N}}}^{(\overline{\sigma
}(1+\alpha),n+n\alpha)}\leq C(\overline{\sigma},n,\alpha)\left(
|f|_{\overline{R_{\infty}^{N}}}^{(\overline{\sigma}\alpha,n\alpha)}%
+{\sum\limits_{i=0}^{n-1}}|u_{i}|_{R^{N}}^{(\overline{\sigma}(1+\alpha
)-\frac{i}{n}\overline{\sigma})}\right)  ,\label{10.3}%
\end{equation}%
\begin{equation}
\left\vert u\right\vert _{\overline{R_{\widetilde{T}}^{N}}}^{(0)}\leq
C(\overline{\sigma},\alpha,n)\left(  |f|_{\overline{R_{\infty}^{N}}%
}^{(\overline{\sigma}\alpha,n\alpha)}+{\sum\limits_{i=0}^{n-1}}|u_{i}|_{R^{N}%
}^{(\overline{\sigma}(1+\alpha)-\frac{i}{n}\overline{\sigma})}\right)
\times\label{10.4}%
\end{equation}%
\[
\times(1+\widetilde{T}^{n+\alpha})+|u_{0}|_{R^{N}}^{(0)},\quad\widetilde
{T}\leq T,
\]
where $T$ can be equal to $\infty$ if the original function $f(x,t)$ was
defined on $R^{N}\times\lbrack0,\infty)$ - analogously to Proposition
\ref{P6.1} in section \ref{s6.4}.

The uniqueness of the solution is obtained completely similar to section
\ref{s6.4} by the extension of the solution and applying the Fourier transform
in the space of distributions, which completes the proof of the solvability
and the estimates for problem \eqref{1.1}, \eqref{1.2} under a sufficiently
small $\alpha>0$.

On the last step, repeating verbatim first reasonings from section \ref{s6.5}
and then from section \ref{s7}, we prove the solvability and the estimates of
problem \eqref{1.1}, \eqref{1.2} for an arbitrary positive $\alpha$.

Thus, as a result, we obtain the existence and the boundedness of the inverse
operator $\emph{L} ^{\emph{-1}}$ to the operator $\emph{L}$, which completes
the proof of Theorem \ref{T2.1}.

\subsection{The proof of Theorem \ref{T2.2}.}

In the case $\theta\in(0,1)$ this theorem, evidently, coincides with Theorem
\ref{T8.1}. In the case of an arbitrary noninteger $\theta>0$ the proof step
by step follows the proofs of Theorems \ref{T5.1} and \ref{T8.1} with some
minor obvious changes. We will just highlight some of them.

Firstly, the reducing of the problem to zero initial data is made not just by
substraction of the initial function similar to section \ref{s5.1}, but by
substraction of a function $w(x,t)$ from the appropriate space, which
satisfies the initial conditions. Such function was constructed in Theorem
\ref{T9.1} for an arbitrary noninteger $\theta>0$ with the restriction
$\{\theta\}+\theta\alpha>1$ on the exponent $\alpha$, and also in Theorem
\ref{T9.2} for a noninteger $\theta\in(0,2)$ and for an arbitrary $\alpha>0$.

Further, after the extension of the unknown function by zero to the domain
$\{t<0\}$, in view of zero initial data, the derivatives in $t$ up to the
order $[\theta]$ stays continuous, and consequently, the same is valid for the
fractional derivative $D_{\ast}^{\theta}u(x,t)$ itself. Therefore, after
applying, similar to section \ref{s5.5}, the Fourier transform to obtain the
representation of the solution in terms of it's Fourier image we can use
Theorem \ref{T4.1}. At that the representations for the highest derivatives
have the form
\begin{equation}
\widehat{D_{\ast}^{\theta}u(x,t)u_{m,\varepsilon}}(\xi,\xi_{0})=\frac
{(i\xi_{0})^{\theta}}{(i\xi_{0})^{\theta}+{\sum_{k=1}^{r}}|\zeta_{k}%
|^{\sigma_{k}}}\widehat{f_{m,\varepsilon}}(\xi,\xi_{0})\equiv\widehat{m}%
_{0}(\xi,\xi_{0})\widehat{f_{m,\varepsilon}}(\xi,\xi_{0}),\label{10.5}%
\end{equation}%
\begin{equation}
\widehat{(-\Delta_{z_{i}})^{\frac{\sigma_{i}}{2}}u_{m,\varepsilon}}(\xi
,\xi_{0})=\frac{|\zeta_{i}|^{\sigma_{i}}}{(i\xi_{0})^{\theta}+{\sum_{k=1}^{r}%
}|\zeta_{k}|^{\sigma_{k}}}\widehat{f_{m,\varepsilon}}(\xi,\xi_{0}%
)\equiv\widehat{m}_{i}(\xi,\xi_{0})\widehat{f_{m,\varepsilon}}(\xi,\xi
_{0}).\label{10.6}%
\end{equation}
Since $\theta$ is a noninteger, then in the first term of the denominator
$(i\xi_{0})^{\theta}=(\pm i)^{\theta}|\xi_{0}|^{\theta}$ the numerical
coefficient $(\pm i)^{\theta}$ has a non-zero imaginary part. Therefore on the
annulus $B_{\nu}$ from \eqref{4.9} the denominator of the multipliers
$\widehat{m}_{0}(\xi,\xi_{0})$ and $\widehat{m}_{i}(\xi,\xi_{0})$ is strictly
separated from zero. This permits to verify the conditions of Theorem
\ref{T4.1} - completely similar to section \ref{s5.5}.

Highlight, finally, one more very simple alteration in the proof, caused by
the high order of the derivative in $t$. To obtain the corresponding analog of
Proposition \ref{P5.2} on the extension of the solution to the whole time
interval $t\in(0,\infty)$ one should use the Taylor polynomial of degree
$[\theta]$ on the role of the function $\widetilde{u}(x,t)$ from \eqref{5.80}
and, correspondingly, to use Lemma \ref{L5.1}, which is proved for an
arbitrary noninteger $\theta>0$.

Besides the pointed above simplest alterations all the others steps in the
proof of Theorem \ref{T2.2} coincide with the corresponding steps in the
proofs first of Theorem \ref{T5.1} and then of Theorem \ref{T8.1}, which leads
to the assertion of Theorem \ref{T2.2}.

\end{document}